\documentclass[11pt,leqn]{amsart}
\usepackage{amssymb}
\usepackage{mathrsfs}
\usepackage[all]{xy}
\usepackage{color}
\usepackage{soul}
\usepackage[latin1]{inputenc}
\usepackage{hyperref}
\usepackage{dsfont}
\usepackage{enumitem}

\makeatletter

\renewcommand{\tocsection}[3]{%
  \indentlabel{\@ifnotempty{#2}{\bfseries\ignorespaces#1 #2\quad}}\bfseries#3}
\renewcommand{\tocsubsection}[3]{%
  \indentlabel{\@ifnotempty{#2}{\ignorespaces#1 #2\quad}}#3}
\renewcommand{\tocsubsubsection}[3]{%
  \indentlabel{\@ifnotempty{#2}{\ignorespaces#1 #2\quad}}#3}
\newcommand\@dotsep{4.5}
\def\@tocline#1#2#3#4#5#6#7{\relax
  \ifnum #1>\c@tocdepth % then omit
  \else
    \par \addpenalty\@secpenalty\addvspace{#2}%
    \begingroup \hyphenpenalty\@M
    \@ifempty{#4}{%
      \@tempdima\csname r@tocindent\number#1\endcsname\relax
    }{%
      \@tempdima#4\relax
    }%
    \parindent\z@ \leftskip#3\relax \advance\leftskip\@tempdima\relax
    \rightskip\@pnumwidth plus1em \parfillskip-\@pnumwidth
    #5\leavevmode\hskip-\@tempdima{#6}\nobreak
    \leaders\hbox{$\m@th\mkern \@dotsep mu\hbox{.}\mkern \@dotsep mu$}\hfill
    \nobreak
    \hbox to\@pnumwidth{\@tocpagenum{\ifnum#1=1\bfseries\fi#7}}\par% <-- \bfseries for \section page
    \nobreak
    \endgroup
  \fi}
\AtBeginDocument{%
\expandafter\renewcommand\csname r@tocindent0\endcsname{0pt}
}
\def\l@subsection{\@tocline{2}{0pt}{2.5pc}{5pc}{}}
\def\l@subsubsection{\@tocline{3}{0pt}{4pc}{5pc}{}}
\makeatother

\usepackage{cancel}

\newcommand\N{{\mathbb N}}
\newcommand\R{{\mathbb R}}
\newcommand\T{{\mathbb T}}
\newcommand\C{{\mathbb C}}

\newcommand\Z{{\mathbb Z}}

\newcommand\Sp{{\mathbb S}}

\def\AA{{\mathcal A}}
\def\BB{{\mathcal B}}

\def\DD{{\mathcal D}}
\def\EE{{\mathcal E}}

\def\GG{{\mathcal G}}
\def\HH{{\mathcal H}}

\def\LL{{\mathcal L}}

\def\NN{{\mathcal N}}

\def\SS{{\mathcal S}}

\def\BBB{{\mathscr B}}

\def\R{\mathbb R}\def\C{\mathbb C}\def\N{\mathbb N}\def\Z{\mathbb Z}
  \def\S{\mathbb S} \def\T{\mathbb T}   
\def\D{\partial}\def\eps{\varepsilon}\def\phi{\varphi}
\def\d{{\rm d}}
\def\norm#1{\left\Vert#1\right\Vert}

\def\abs#1{\left\vert#1\right\vert}

\def\seq#1{\left<#1\right>}
\def\sep#1{\left(#1\right)}
\def\Re{{\operatorname {Re}\,}}

\def\wick{{\rm Wick}}

\newcommand{\dv}{{{\rm d}v}}

\def\exp{\textrm{e}}
\newcommand{\ddt}{{\frac{{{\rm d}}}{{{\rm d}t}}}}

\def\eps{{\varepsilon}}

\def\Nt{|\hskip-0.04cm|\hskip-0.04cm|}

\newtheorem{theo}{Theorem}
\newtheorem{prop}[theo]{Proposition}
\newtheorem{lem}[theo]{Lemma}
\newtheorem{cor}[theo]{Corollary}
\newtheorem{rem}[theo]{Remark}

\newtheorem{defin}[theo]{Definition}

\newcommand{\beqn}{\begin{equation}}
\newcommand{\eeqn}{\end{equation}}
\newcommand{\bal}{\begin{aligned}}
\newcommand{\eal}{\end{aligned}}
\newcommand{\bear}{\begin{eqnarray}}
\newcommand{\eear}{\end{eqnarray}}
\newcommand{\bean}{\begin{eqnarray*}}
\newcommand{\eean}{\end{eqnarray*}}

\newcommand{\e}{{\varepsilon}}

\newcommand{\la}{\langle}
\newcommand{\ra}{\rangle}

\title[Boltzmann equation without cut-off]{Regularization estimates and Cauchy theory for inhomogeneous Boltzmann equation for hard potentials without cut-off}

\begin{document}
\author{{\sc Fr\'{e}d\'{e}ric H\'{e}rau}}
\address{LMJL - UMR6629,
Universit\'{e} de
Nantes, CNRS, 2 rue de la Houssini\`ere, BP 92208, F-44322 Nantes cedex 3,
France. E-mail: {\tt frederic.herau@univ-nantes.fr}}

\author{{\sc Daniela Tonon}}
\address{CEREMADE, UMR 7534, Universit\'e Paris-Dauphine, PSL Research University,
Place du Mar\'echal de Lattre de Tassigny, 75775 Paris Cedex 16, France.
E-mail: {\tt tonon@ceremade.dauphine.fr }}

\author{{\sc Isabelle Tristani}}
\address{D\'epartement de math\'ematiques et applications, \'Ecole normale sup\'erieure, CNRS, PSL University, 75005 Paris, France
E-mail: {\tt isabelle.tristani@ens.fr}}

\date\today

\begin{abstract}
In this paper, we investigate the problems of Cauchy theory and exponential stability for the inhomogeneous Boltzmann equation without angular cut-off.
We only deal with the physical case of hard potentials type interactions (with a moderate angular singularity). We prove a result of existence and uniqueness of solutions in a close-to-equilibrium regime for this equation
in weighted Sobolev spaces with a polynomial weight, contrary to previous works on the subject, all developed with a weight prescribed by the equilibrium. It is the first result in this more physically relevant framework
for this equation. Moreover, we prove an exponential stability for such a solution, with a rate as close as we want to the optimal rate given by the semigroup decay of the linearized equation. Let us highlight the fact that a key point of the development of our Cauchy theory is the proof of new regularization estimates in short time for the linearized operator thanks to pseudo-differential tools.
\end{abstract}

\maketitle

%\vspace{1cm}
%\textbf{Mathematics Subject Classification (2010)}: 76P05 Rarefied gas flows, Boltzmann equation; %[See also 82B40, 82C40, 82D05];
%47H20 Semigroups of nonlinear operators; %[See also 37L05, 47J35, 54H15, 58D07].
%35B40 Asymptotic behavior of solutions.
%
%\vspace{0.3cm}
%\textbf{Keywords}: Boltzmann equation without cut-off; hard potentials; Cauchy theory; spectral gap; dissipativity; exponential rate of convergence; long-time asymptotic.

\vspace{0.5cm}
\tableofcontents

%%%%%%%%%%%%%%%%%%%%%%%%%%%%%%%%%%%%%%%%%%%%%%%%%%%%%%%%%%%%%%%%%%%%%%%%%%%%%%%%%%%%%%%%%%%%%%%%%%%%%%%%%%%%%%

\section{Introduction}
\label{sec:intro}
\setcounter{equation}{0}
\setcounter{theo}{0}

%%%%%%%%%%%%%%%%%%%%%%%%%%%%%%%%%%%%%%%%%%%%%%%%%%%%%%%%%%%%%%%%%%%%%%%%%%%%%%%%%%%%%%%%%%%%%%%%%%%%%%%%%%%%%

\subsection{The model}
In the present paper, we investigate the Cauchy theory and the asymptotic behavior of solutions to the spatially inhomogeneous
Boltzmann equation without angular cut-off, that is, for long-range interactions. Previous works have shown that there exist solutions in a close-to-equilibrium regime
but in spaces of type $H^q(e^{|v|^2/2})$ which are very restrictive. Here, we are interested in improving this result in the following sense:
we enlarge the space in which we develop a Cauchy theory in several ways, we do not require any assumption on the derivatives in velocity and more importantly, our weight is polynomial.
We thus only require a condition of finite moments on our data, which is more physically relevant.
Moreover, we jointly obtain a convergence to equilibrium for the solutions that we construct with an exponential and explicit rate.

We consider a system of particles described by its space inhomogeneous distribution density $f = f(t,x,v)$ with $t\in \R^+$ the time, $x \in \T^3$ the position and $v \in \R^3$ the velocity.
We hence study the so-called spatially inhomogeneous Boltzmann equation:
	\beqn \label{eq:Bol}
		\partial_t f + v \cdot \nabla_x f = Q(f,f).
	\eeqn
The Boltzmann collision operator is defined as
	$$
	Q(g,f):=\int_{\R^3 \times \Sp^2} B(v-v_*,\sigma) \left(g'_*f' - g_* f \right) \,\d \sigma \,\dv_*.
	$$
Here and below, we are using the shorthand notations $f=f(v)$, $g_*=g(v_*)$, $f'=f(v')$ and $g'_*=g(v'_*)$.
In this expression, $v$, $v_*$ and $v'$, $v'_*$ are the velocities of a pair of particles after and before collision.
We make a choice of parametrization of the set of solutions to the conservation of momentum and energy (physical law of elastic collisions):
$$
	v+v_*=v'+v'_*,
$$
$$
	|v|^2+|v_*|^2=|v'|^2+ |v'_*|^2,
$$
so that the pre-collisional velocities are given by:
$$
	v'=\frac{v+v_*}{2} + \frac{|v-v_*|}{2} \sigma, \quad v'_*=\frac{v+v_*}{2}- \frac{|v-v_*|}{2} \sigma, \quad \sigma \in \Sp^2.
$$
The Boltzmann collision kernel $B(v-v_*,\sigma)$ only depends on the relative velocity $|v-v_*|$ and on the deviation angle $\theta$ through $\cos \theta = \langle \kappa, \sigma \rangle$
where $\kappa = (v-v_*)/|v-v_*|$ and $\langle \cdot, \cdot \rangle$ is the usual scalar product in $\R^3$.
By a symmetry argument, one can always reduce to the case where $B(v-v_*, \sigma)$ is supported on $ \langle \kappa, \sigma \rangle \geq 0$ i.e. $0 \leq \theta \leq \pi/2$.
So, without loss of generality, we make this assumption.

In this paper, we shall be concerned with the case when the kernel $B$ satisfies the following conditions:
	\begin{itemize}[leftmargin=*]
	\item It takes product form in its arguments as
		\beqn \label{eq:product}
			B(v-v_*,\sigma) = \Phi(|v-v_*|) \, b(\cos \theta);
		\eeqn
	\item The angular function $b$ is locally smooth, and has a nonintegrable singularity for $\theta \rightarrow 0$:
	it satisfies for some $c_b>0$ and $s \in (0,1/2)$ (moderate angular singularity)
		\beqn \label{eq:angularsing}
			\forall \, \theta \in (0, \pi/2], \quad \frac{c_b}{\theta^{1+2s}} \leq \sin \theta \, b(\cos \theta) \leq \frac{1}{c_b \,  \theta^{1+2s}};
		\eeqn
	\item The kinetic factor $\Phi$ satisfies
		\beqn \label{eq:Phi}
			\Phi(|v-v_*|)= |v-v_*|^\gamma \quad \text{with} \quad \gamma \in (0,1),	
\eeqn
		this assumption could be relaxed to assuming only that $\Phi$ satisfies $\Phi(\cdot) = C_\Phi \, |\cdot|^\gamma$ for some $C_\Phi>0$.
	\end{itemize}
{ Note that, since we restrict ourselves to the case $s \in (0,1/2)$, the estimate \eqref{eq:angularsing} implies that $\int_{\Sp^2} \sin \theta \, b(\cos \theta) \,\d \sigma<\infty$, this will be used often in the following.}

Our main physical motivation comes from particles interacting according to a repulsive potential of the form
	\beqn \label{eq:potential}
	\phi(r)=r^{-(p-1)}, \quad p \in (2,+\infty).
	\eeqn

The assumptions made on $B$ throughout the paper include the case of potentials of the form~(\ref{eq:potential}) with $p>5$.
Indeed, for repulsive potentials of the form~(\ref{eq:potential}), the collision kernel cannot be computed explicitly but Maxwell~\cite{Max} has shown
that the collision kernel can be computed in terms of the interaction potential~$\phi$.
More precisely, it satisfies the previous conditions (\ref{eq:product}), (\ref{eq:angularsing}) and (\ref{eq:Phi}) in dimension~$3$ (see~\cite{BookCer,CIP,BookVill})
with $ s:=\frac{1}{p-1} \in (0,1)$ and $\gamma:= \frac{p-5}{p-1} \in (-3,1)$.

One traditionally calls \emph{hard potentials} the case $p>5$ (for which $0<\gamma<1$), \emph{Maxwell molecules} the case $p=5$ (for which $\gamma=0$)
and \emph{soft potentials} the case $2<p<5$ (for which $-3 < \gamma <0$).
We can hence deduce that our assumptions made on $B$ include the case of hard potentials.

\smallskip
Let us give a weak formulation of the collision operator $Q$. For any suitable test function~$\varphi=\varphi(v)$, we have:
\beqn \label{eq:weak}
\begin{aligned}
&\int_{\R^3} Q(f,f)(v) \, \varphi(v) \,\dv \\
&\quad = \frac{1}{4} \int_{\R^3 \times \R^3 \times \Sp^2} B(v-v_*,\sigma) \,
\left(f'_* f' - f_* f\right) \, \left(\varphi + \varphi_* - \varphi' - \varphi'_*\right)
\,\d \sigma \,\dv_* \,\dv.
\end{aligned}
	\eeqn
From this formula, we can deduce some features of equation~(\ref{eq:Bol}): It preserves mass, momentum and energy. Indeed, at least formally, we have:
	$$
	\int_{\R^3} Q(f,f)(v) \, \varphi(v) \,\dv = 0 \quad \text{for} \quad \varphi(v) = 1,  v,  |v|^2;
	$$
from which we deduce that a { solution $f$} to equation (\ref{eq:Bol}) is conservative, meaning that for any $t \ge 0$,
	\beqn \label{eq:conserv}
		\int_{\T^3 \times \R^3} f(t,x,v) \, \varphi(v) \,\dv \,\d x = \int_{\T^3 \times \R^3} f_0(x,v) \, \varphi(v) \,\dv \,\d x \quad \text{for} \quad \varphi(v) = 1,  v,  |v|^2.
	\eeqn

We introduce the entropy $H(f) = \int_{\T^3 \times \R^3} f \, \log(f) \,\dv \,\d x$ as well as the entropy production~$D(f)$ defined through:
	\beqn \label{eq:entroprod}
		\begin{aligned}
			D(f) &:= -\frac{\d}{\d t} H(f)\\
			&= \frac{1}{4} \int_{\T^3 \times \R^3 \times \R^3 \times \Sp^2} B(v-v_*, \sigma) \, (f' f'_* - f f_*) \log\frac{f' f'_*}{f f_*} \,\d \sigma \,\dv_* \,\dv \,\d x.
		\end{aligned}
	\eeqn
Boltzmann's $H$ theorem asserts that
	\beqn \label{eq:entroprod2}
		\frac{\d}{\d t} H(f) = -D(f) \leq 0
	\eeqn
and states that any equilibrium (i.e. any distribution which maximizes the entropy) is a Maxwellian distribution.
Moreover, it is known that global equilibria of~\eqref{eq:Bol} are global Maxwellian distributions that are independent of time $t$ and position $x$, { since we are working on the torus}.
In this paper, we shall only consider the case of an initial datum satisfying
	\beqn \label{eq:initialdatum}
		\int_{\T^3 \times \R^3} f_0 \,\dv \,\d x=1, \quad \int_{\T^3 \times \R^3} f_0 \, v \,\dv \,\d x = 0, \quad \int_{\T^3 \times \R^3} f_0 \,|v|^2\dv \,\d x=3,
	\eeqn
and therefore consider $\mu$ the Maxwellian with same mass, momentum and energy as $f_0$:
	\beqn \label{eq:mu}
	\mu(v):=(2 \pi)^{-3/2} e^{-|v|^2/2}.
	\eeqn

%----------------------------------%----------------------------------%----------------------------------%----------------------------------%----------------------------------%----------------------------------%----------------------------------%----------------------------------%
\medskip
\subsection{Notations and function spaces}
Let $X,Y$ be Banach spaces and consider a linear operator $\Lambda : X \to X$. When defined, we shall denote by $\SS_{\Lambda}(t) = e^{t\Lambda}$ the semigroup generated by $\Lambda$. Moreover we denote by $\BBB(X,Y)$ the space of bounded linear operators from $X$ to $Y$ and by $\| \cdot \|_{\BBB(X,Y)}$ its norm operator and we shall use the usual simplification~$\BBB(X) = \BBB(X,X)$.

\smallskip
For simplicity of notations, hereafter, we denote $\langle v \rangle=(1+|v|^{2})^{1/2}$; $a\approx b$ means that there exist constants $c_{1}, c_{2}>0$ depending only on fixed numbers such that $c_{1}b\leq a \leq c_{2}b$;
we shall use the same notation $C$ for positive constants that may change from line to line or abbreviate ``{ $\leq C$} " to ``{ $ \lesssim$ }",
where $C$ is a positive constant depending only on fixed number.

\smallskip
In what follows, we denote $m(v) := \la v \ra^k$ with $k\ge0$, the range of admissible $k$ will be specified throughout the paper. We also introduce $\chi \in \DD(\R)$ a truncation function which satisfies $\mathds{1}_{[-1,1]} \le \chi \le \mathds{1}_{[-2,2]}$ and we denote $\chi_{a} (\cdot) := \chi(\cdot/a)$ for $a>0$.

\smallskip
{ Throughout} the paper, we shall consider functions~$f=f(x,v)$ with~$x \in \T^3$ and $v \in \R^3$.
Let $\nu=\nu(v)$ be a positive Borel weight function and $1\leq p \leq \infty$.
We then define the space~$L^p_{x,v} (\nu) $ as the Lebesgue space associated to the norm, for $f=f(x,v)$,
	$$
	\bal
		\| f \|_{L^p_{x,v} (\nu)} &:= \big\| \|  f \|_{L^p_v(\nu)}   \big\|_{L^p_{x}}
		:= \big\| \| \nu \, f \|_{L^p_v}   \big\|_{L^p_{x}}  \\
	\eal
	$$
which writes if $p<\infty$:
	$$
	\bal
		\| f \|_{L^p_{x,v} (\nu)}  & = \left(\int_{\T^3_x}  \|  f (x, \cdot )\|^p_{L^p_v(\nu)} \,\d x \right)^{1/p} \\
		& = \left(\int_{\T^3_x}  \int_{\R^3_v} |f(x,v)|^p \, \nu(v)^p \,\d v  \,\d x\right)^{1/p} .
	\eal
	$$
We define the high-order Sobolev spaces $ H^{n}_{x} H^{\ell}_{v} (\nu)$, for $n,\ell \in \N$:

%{\color{blue} Shall we use $H^{n,\ell}_{x,v} (\nu)$ as notation?}
	\beqn \label{eq:HnL2}
	\| f \|^2_{ H^{n}_{x} H^{\ell}_{v} (\nu)} := \sum_{\substack{|\alpha| \leq \ell, \,  |\beta| \leq n\\ |\alpha| + |\beta| \leq \max(\ell,n)}}
	\| \partial^\alpha_v \partial^\beta_x (f \nu) \|^2_{L^2_{x,v}}.
	\eeqn
This definition reduces to the usual weighted Sobolev space $H^n_{x,v}(\nu)$ when $\ell=n$. {We use Fourier transform to define the general space $H^r_{x,v}(\nu)$ for $r \in \R^+$:
	\beqn \label{eq:Hnfrac}
	\|f\|^2_{H^r_{x,v}(\nu)} := \|f \nu\|^2_{H^r_{x,v}} = \sum_{\xi \in \Z^3} \int_{\R^3_\eta} (1+|\xi|^2+|\eta|^2)^r \,|\widehat{f\nu}(\xi,\eta)|^2 \, \d\eta
	\eeqn
 where the hat corresponds to the Fourier transform in both $x$ (with corresponding variable~$\xi \in \Z^3$) and $v$ (with corresponding variable $\eta \in \R^3$).
In this case, the norms given by~\eqref{eq:HnL2} and~\eqref{eq:Hnfrac} are equivalent. We won't make any difference in the notation and will use one norm or the other at our convenience. It won't have any impact on our estimates since it will only add multiplicative universal constants.}

{Let us remark that by classical results of interpolation (see for example~\cite{BookBL}), for every~$r \in \R^+$, one can write
$$
H^r_{x,v}(m) = \left[H^{\lfloor r \rfloor}_{x,v}(m), H^{\lfloor r \rfloor+1}_{x,v}(m)\right]_{r-\lfloor r \rfloor,2}.
$$
%where $\lfloor r \rfloor$ is the integer part of $r$.
The notation used above is the classical one of real interpolation. For sake of completeness, we briefly recall the meaning of this notation. For $C$ and $D$ two Banach spaces which are both embedded in the same topological separating vector space, for any { $z \in C+D$}, we define the $K$-function by
$$
K(t,z) := \inf_{z=c+d} \left(\|c\|_C+t\|d\|_D\right), \quad \forall \, t>0.
$$
We then give the definition of the space $[C,D]_{\theta,p}$ for $\theta \in (0,1)$ and $p \in [1,+\infty]$:
$$
[C,D]_{\theta,p} := \left\{z \in C+D, \, \, t \mapsto K(t,z)/t^\theta \in L^p\left(dt/t^{1/p}\right)\right\}.
$$}

We also introduce the fractional Sobolev space $H^{r,\varsigma}_{x,v}(\nu)$  for $r, \, \varsigma \in \R^+$
associated to the norm:
	\beqn \label{eq:frac}
	\|f\|^2_{H^{r,\varsigma}_{x,v}(\nu)} := \|f  \nu\|^2_{H^{r,\varsigma}_{x,v}} = \sum_{\xi \in \Z^3} \int_{\R^3_\eta} (1+|\xi|^2)^r \, (1+|\eta|^2)^{\varsigma} \,|\widehat{f\nu}(\xi,\eta)|^2 \,\d\eta.
	\eeqn
When $r \in \N$, we can also define the space $H^{r,\varsigma}_{x,v}(\nu)$ through the norm:
	\beqn \label{eq:Hn(L2)}
		\bal
		&\|f\|^2_{H^{r,\varsigma}_{x,v}(\nu)} := \sum_{0 \le j \le r} \int_{\T^3_x} \|\nabla^j_x f\|^2_{H^\varsigma_v(\nu)} \d x=  \sum_{0 \le j \le r} \|\nabla^j_x f\|^2_{L^2_xH^\varsigma_v(\nu)}.
		\eal
	\eeqn
As previously, when $r \in \N$, the norms given by~\eqref{eq:frac} and~\eqref{eq:Hn(L2)} are equivalent and we will use one norm or the other at our convenience.
Finally, denoting for $\varsigma \in \R^+$,
	$$
	\|f\|^2_{\dot{H}^\varsigma_v(\nu)} := \|f  \nu\|^2_{\dot{H}^\varsigma_v}  = \int_{\R^3_\eta} {|\eta|^{2\varsigma}} \, |\widehat{f\nu}(\eta)|^2 \,\d\eta,
	$$
we introduce the space $\dot H^{n,\varsigma}_{x,v}(\nu)$ for $(n,\varsigma) \in \N \times \R^+$ defined through the norm:
	\beqn \label{eq:hom}
		\|f\|^2_{\dot H^{n,\varsigma}_{x,v}(\nu)} := \sum_{0 \le j \le n} \int_{\T^3_x} \|\nabla^j_x f\|^2_{\dot H^\varsigma_v(\nu)} \,\d x=  \sum_{0 \le j \le n} \|\nabla^j_x f\|^2_{L^2_x\dot H^\varsigma_v(\nu)}.
	\eeqn
Notice also that in the case $\varsigma=0$, the spaces $H^n_x L^2_v (\nu)$ and $H^{n,0}_{x,v} (\nu)$ associated respectively to the norms given by~\eqref{eq:HnL2} and~\eqref{eq:Hn(L2)} are the same.

We now introduce some ``twisted'' Sobolev spaces (useful for the development of our Cauchy theory in Section~\ref{sec:nonlinear}), we denote them $\HH^{n,\varsigma}_{x,v}(\nu)$ for $(n,\varsigma) \in \N \times \R^+$ and they are associated to the norm:
	\beqn \label{eq:HHnxL2v}
		\bal
			\| f \|_{\HH^{n,\varsigma}_{x,v}(\nu)}^2
			&:= \sum_{0 \le j \le n} \int_{\T^3_x} \|\nabla^j_x f\|^2_{H^\varsigma_v(\la v \ra^{-2js}\nu)} \, \d x=  \sum_{0 \le j \le n} \|\nabla^j_x f\|^2_{L^2_xH^\varsigma_v(\la v \ra^{-2js}\nu)}
		\eal
	\eeqn
where $s$ is the angular singularity of the Boltzmann kernel introduced in~\eqref{eq:angularsing}. For the case $\varsigma=0$, since the notation is consistent, we will use the notation~$\HH^n_x L^2_v(\nu)$ or $\HH^{n,0}_{x,v}(\nu)$ indifferently.

{ Moreover, we introduce the spaces $H^n_{x}\HH^\ell_{v}(\nu)$ and $\HH^n_x \HH^\ell_v(\nu)$, $(n,\ell) \in \N^2$ which are respectively associated to the following norms:
	\beqn \label{eq:norm}
		\|f\|^2_{H^n_{x}\HH^\ell_{v}(\nu)} := \sum_{|\alpha|\le \ell, \, \, |\beta|
\le n, \, \, |\alpha| + |\beta| \le \max(\ell,n)} \| \partial^\alpha_v \partial^\beta_x
 f\|^2_{L^2_{x,v}(\nu {\langle v \rangle^{- 2 |\alpha| s}})},
\eeqn
and
\beqn \label{eq:norm2}
		\|f\|^2_{{\HH^n_{x}\HH^\ell_{v}}(\nu)} := \sum_{|\alpha|\le \ell, \, \, |\beta| \le n, \, \, |\alpha| + |\beta| \le \max(\ell,n)} \| \partial^\alpha_v \partial^\beta_x  f\|^2_{L^2_{x,v}(\nu {\langle v \rangle^{- 2 |\alpha| s - 2 |\beta| s}})}.
	\eeqn
Note that those spaces are only needed to state our main result on the linearized problem (see~\eqref{def:EE2} and Theorem~\ref{theo:main2}). 	
}

{Finally, following works from Alexandre et al. (see~\cite{AMUXY-CMP}), we introduce an anisotropic norm that we denote $\|\cdot\|_{\dot H^{s,*}_v}$ (the notation will be explained by Lemma~\ref{lem:anis}) and which is defined through
\begin{equation} \label{eq:Hs*}
\|f\|_{\dot H^{s,*}_{v}}^2:= \int_{\R^3 \times \R^3 \times \Sp^2} b_\delta(\cos \theta) \mu_* \langle v_* \rangle^{-\gamma} (f'\langle v' \rangle^{\gamma/2}-f \langle v \rangle^{\gamma/2})^2 \,\d \sigma \,\dv_* \,\dv.
\end{equation}
In this definition, $\gamma$ is the power of the kinetic factor in~\eqref{eq:Phi} and $\mu$ is given by~\eqref{eq:mu}.
Moreover, we recall that $b$ is the angular function of the Boltzmann kernel which satisfies~\eqref{eq:angularsing} and we define $b_\delta$ as the following truncation of $b$:
$b_\delta(\cos \theta) := \chi_\delta(\theta) b(\cos \theta)$ with~$\delta$ fixed so that the conclusion of Lemma~\ref{lem:coercive} holds. Since the constant $\delta$ is fixed, we do not mention the dependency of the norm defined above with respect to $\delta$.
Let us also introduce the space $H^{s,*}_v(\nu)$ associated with the norm
\beqn \label{eq:anis0}
\|f\|^2_{H^{s,*}_v(\nu)} := \|f\|^2_{L^2_v(\langle v \rangle^{\gamma/2}\nu)} + \|f \nu\|^2_{\dot H^{s,*}_{v}}.
\eeqn
For $n \in \N$, we also define the space $\HH^{n,s,*}_{x,v} (\nu)$ associated with the norm
\beqn \label{eq:anis}
\|f\|^2_{\HH^{n,s,*}_{x,v} (\nu)} :=  \sum_{0 \le j \le n} \int_{\T^3_x}\|\nabla^j_x f \|_{H^{s,*}_v(\langle v \rangle^{-2js} \nu)}^2 \,\d x
\eeqn
where $s$ is still the angular singularity in~\eqref{eq:angularsing}. }

%----------------------------------%----------------------------------%----------------------------------%----------------------------------%----------------------------------%----------------------------------%----------------------------------%----------------------------------%
\medskip
In what follows, we shall state our main results as well as some known results on the subject.
\subsection{Cauchy theory and convergence to equilibrium}

We state now the main result on the fully nonlinear problem~\eqref{eq:Bol}. {Let $m(v)= \langle v \rangle^k$ with
{$$
k>{21 \over 2} + \gamma+22s.
$$}We then denote $\mathbf X := \HH^3_x L^2_v(m)$ and we introduce $\mathbf Y^*:=\HH^{3,s,*}_{x,v}(m)$
(see~\eqref{eq:HHnxL2v} and~\eqref{eq:anis} for the definition of the spaces).
 }

\begin{theo}\label{main1}
	%Consider $m(v) = \la v \ra^k$ with {\Blue $k>k_*$}.
	We assume that $f_0$ has same mass, momentum and energy as $\mu$ (i.e. satisfies~\eqref{eq:initialdatum}).
	There is a constant $\e_0 >0$ such that if $\| f_0 - \mu \|_{\mathbf X} \le \e_0$,
	then there exists a unique global weak solution $f$ to the Boltzmann equation~\eqref{eq:Bol}, which satisfies, for some constant $C>0$,
		$$
		\| f - \mu\|_{L^\infty ([0,\infty); \mathbf X)} + \| f - \mu \|_{L^2([0,\infty); \mathbf Y^*)} \le C \e_0.
		$$
	Moreover, this solution satisfies the following estimate: For any $0<\lambda_2 < \lambda_1$ there exists~$C>0$ such that
		$$
		\forall\, t \ge 0, \quad
		\| f(t) - \mu  \|_{\mathbf X} \le C \, e^{- \lambda_2 t} \, \| f_0 -\mu \|_{\mathbf X},
		$$
	where $\lambda_1 >0$ is the optimal rate given by the semigroup decay of the associated linearized operator in Theorem~\ref{theo:extension}.
\end{theo}

We refer to Remark~\ref{rem:weights} in which the imposed condition on the power $k$ of our weight is explained.

\smallskip
Let us now comment our result and give an overview on the previous works on the Cauchy theory for the inhomogeneous Boltzmann equation.
For general large data, we refer to the paper of DiPerna-Lions~\cite{DPL-AM} for global existence of the so-called renormalized solutions in the case of the Boltzmann equation with cut-off. This notion of solution has been extended to the case of long-range interactions by Alexandre-Villani~\cite{AV-CPAM} where they construct global renormalized solutions with a defect measure. We also mention the work of Desvillettes-Villani~\cite{DV-IM} that proves the convergence to equilibrium of a priori smooth solutions for both Boltzmann and Landau equations for large initial data. { Let us point out the fact that a consequence of our result combined with the one of Desvillettes and Villani is a proof of the exponential H-theorem:  We can show  exponential decay in time of solutions to the fully nonlinear Boltzmann equation, conditionally  to  some  regularity and moment  bounds (the assumption on the exponential lower bound can be removed thanks to the work of Mouhot~\cite{Mou-CPDE1}). As noticed in~\cite{GMM*} for example in Theorem 5.19, the result of Desvillettes and Villani which is expressed in terms of relative entropy can be translated into stronger norms. This fact allows to do the link between their result and ours. }

In a close-to-equilibrium framework, Gressman and Strain~\cite{GS-JAMS} in parallel with Alexandre et al.~\cite{AMUXY-CMP} have developed a Cauchy theory in spaces of type $H^n_xH^\ell_v(\mu^{-1/2})$. One of the famous difficulty of the Boltzmann equation without cut-off is to well understand coercivity estimates. In both papers~\cite{AMUXY-CMP} and~\cite{GS-JAMS}, the gain induced is seen and understood through a non-isotropic norm. {Our strategy uses this type of approach but we also exploit the fact that the linearized Boltzmann operator can be seen as a pseudo-differential operator in order to understand the gain of regularity induced by the linearized operator.} It allows us to obtain regularization estimates (quantified in time) on the semigroup associated to the linearized operator (see Theorem~\ref{thm:LIBthm}).
To end this brief review, we also refer to a series of papers by Alexandre et al.~\cite{AMUXY-ARMA1,AMUXY-AAS,AMUXY-ARMA2,AMUXY-CMP,AMUXY-JFA} in which the Boltzmann equation without cut-off is studied in various aspects (different type of collision kernels, Cauchy theory in exponentially weighted spaces, regularity of the solutions etc...).

\smallskip
Let us underline the fact that Theorem~\ref{main1} largely improves previous results on the Cauchy theory associated to the Boltzmann equation without cut-off for hard potentials in a perturbative setting. Indeed, we have enlarged the space in which the Cauchy theory has been developed in the sense that the weight of our space is much less restrictive (it is polynomial instead of the inverse Maxwellian equilibrium) and we also require few assumptions on the derivatives, in particular no derivatives in the velocity variable. However, we need three derivatives in the space variable (Gressman and Strain only require two derivatives in $x$ in~\cite{GS-JAMS}): This is the counterpart of the gain in weight we have obtained. Indeed, our framework is less favorable  and needs more attention due to the lack of symmetry of the operator in our spaces to obtain nonlinear estimates on the Boltzmann collision operator. And thus, to close our estimates, we require regularity on three derivatives in~$x$. { Let us also mention that it would be interesting to obtain our results in a space of type~$L^2_vL^\infty_x(m)$. However, even at the linear level, we are not able to get satisfying estimates. More precisely, dissipativity and regularization estimates seem unreachable at the moment in spaces of type $L^2_vL^p_x(m)$ with $p \neq 2$. We mention anyway that in a series of recent works \cite{IMS19}, \cite{49}, \cite{38},\cite{39}, the authors have developed a general approach concerning estimates in $L^1_vL^\infty_x(m)$ spaces, which are naturally associated to the standard macroscopic quantities (mass, energy and entropy). In \cite{IMS19} in particular, some results were obtained about $L^\infty$ control of the solution, assuming only a priori positive bounds from below and above of these macroscopic quantities on  a given interval $[0,T]$. Some regularization estimates were also proven in this context in \cite{38} (see also \cite{IM19} for a toy model). Neither existence nor decay for large time are at the center of these works, but they surely will provide tools for a deeper understanding and advances in the study of the inhomogeneous Boltzmann without cutoff equation.  }

%----------------------------------%----------------------------------%----------------------------------%----------------------------------%----------------------------------%----------------------------------%----------------------------------%----------------------------------%
\medskip
\subsection{Strategy of the proof}
 Our strategy is based on the study of the linearized equation. And then, we go back to the fully nonlinear problem. This is a standard method to develop a Cauchy theory in a close-to-equilibrium regime. However, we point out that both studies of the linear and the nonlinear problems are very tricky.

Usually, for example in the case of the non-homogeneous Boltzmann equation for hard spheres in~\cite{GMM*}, the gain induced by the linear part of the equation is quite easy to understand and directly controls the loss due to the nonlinear part of the equation so that the linear part is dominant and thus dictates the dynamics of the equation. In our case, it is more difficult because the gain induced by the linear part is at first sight not strong enough to control the nonlinear loss and it is not possible to conclude using only rough estimates on the Boltzmann collision operator (this fact was for example pointed out by Mouhot and Neumann in~\cite{MN-Nonlinearity}).
 As a consequence:
 \begin{itemize}[leftmargin=*]
\item We establish some new very accurate nonlinear estimates on the Boltzmann collision operator (see Lemma~\ref{lem:nonlinearnonhom}) (notice that in the spirit of what was done in~\cite{CTW-ARMA} by Carrapatoso, Wu and the third author, we work in Sobolev spaces in which the weights depend on the order of the derivative in the space variable).

\item We analyze precisely the gain induced by the linear part of the equation in both $x$ and~$v$ variables. It is crucial for two reasons: First, to get the large time behavior of the semigroup associated to the linearized operator in our large Banach space in which we want to develop our Cauchy theory (Theorem~\ref{theo:main2}); Secondly, to be sure that the linear gain exactly compensates the nonlinear loss identified in Lemma~\ref{lem:nonlinearnonhom}.
This analysis is based on two different points of view: The one already adopted by Alexandre et al. in~\cite{AMUXY-ARMA2} using the anisotropic norm defined in~\eqref{eq:Hs*} (we use it in our dissipativity estimates in Lemma~\ref{lem:coercive} and in our nonlinear estimates in Lemma~\ref{lem:nonlinearhom}); But also a new one which is detailed in the next paragraph and consists in new short time regularization estimates for the linearized operator (we use it in Section~\ref{sec:nonlinear} to conclude the proof of Theorems~\ref{theo:main2} and~\ref{main1}).
 \end{itemize}
 Those key elements allow us to close our estimates and thus, to develop our Cauchy theory in our ``twisted'' Sobolev spaces.

 {Concerning the above second point, notice that one could probably improve our analysis in the sense that we do not clearly make the link between the regularization properties studied in Section~\ref{sec:reg} and the gain of regularity provided by the norm~\eqref{eq:Hs*}. Doing the link between those two type of estimates would require to be more accurate in Section~\ref{sec:reg}. Indeed, in the latter section, we authorize ourselves not to be optimal in our estimates in terms of weights because we have some leeway in the use of Theorem~\ref{thm:LIBthm} that we make in Subsections~\ref{subsec:reg} and~\ref{subsec:reg2}. Conversely, we have to get sharp estimates on the gain of regularity in the coercivity estimates because it has to match exactly the loss of regularity and weights coming from the nonlinear part of the equation (see Subsection~\ref{subsec:conclusion}).  }

%----------------------------------%----------------------------------%----------------------------------%----------------------------------%----------------------------------%----------------------------------%----------------------------------%----------------------------------%
\medskip
\subsection{Regularization properties}

In this paragraph, we state our main result about the short time regularization properties of the linearized Boltzmann operator. A key point is that the linearized operator is seen as a pseudo-differential operator, following the framework introduced in~\cite{AHL*} by Alexandre, Li and the first author.

The linearized operator around equilibrium is defined at first order through
	$$
	\Lambda h := Q(\mu,h) + Q(h,\mu) - v \cdot \nabla_x h
	$$
	and we denote $\SS_\Lambda(t)$ the semigroup associated with $\Lambda$. In the following statement, we denote $(H^{r,s}_{x,v}(\langle v \rangle^{k}))'$ (resp. $(H^{r+s,0}_{x,v}(\langle v \rangle^{k}))'$) the dual space of $H^{r,s}_{x,v}(\langle v \rangle^{k})$ (resp. $H^{r+s,0}_{x,v}(\langle v \rangle^{k})$) with respect to $H^{r,0}_{x,v}(\langle v \rangle^{k})$.
	Here is our main regularization result (the condition on the weights in this result are made in order to be sure that our operator $\Lambda$ generates a semigroup in the spaces that we consider - see the conditions in Theorem~\ref{theo:main2}).

\begin{theo} \label{thm:LIBthm}
Let { $r \in \N$,  $k' \ge 0$, $k >\max(\gamma/2 + 3+2(\max(1,r)+1)s, k'+\gamma+5/2)$}. Consider~{$h_0 \in H^{r,0}_{x,v}(\langle v \rangle^k)$, resp. $h_0 \in (H^{r,s}_{x,v}
(\langle v \rangle^{k})'$}. Then, there exists $C_r>0$ independent of~$h_0$ such that for any $t \in (0,1]$,
$$
 \|\SS_\Lambda(t)h_0\|_{H^{r,s}_{x,v}(\langle v \rangle^{k'})}
 \le \frac{C_r}{t^{1/2}} \|h_0\|_{H^{r,0}_{x,v}(\langle v \rangle^{k})},
$$
respectively
$$
\|\SS_\Lambda(t)h_0\|_{H^{r,0}_{x,v}(\langle v \rangle^{k'})} \leq \frac{C_r}{t^{1/2}} \|h_0\|_{(H^{r,s}_{x,v}(\langle v \rangle^{k}))'}.
$$
Consider $h_0 \in H^{r,0}_{x,v}(\langle v \rangle^k)$, resp. $h_0 \in (H^{r+s,0}_{x,v}(\langle v \rangle^{k}))'$. Then, there exists $C'_r>0$ independent of $h_0$ such that for any $t \in (0,1]$,
$$
\|\SS_\Lambda(t)h_0\|_{H^{r+s,0}_{x,v}(\langle v \rangle^{k'})}
\le \frac{C'_r}{t^{1/2+s}} \|h_0\|_{H^{r,0}_{x,v}(\langle v \rangle^{k})},
$$
respectively
$$
 \|\SS_\Lambda(t)h_0\|_{H^{r,0}_{x,v}(\langle v \rangle^{k'})} \leq \frac{C'_r}{t^{1/2+s}} \|h_0\|_{(H^{r+s,0}_{x,v}(\langle v \rangle^{k}))'}.
 $$

\end{theo}

First, we have to underline that it is the first result of regularization quantified in time on the Boltzmann equation without cutoff.
It is well-known that the singularity of the Boltzmann kernel in the non cutoff case implies that the Boltzmann operator without cutoff (that we will describe later on) roughly behaves as a fractional Laplacian in velocity:
$$
Q(g,h) \approx - C_g (-\Delta_v)^s h + \, \text{lower order terms}
$$
with $C_g$ depending only on the physical properties of $g$.
 This type of result has already been studied in the homogeneous and non-homogeneous cases. As mentioned above, the gain in velocity is quite obvious to observe even if it is complicated to understand it precisely: Up to now, the most common way to understand it is through an anisotropic norm (see~\cite{GS-JAMS} by Gressman and Strain and~\cite{AMUXY-ARMA2} by Alexandre et al.). It is then natural to expect that the transport term allows to transfer the gain in velocity to the space variable. We refer to the references quoted in \cite{AHL*} for a review of this type of hypoelliptic properties. Let us mention that the paper~\cite{AHL*} by Alexandre et al. is the first one in which the hypoellipticity features of the operator have been deeply analyzed.

Our strategy here is to use the same method as for Kolmogorov type equations introduced in~\cite{Herau-JFA} by the first author. In short, except from the fact that the use of pseudodifferential tools is required and thus there are many additional technical difficulties, the spirit of the method is the same as for the fractional Kolmogorov equation in~{\cite{HTT2*}}.
For purposes of comparison, we can also mention that this type of strategy has also been applied successfully to the Landau equation in \cite{CTW-ARMA} by Carrapatoso et al..
However, the study of this kind of properties is much harder in the case of the Boltzmann equation without cutoff since the gain in regularity is less clear and consists in an anisotropic gain of fractional derivatives: We have to exploit the fact that one can write a part the Boltzmann linearized operator as a pseudo-differential operator, in the spirit of what has been done in~\cite{AHL*}.

 Indeed, we adapt here some ideas from there allowing to do computations for operators - including the Boltzmann one - whose symbols are in an adapted class called here $S_K$, where  $K$ is a large parameter. Let us point out that those classes are complicated partly because the order of the symbols does not decrease with derivation, which induces some great technical difficulties.
The computations are done using the Wick quantization, widely studied in particular by Lerner (see \cite{LernerBook1} and \cite{LernerBook2}), which has very nice positivity properties. This allows to adapt  to the Boltzmann case the Lyapunov strategy already introduced in~\cite{Herau-JFA} for the Kolmogorov case and in~{\cite{HTT2*}} for the fractional Kolmogorov one.

 It is also important to underline the fact that this pseudo-differential study is not done on the whole linearized operator but only on a well-chosen part of it (this is the object of Subsection~\ref{subsec:thm:LIBthm}). Indeed, thanks to Duhamel formula, we will then be able to recover an estimate on the whole semigroup, the one associated to $\Lambda$ (see Lemma~\ref{lem:lambda_1/lambda}).

 {Even if we do not investigate this problem in this paper, let us finally mention that we believe that the solution that we construct in Theorem~\ref{main1} immediately becomes smooth. Indeed, we think that the regularization estimates on the linearized operator performed thanks to the Lyapunov functional introduced in Paragraph~\ref{subsec:propduallambda1} could be propagated to the whole nonlinear equation: The additional nonlinear terms would be treated using our nonlinear estimates and the fact that our solutions are close to the equilibrium. This may be the aim of a future work. }

%----------------------------------%----------------------------------%----------------------------------%----------------------------------%----------------------------------%----------------------------------%----------------------------------%----------------------------------%
\medskip
\subsection{Exponential decay of the linearized semigroup}
We study spectral properties of the linearized operator~$\Lambda$ in various weighted Sobolev spaces of type $H^n_xH^\ell_v(\langle v \rangle^k)$ up to~$L^2_{x,v}(\langle v \rangle^k)$ for $k$ large enough. It will provide us the large time behavior of the semigroup in all those spaces and in particular in the one in which we want to develop our Cauchy theory.
It is important to highlight the fact that, in order to take advantage of symmetry properties,
most of the previous studies have been made in Sobolev weighted spaces of type~$H^q_{x,v}(\mu^{-1/2})$. We largely improve theses previous results in the sense that we are able to get similar spectral estimates in
larger Sobolev spaces, with a polynomial weight and with less assumptions on the derivatives.

{ To be more precise, we establish exponential decay of the semigroup $\SS_{\Lambda}(t)$ in various Lebesgue and Sobolev spaces that we will denote $\EE$:
	\begin{equation} \label{def:EE2}
	\begin{gathered}
		\EE :=
		\left\{ \bal
			%&L^2_{x,v}(m) \qquad \qquad \qquad \qquad \qquad \quad  \text{with} \quad {k > \gamma/2 + 3+2s} \\
			&H^n_x \HH^\ell_v(\langle v \rangle^k), \, \, (n, \ell) \in \N^2, \; n \ge \ell \\
			&\HH^n_x \HH^\ell_v(\langle v \rangle^k),  \, \,  (n, \ell) \in \N^2, \; n \ge \ell 		\eal \right.
			\quad \text{with}\quad
	{ k>{\gamma\over 2}+3+2( \max(1,n)+1)s .}
	\end{gathered}
	\end{equation}
Notice that those definitions include the case $L^2_{x,v}(\langle v \rangle^k)$ which can be obtained in one or the other type of space taking $n=\ell=0$. See \eqref{eq:norm}, \eqref{eq:norm2} for the definition of the spaces above.
}

Here is a rough version of the main result (Theorem~\ref{theo:extension}) that we obtain on the linearized operator $\Lambda$:
\begin{theo} \label{theo:main2}
	Let $\EE$ be one of the admissible spaces defined { above.} %in~\eqref{def:EE}.
	 Then, there exist explicit constants $\lambda_1 >0$ and $C\ge1$ such that
		$$
		\forall \, t \ge 0,  \quad \forall\, h_0 \in \EE, \quad
		\| \SS_{\Lambda}(t) h_0 - \Pi_0 h_0 \|_{\EE} \le C \, e^{-\lambda_1 t} \, \| h_0 - \Pi_0 h_0 \|_{\EE},
		$$
	where $\Pi_0$ the projector onto the null space of~$\Lambda$ defined by~(\ref{eq:Pi0}).
\end{theo}
	
As mentioned above, the non homogeneous linearized operator $\Lambda$ (and its homogeneous version $\LL h := Q(\mu,h) + Q(h,\mu)$) has already been widely studied. Let us first briefly review the existing results concerning spectral gap estimates for the homogeneous case. Pao~\cite{Pao} studied spectral properties of the linearized operator $\LL$ for hard potentials by non-constructive and very technical means. This article was reviewed by Klaus~\cite{Klaus}. Then, Baranger and Mouhot gave the first explicit estimate on this spectral gap in~\cite{BM-RMI} for hard potentials ($\gamma>0$). If we denote $\DD$ the Dirichlet form associated to $-\LL$:
$$
\DD(h):=\int_{\R^3}
(-\LL h) \, h \, \mu^{-1} \, \d v,
$$
and $\NN(\LL)^\perp$ the orthogonal of the null space of $\LL$, $\NN(\LL)$ which is given by
$$
\NN(\LL) = \mathrm {Span} \{ \mu , v_1 \mu, v_2 \mu, v_3 \mu , |v|^2 \mu   \},
$$
the Dirichlet form $\DD$ satisfies
\beqn \label{eq:entropyL2}
\forall \, h \in \NN(\LL)^\perp, \quad \DD(h) \geq \lambda_0 \, \|h\|^2_{L^2(\mu^{-1/2})},
\eeqn
for some constructive constant $\lambda_0>0$. This result was then improved by Mouhot~\cite{Mou-CPDE2} and later by Mouhot and Strain~\cite{MS-JMPA}. In the last paper, it was conjectured that a spectral gap exists if and only if $\gamma+2s \geq 0$. This conjecture was finally proven by Gressman and Strain in~\cite{GS-JAMS}. Finally, let us point out that the analysis that we carry on can be seen as the sequel of the one handled in~\cite{Trist-JSP} by the third author which focuses on the homogeneous linearized operator $\LL$.
We improve it in several aspects: We are able to deal with the spatial dependency and we are able to do computations in $L^2$ (only the $L^1$-case was treated in the latter).

Concerning the non-homogeneous case, we  state here a result coming from Mouhot and Neumann~\cite{MN-Nonlinearity} (which takes advantage of the results proven in~\cite{BM-RMI} by Baranger and Mouhot),
it gives us a spectral gap estimate in $H^q_{x,v}(\mu^{-1/2})$, $q \in \N^*$, thanks to hypocoercivity methods.
Let us underline the fact that it provides us the existence of spectral gap and an estimate on the semigroup decay associated to $\Lambda$ in the ``small'' space $E = H^q_{x,v}(\mu^{-1/2})$,
which is a crucial point in view of applying the enlargement theorem of~\cite{GMM*}. It is also important to precise that Mouhot and Neumann~\cite{MN-Nonlinearity} only { obtained} a result on the linearized operator,
they { were} not able to go back to the nonlinear problem.

\begin{theo}[\cite{MN-Nonlinearity}] \label{theo:gapE}
Consider $E:=H^{q}_{x,v}(\mu^{-1/2})$ with $q \in \N^*$. Then, there exists a constructive constant $\lambda_0>0$ (spectral gap) such that $\Lambda$ satisfies on $E$:
	\begin{enumerate}[leftmargin=1.1cm]
		\item[(i)] The spectrum $\Sigma(\Lambda) \subset \left\{ z \in \C : \Re  \, z \leq -\lambda_0 \right\} \cup \left\{0\right\}$;
		\item[(ii)] The null space $N(\Lambda)$ is given by
			\beqn \label{eq:kerlambda}
			N(\Lambda) = \mathrm {Span} \{ \mu , v_1 \mu, v_2 \mu, v_3 \mu , |v|^2 \mu   \},
			\eeqn
		and the projection $\Pi_0$ onto $N(\Lambda)$ by
			\beqn\label{eq:Pi0}
			\bal
			\Pi_0 h &= \left(\int_{\T^3 \times \R^3} h \,\dv \,\d x \right) \mu + \sum_{i=1}^3 \left(\int_{\T^3 \times \R^3} v_i h  \,\dv \,\d x \right) v_i \mu \\
			&\quad
			 \left(\int_{\T^3 \times \R^3} \frac{|v|^2-3}{6}\, h  \,\dv \,\d x\right) \, \frac{(|v|^2-3)}{6} \, \mu;
			\eal
			\eeqn
		\item[(iii)] $\Lambda$ is the generator of a strongly continuous semigroup $\SS_\Lambda(t)$ that satisfies
			\beqn \label{eq:spectralgapE}
			\forall \, t \ge 0,  \, \forall\, h_0 \in E, \quad
			\| \SS_{\Lambda}(t) h_0 - \Pi_0 h_0 \|_{E} \le  e^{-\lambda_0 t} \| h_0 - \Pi_0 h_0 \|_{E}.
			\eeqn
	\end{enumerate}
\end{theo}
To prove Theorem \ref{theo:main2}, our strategy follows the one initiated by Mouhot in~\cite{Mou-CMP} for the homogeneous Boltzmann equation for hard potentials with cut-off.
This argument has then been developed and extended in an abstract setting by Gualdani, Mischler and Mouhot~\cite{GMM*}, and Mischler and Mouhot~\cite{MM-ARMA}.
Let us describe in more details this strategy. We want to apply the abstract theorem of enlargement of the space of semigroup decay from \cite{GMM*,MM-ARMA} to our linearized operator $\Lambda$. We shall deduce the spectral/semigroup estimates of Theorem~\ref{theo:main2} on ``large spaces'' $\EE$ using the already known spectral gap estimates for $\Lambda$ on $H^q_{x,v}(\mu^{-1/2})$, for $q \ge 1$, described in Theorem~\ref{theo:gapE}. Roughly speaking,
to do that, we have to find a splitting of $\Lambda$ into two operators $\Lambda = \AA + \BB$ which satisfy some properties. The first part $\AA$ has to be bounded, the second one $\BB$ has to have some dissipativity properties (see Subsection~\ref{subsec:dissip}), and also the operator $(\AA \SS_\BB(t))$ is required to have some regularization properties (which will be satisfied thanks to Theorem~\ref{thm:LIBthm} in our case). Note that, compared to the work by the third author~\cite{Trist-JSP}, a new splitting of the linearized operator is exhibited and both the dissipativity and regularity estimates are completely new.

%----------------------------------%----------------------------------%----------------------------------%----------------------------------%----------------------------------%----------------------------------%----------------------------------%----------------------------------%
\medskip
\subsection{Outline of the paper}
We end this introduction by describing the organization of the paper. In Section~\ref{sec:prelim}, we prove various estimates on the Boltzmann collision operator. Section~\ref{sec:reg} is dedicated to the proof of Theorem~\ref{thm:LIBthm} (note that the pseudodifferential study is confined to Subsection~\ref{subsec:quantization}). In Section~\ref{sec:lin}, we study the linearized equation and develop our dissipativity estimates before proving Theorem~\ref{theo:main2}. Finally, in Section~\ref{sec:nonlinear}, we end the proof of our main result Theorem~\ref{main1}.

\medskip
\noindent\textbf{Acknowledgments.} %The third author has been partially supported by the {\em Fondation Math\'{e}matique Jacques Hadamard}.
This research has been supported by the \'Ecole Normale Sup\'erieure through the project Actions Incitatives {\it Analyse de solutions d'\'equations de la th\'eorie cin\'etique des gaz}. The first author thanks the Centre Henri Lebesgue ANR-11-LABX-0020-01 for its support and the third author thanks the ANR EFI:  ANR-17-CE40-0030. The authors thank St\'{e}phane Mischler and Kleber Carrapatoso for fruitful discussions and the anonymous referees for their careful reading and valuable comments.  

\bigskip
%%%%%%%%%%%%%%%%%%%%%%%%%%%%%%%%%%%%%%%%%%%%%%%%%%%%%%%%%%%%%%%%%%%%%%%%%%%%%%%%%%%%%%%%%%%%%%%%%%%%%%%%%%%%%%%%%
\section{Preliminaries on the Boltzmann collision operator}
\label{sec:prelim}
\setcounter{equation}{0}
\setcounter{theo}{0}
%%%%%%%%%%%%%%%%%%%%%%%%%%%%%%%%%%%%%%%%%%%%%%%%%%%%%%%%%%%%%%%%%%%%%%%%%%%%%%%%%%%%%%%%%%%%%%%%%%%%%%%%%%%%%%%%%

In this part, we give estimates on the trilinear form $\langle Q(g,h), f \rangle$ in our physical framework (meaning that the collision kernel $B$ satisfies conditions~\eqref{eq:product}, \eqref{eq:angularsing}, \eqref{eq:Phi}).
We start by recalling some homogeneous estimates and then establish some new estimates in weighted Sobolev (or Lebesgue) non homogeneous spaces.
These estimates will be used in the linear (Section~\ref{sec:lin}) and nonlinear (Section~\ref{sec:nonlinear}) studies. At the end of this section, we also give some estimates that will be useful in the study of regularization properties of the linearized operator (see Section~\ref{sec:reg}).

For sake of clarity, we recall that $m(v) = \langle v \rangle^k$ with $k\ge 0$ and that we will specify the range of admissible $k$ in each result. %----------------------------------%----------------------------------%----------------------------------%----------------------------------%----------------------------------%----------------------------------%----------------------------------%----------------------------------%
\medskip
{ \subsection{Bound on the anisotropic norm}
In this subsection, we compare the anisotropic norm defined in~\eqref{eq:anis0} with usual Sobolev norms.
\begin{lem} \label{lem:anis}
Let $k \ge 0$. We have the following estimate: For $g \in H^s_v(\langle v \rangle^{\gamma/2+s}m)$,
$$
\delta^{2-2s}\|g\|_{H^s_v(\langle v \rangle^{\gamma/2}m)} \lesssim \|g\|_{H^{s,*}_v(m)} \lesssim \|g\|_{H^s_v(\langle v \rangle^{\gamma/2+s}m)}.
$$
\end{lem}
\begin{proof}
Adapting the proof of~\cite[Theorem~3.1]{He-JSP}, we know that there exist $c_0$ and $c_1$ such that
$$
 \|gm\|^2_{\dot H^{s,*}_v} \ge c_0 \, \delta^{2-2s} \|g\|^2_{H^s_v(\langle v \rangle^{\gamma/2}m)} - c_1 \, \delta^{2-2s} \|g\|^2_{L^2(\langle v \rangle^{\gamma/2}m)}.
$$
As a consequence, we have for $\lambda \in (0,1)$,
\begin{align*}
\|g\|^2_{H^{s,*}_v(m)} &= \|g\|^2_{L^2_v(\langle v \rangle^{\gamma/2}m)} +  \|gm\|^2_{\dot H^{s,*}_v} \\
&\ge  \|g\|^2_{L^2_v(\langle v \rangle^{\gamma/2}m)} + \lambda \|gm\|^2_{\dot H^{s,*}_v} \\
&\ge  \|g\|^2_{L^2_v(\langle v \rangle^{\gamma/2}m)} (1-\lambda \, c_1 \, \delta^{2-2s}) + \lambda \, c_0 \, \delta^{2-2s} \|g\|^2_{H^s_v(\langle v \rangle^{\gamma/2}m)}.
\end{align*}
Taking $\lambda>0$ small enough, we obtain the bound $\delta^{2-2s} \|g\|_{H^s_v(\langle v \rangle^{\gamma/2}m)} \lesssim \|g\|_{H^{s,*}_v(m)}$. The reverse bound is directly given by~\cite[Lemma~2.4]{AMUXY-CMP} since
\begin{align*}
&\int_{\R^3 \times \R^3 \times \Sp^2} b_\delta(\cos \theta) \mu_* \la v_* \ra^{-\gamma} (g'm'\la v'\ra^{\gamma/2} - gm\la v \ra^{\gamma/2})^2 \,\d \sigma \,\dv_* \,\dv\\
&\qquad \le \int_{\R^3 \times \R^3 \times \Sp^2} b(\cos \theta) \mu_* \la v_* \ra^{-\gamma} (g'm'\la v'\ra^{\gamma/2} - gm\la v \ra^{\gamma/2})^2\,\d \sigma \,\dv_* \,\dv.
\end{align*}
\end{proof}}
We will use the fact that our lower bound in the previous lemma depends on $\delta$ in the proof of Lemmas~\ref{lem:coercive} and~\ref{lem:BdissipL2}. However, in the next subsection, $\delta$ is fixed so that the conclusion of Lemma~\ref{lem:coercive} is satisfied, we thus do not mention anymore the dependency of constants with respect to $\delta$.

%----------------------------------%----------------------------------%----------------------------------%----------------------------------%----------------------------------%----------------------------------%----------------------------------%----------------------------------%
\medskip
\subsection{Homogeneous estimates} \label{subsec:hom}
\begin{lem}[\cite{He*}] \label{lem:CH}
For smooth functions $f$, $g$, $h$, one has:
$$
		|\langle Q(f,g), h \rangle_{L^2_v} | \lesssim \|f\|_{L^1_v(\langle v \rangle^{\gamma + 2s})} \, \| g\|_{H^{\varsigma_1}_v(\langle v \rangle^{N_1})} \, 	\|h\|_{H^{\varsigma_2}_v(\langle v \rangle^{N_2})}
$$
	with $\varsigma_1$, $\varsigma_2 \in [0,2s]$ satisfying $\varsigma_1+\varsigma_2=2s$ and $N_1$, $N_2 \ge 0$ such that $N_1+N_2=\gamma+2s$.
\end{lem}

%\begin{lem}[{\cite[Proposition~3.6]{AMUXY-Kyoto}}]  \label{lem:AMUXY}
%	For smooth functions $g$, $h$, $f$, one has:
%		$$
%		|\langle Q(f,g), h \rangle_{L^2_v}| \lesssim \|f\|_{L^1_v(\langle v \rangle^{\gamma+2s})} \, \|g\|_{H^{2s}_v(\langle v \rangle^{\gamma+2s})} \, \|h\|_{L^2_v}.
%		$$
%	In particular,
%		\beqn \label{eq:QL2}
%			\| Q(f,g) \|_{L^2_v} \lesssim \|f\|_{L^1_v(\langle v \rangle^{\gamma+2s})} \, \|g\|_{H^{2s}_v(\langle v \rangle^{\gamma+2s})}.
%		\eeqn
%\end{lem}

The goal of what follows is to extend this type of estimates to polynomial weighted Lebesgue spaces: Lemma~\ref{lem:nonlinearhom} is a ``weighted version'' of Lemma~\ref{lem:CH}.

\begin{lem} \label{lem:nonlinearhom}
Assume { $k>\gamma/2+2+2s$}.
\begin{enumerate}[leftmargin=*]
		\item[(i)] For any {$\ell>\gamma+1+3/2$}, there holds
			\beqn \label{eq:Q(f,g),h}
			\bal
				&|\la Q(f,g),h \ra_{L^2_v(m)}|
				\lesssim \|f\|_{L^2_v(\la v \ra^\ell)} \, \|g\|_{H^{\varsigma_1}_v(\la v \ra^{N_1} m)} \, \|h\|_{H^{\varsigma_2}_v(\la v \ra^{N_2} m)} \\
				&\qquad + \|f\|_{L^2_v (\la v \ra^{\gamma/2}m)} \, \|g\|_{L^2_v(\la v \ra^\ell)} \, \|h\|_{L^2_v (\la v \ra^{\gamma/2}m)}\\
			\eal
			\eeqn
			with $\varsigma_1$, $\varsigma_2 \in [0,2s]$ and { $N_1 \ge \gamma/2$}, $N_2 \ge 0$ satisfying respectively $\varsigma_1+\varsigma_2=2s$ and~$N_1+N_2=\gamma+2s$.
		\item[(ii)] For any { $\ell>4-\gamma+3/2$}, there holds
			{\beqn \label{eq:Q(f,g),g}
				|\la Q(f,g),g \ra_{L^2_v(m)}| \lesssim \|f\|_{L^2_v(\la v \ra^\ell)} \, \|g\|^2_{H^{s,*}_v(m)} + \|f\|_{L^2_v (\la v \ra^{\gamma/2}m)} \, \|g\|_{L^2_v(\la v \ra^\ell)} \, \|g\|_{L^2_v (\la v \ra^{\gamma/2}m)}.
			\eeqn}
%		\item[(iii)] For any {\Blue $\ell>4-\gamma+3/2$}, there holds
%			{ $$
%			\la Q(f,f),f\ra_{L^2_v(m)} \lesssim \|f\|_{L^2_v(\la v \ra^\ell)} \, \|f\|^2_{H^{s,*}_v(m)}.
%			$$}
	\end{enumerate}
\end{lem}

\noindent {\it Proof of (i).} We write
		$$
		\bal
			\la Q(f,g),h \ra_{L^2_v(m)}
			&= \int_{\R^3 \times \R^3 \times \Sp^2} B(v-v_*,\sigma) \, (f'_*g' - f_*g) \, h \, m^2 \,\d \sigma \,\dv_* \,\dv \\
			&= \int_{\R^3 \times \R^3 \times \Sp^2} B(v-v_*,\sigma) \, (f'_* g' m' - f_* g\, m) \, h \, m \,\d \sigma \,\dv_* \,\dv \\
			&\quad
			+ \int_{\R^3 \times \R^3 \times \Sp^2} B(v-v_*,\sigma) \, f'_* g' h \, m \, (m-m') \,\d \sigma \,\dv_* \,\dv \\
			&=: I_1+I_2.
		\eal
		$$
	We deal with the first term $I_1$ using Lemma~\ref{lem:CH}:
		$$
		\bal
			I_1 = \la Q(f,gm), hm \ra_{L^2_v} &\lesssim \|f\|_{L^1_v(\la v \ra^{\gamma+2s})} \, \|g\|_{H^{\varsigma_1}_v(\la v \ra^{N_1} m)} \, \|h\|_{H^{\varsigma_2}_v(\la v \ra^{N_2} m)} \\
			&\lesssim  \|f\|_{L^2_v(\la v \ra^\ell)} \, \|g\|_{H^{\varsigma_1}_v(\la v \ra^{N_1} m)} \, \|h\|_{H^{\varsigma_2}_v(\la v \ra^{N_2} m)}
		\eal
		$$
	because $\ell>\gamma+2s+3/2$, with $\varsigma_1$, $\varsigma_2 \in [0,2s]$ satisfying $\varsigma_1+\varsigma_2=2s$, with { $N_1 \ge \gamma/2$} and~$N_2 \ge 0$ such that $N_1+N_2=\gamma+2s$.
	To deal with $I_2$, we use the following estimate on $|m'-m|$ (see the proof in~\cite[Lemma~2.3]{AMUXY-ARMA1}):
		\beqn \label{eq:m'-m}
			|m'-m| \lesssim \sin({\theta/2}) \left(m' + \la v'_* \ra \, \la v'\ra^{k-1} +\sin^{k-1}({\theta/2}) \, m'_* \right).
		\eeqn
	Notice that $|v-v_*| = |v'-v'_*| \lesssim |v-v'_*|$ which implies
		\beqn \label{eq:v-v*1}
			|v-v_*|^{\gamma} \lesssim |v'-v'_*|^{\gamma/2} \, |v-v'_*|^{\gamma/2} \lesssim \la v \ra^{\gamma/2} \, \la v' \ra^{\gamma/2} \, \la v'_* \ra^{\gamma}.
		\eeqn
	Also, we have,
		\beqn \label{eq:v-v*2}
		|v-v_*|^\gamma \lesssim |v'-v|^{\gamma/2} \, \sin^{-\gamma/2}({\theta/2}) \, |v'-v'_*|^{\gamma/2} \lesssim \sin^{-\gamma/2}({\theta/2}) \, \la v' \ra^\gamma \, \la v \ra^{\gamma/2} \, \la v'_* \ra^{\gamma/2}.
		\eeqn
	This bound induces the appearance of a singularity in $\theta$. However, we notice that  in the third term of the estimate~\eqref{eq:m'-m} we have a gain in the power of $\sin ({\theta/2})$ depending on the value of~$k$,
	the power of our polynomial weight. As a consequence, if~$k$ is large enough, we can keep a power of $\sin ({\theta/2})$ that is large enough to remove the singularity of $b(\cos \theta)$ at~$\theta=0$.
	Consequently, we have:
		$$
		\bal
			I_2 &\lesssim  \int_{\R^3 \times \R^3 \times \Sp^2} b(\cos \theta) \, \sin ({\theta/2}) \, |v-v_*|^\gamma \,| f'_*| |g'| |h| \, m  \\
			&\qquad \qquad \left(m' + \la v'_* \ra \, \la v'\ra^{k-1} +\sin^{k-1}({\theta/2}) \, m'_* \right) \,\d \sigma \,\dv_* \,\dv \\
			&=: I_{21} + I_{22} + I_{23}.
		\eal
		$$
	The two first terms $I_{21}$ and $I_{22}$ are treated in the same way using the estimate~\eqref{eq:v-v*1}, we obtain:
		$$
		\bal
			I_{21} + I_{22}
			&\lesssim  \int_{\R^3 \times \R^3 \times \Sp^2} b(\cos \theta) \,\sin ({\theta/2}) \, |f'_*| \la v'_*\ra^{\gamma+1} \, |g'| m' \la v' \ra^{\gamma/2} \, |h| \, m \la v \ra^{\gamma/2} \,\d \sigma \,\dv_* \,\dv \\
			& \lesssim  \left(\int_{\R^3 \times \R^3 \times \Sp^2} b(\cos \theta) \, \sin ({\theta/2}) \, |f'| \la v'\ra^{\gamma+1} \, (g'_*)^2 (m'_*)^2 \la v'_* \ra^{\gamma} \,\d \sigma \,\dv_* \,\dv \right)^{1/2} \\
			&\quad \times \left(\int_{\R^3 \times \R^3 \times \Sp^2} b(\cos \theta) \, \sin ({\theta/2})\, |f'| \la v'\ra^{\gamma+1} \,h_*^2 \, m_*^2 \la v_* \ra^{\gamma} \,\d \sigma \,\dv_* \,\dv\right)^{1/2} \\
			&=: J_1 \times J_2.
		\eal
		$$
	The term $J_1$ is easily handled just using the pre-post collisional change of variable:
		$$
		J_1^2 \lesssim \|f\|_{L^1_v (\la v \ra^{\gamma+1})} \, \|g\|^2_{L^2_v( \la v \ra^{\gamma/2} m)} \lesssim \|f\|_{L^2_v (\la v \ra^\ell)} \, \|g\|^2_{L^2_v( \la v \ra^{\gamma/2} m)}
		$$
	since $\ell > \gamma + 1 + 3/2$. To deal with $J_2$, we use the regular change of variable $v \to v'$ meaning that
	for each $\sigma$, with $v_*$ still fixed, we perform the change of variables $v \rightarrow v'$.
	This change of variables is well-defined on the set $\left\{ \cos \theta > 0 \right\}$. Its Jacobian determinant is
$$
\left| \frac{\d v'}{\d v}\right|
= \frac{1}{8} (1 + \kappa \cdot \sigma) = \frac{(\kappa' \cdot \sigma)^2}{4},
$$
	where $\kappa:=(v-v_*)/|v-v_*|$ and $\kappa':=(v'-v_*)/|v'-v_*|$. We have $\kappa' \cdot \sigma = \cos(\theta/2) \geq 1/\sqrt{2}$.
	The inverse transformation $v' \rightarrow \psi_\sigma(v')=v$ is then defined accordingly. Using the fact that
		$$
		\cos \theta = \kappa \cdot \sigma = 2 (\kappa' \cdot \sigma)^2 -1 \quad \text{and} \quad \sin (\theta/2) = \sqrt{ 1 - \cos^2(\theta/2)} = \sqrt{1 - (\kappa' \cdot \sigma)^2},
		$$
%	and also
%		$$
%		|\psi_\sigma(v)-v_*|=|v-v_*|/\kappa\cdot\sigma,
%		$$
we obtain
$$
		\bal
			&\int_{\R^3 \times \Sp^2} b(\cos \theta) \,\sin ({\theta/2}) \, |f'| \la v'\ra^{\gamma+1} \,\d \sigma \,\dv \\
			&\quad = \int_{\R^3 \times \Sp^2} b(2 (\kappa' \cdot \sigma)^2 -1) \,  \sqrt{ 1 - (\kappa' \cdot \sigma)^2} \, |f'| \la v'\ra^{\gamma+1} \,\d \sigma \,\dv \\
&\quad = \int_{\kappa'\cdot \sigma \ge 1/\sqrt{2}} b(2 (\kappa' \cdot \sigma)^2 -1) \,  \sqrt{ 1 - (\kappa' \cdot \sigma)^2} \, |f'| \la v'\ra^{\gamma+1}
\,\d \sigma \,{{4 \,\dv'}\over  {(\kappa'\cdot \sigma)^2}} \\
			&\quad \lesssim \int_{\Sp^2} b(\cos 2\theta) \, \sin \theta \,\d \sigma \int_{\R^3} |f| \langle v \rangle^{\gamma+1} \,\dv.
		\eal
$$
	We deduce:
		$$
		J_2^2 \lesssim  \|f\|_{L^1_v (\la v \ra^{\gamma+1})} \, \|h\|^2_{L^2_v( \la v \ra^{\gamma/2} m)} \lesssim \|f\|_{L^2_v (\la v \ra^\ell)} \, \|h\|^2_{L^2_v( \la v \ra^{\gamma/2} m)}.
		$$
	In summary, gathering the three previous estimates, we have
		$$
		I_{21} + I_{22} \lesssim \|f\|_{L^2_v (\la v \ra^\ell)} \, \|g\|_{L^2_v( \la v \ra^{\gamma/2} m)} \, \|h\|_{L^2_v( \la v \ra^{\gamma/2} m)}.
		$$
	Concerning $I_{23}$, we take advantage of the bound given by~\eqref{eq:v-v*2}:
		$$
		\bal
			I_{23}
			&\lesssim \int_{\R^3 \times \R^3 \times \Sp^2} b(\cos \theta) \, \sin^{k-\gamma/2}({\theta/2})\, |f'_*| m'_* \la v'_* \ra^{\gamma/2} \, |g'| \la v' \ra^{\gamma} \, |h| \, m \, \la v \ra^{\gamma/2} \,\d \sigma \,\dv_* \,\dv \\
			&\lesssim \left(\int_{\R^3 \times \R^3 \times \Sp^2} b(\cos \theta) \, \sin^{k-\gamma/2} ({\theta/2}) \, |g'| \la v' \ra^{\gamma} \,|f'_*|^2 {m'_*}^2 \la v'_* \ra^{\gamma }\,\d \sigma \,\dv_* \,\dv \right)^{1/2} \\
			&\quad \times \left(\int_{\R^3 \times \R^3 \times \Sp^2} b(\cos \theta) \, \sin^{k-\gamma/2} ({\theta/2}) \, |g'| \la v' \ra^{\gamma} \, h^2 \, m^2 \, \la v \ra^{\gamma} \,\d \sigma \,\dv_* \,\dv \right)^{1/2} \\
			&=: T_1 \times T_2.
		\eal
		$$
	As far as $T_1$ is concerned, a simple pre-post collisional change of variable allows us to get
		$$
		T_1^2 \lesssim \| g\|_{L^1_v(\la v \ra^\gamma)} \, \|f\|^2_{L^2_v(\la v \ra^{\gamma/2} m)} \lesssim \| g\|_{L^2_v(\la v \ra^\ell)} \, \|f\|^2_{L^2_v(\la v \ra^{\gamma/2} m)}
		$$
	since $\ell > \gamma+3/2$. The second term requires more attention since we have to perform a singular change of variable $v_* \to v'$ showed for example in the proof of Lemma 2.4 in~\cite{AMUXY-ARMA1}.
	Recall that the Jacobian of this { transformation} is
		$$
		\left|\frac {\d v_*}{\d v'}\right|= \frac{4}{\sin^2(\theta/2)}\le 16  \, \theta^{-2}, \, \theta \in (0,\pi/2],
		$$
	therefore, this change of variable gives rise to an additional singularity in $\theta$ around $0$. However, we can take advantage of the fact that we have a power $k$ in  $\sin  ({\theta/2})$,
	indeed taking~$k$ large enough allows us to control this singularity.
	Notice that $\theta$ is no longer the good polar angle to consider, we set $\psi = (\pi-\theta)/2$ for $\psi \in [\pi/4,\pi/2]$ so that
		$$
		\cos \psi = {{v'-v} \over |v'-v|} \cdot \sigma \quad \text{and} \quad\d \sigma = \sin \psi \, \d\psi \, \d\phi.
		$$
	This measure does not cancel any of the singularity of $b(\cos \theta)$ unlike in the case of the usual polar coordinates but it will be counterbalanced taking $k$ large enough. We then have { (using the fact that $b$ is supported on $0 \leq \theta \leq \pi/2$)}:
		$$
		\bal
			&\int_{\R^3 \times \Sp^2} b(\cos \theta) \, \sin^{k-\gamma/2} ({\theta/2}) \, |g'| \la v' \ra^{\gamma} \,\d \sigma \,\dv_* \\
			&\quad\lesssim \int_{\R^3 \times \Sp^2} { \mathds{1}_{[\pi/4,\pi/2]}(\psi)}(\pi - 2\psi)^{ k-\gamma/2-3-2s} \, |g'| \la v' \ra^{\gamma} \,\d \sigma \,\dv' \\
			& \quad \lesssim \int_{\pi/4}^{\pi/2} (\pi - 2\psi)^{ k-\gamma/2-3-2s} \, \sin \psi \, \d\psi \, \int_{\R^3} |g| \, \la v \ra^\gamma \,\dv \lesssim \int_{\R^3} |g| \, \la v \ra^\gamma \, \d v
		\eal
		$$
	since ${ k>\gamma/2+2+2s}$. We deduce that
		$$
		T_2^2 \lesssim \| g\|_{L^1_v(\la v \ra^\gamma)} \, \|h\|^2_{L^2_v(\la v \ra^{\gamma/2} m)} \lesssim \| g\|_{L^2_v(\la v \ra^\ell)} \, \|h\|^2_{L^2_v(\la v \ra^{\gamma/2} m)}
		$$
	and thus
		$$
		I_{23} \lesssim  \|f\|_{L^2_v(\la v \ra^{\gamma/2} m)}\,  \| g\|_{L^2_v(\la v \ra^\ell)} \, \|h\|_{L^2_v(\la v \ra^{\gamma/2} m)},
		$$
	which concludes the proof of estimate~\eqref{eq:Q(f,g),h}.
	
\noindent {\it Proof of (ii).}  {We have:
	\begin{align*}
	&\langle Q(f,g),g \rangle_{L^2_v(m)}  \\
	&\quad = \langle Q(f,gm),gm) \rangle_{L^2_v} + \int_{\R^3 \times \R^3 \times \Sp^2} B(v-v_*,\sigma) f'_* g' g \, m \, (m-m')\,\d\sigma \,\dv_* \,\dv \\
	&\quad =: I+J.
	\end{align*}

The term $J$ is done in the first step of the proof, it corresponds to the term $I_2$ replacing~$h$ by $g$, we thus have
$$
J \lesssim\|f\|_{L^2_v (\la v \ra^{\gamma/2}m)} \, \|g\|_{L^2_v(\la v \ra^\ell)} \, \|g\|_{L^2_v (\la v \ra^{\gamma/2}m)} { + \|f\|_{L^2_v (\la v \ra^\ell)} \, \|g\|_{L^2_v(\la v \ra^{\gamma/2}m)} \, \|g\|_{L^2_v (\la v \ra^{\gamma/2}m)}} .
$$

In order to deal with the term $I$, we denote $G:=gm$. We also recall that
$$
	b_\delta (\cos \theta) = \chi_{\delta}(\theta) \, b(\cos \theta)
$$
and we introduce the notations
$$
	b^c_\delta(\cos \theta) := (1-\chi_\delta(\theta)) \, b(\cos \theta),
$$
$$
	B_\delta(v-v_*, \sigma) := b_\delta(\cos \theta) \, |v-v_*|^\gamma \quad
\text{and} \quad B^c_\delta(v-v_*, \sigma) := b^c_	\delta(\cos \theta)\, |v-v_*|^\gamma.
	$$
The two previous kernels correspond respectively to grazing collisions and non grazing collisions (which encodes the cut-off part of the operator). We also denote $Q_\delta$ (resp. $Q_\delta^c$) the operator associated with the kernel $B_\delta$ (resp. $B_\delta^c$).
Note that this splitting of the collision kernel will be used in Section~\ref{sec:lin}.
We have for $G=gm$:
$$
I = \la Q_\delta(f,G),G \ra_{L^2_v} +\la Q_\delta^c(f,G),G \ra_{L^2_v} =:I^\delta + I^{\delta,c}.
$$
We start by dealing with the cut-off part:
\begin{align*}
I^{\delta,c} &= \int_{\R^3 \times \R^3 \times \Sp^2} B_\delta^c(v-v_*,\sigma) f_* \, G (G'-G)
\,\d \sigma \,\dv_* \,\dv \\
&\lesssim \int_{\R^3 \times \R^3 \times \Sp^2} |v-v_*|^\gamma\, b_\delta^c(\cos \theta)\,
|f_*| \, (G^2+(G')^2) \,\d \sigma \,\dv_* \,\dv.
\end{align*}
Using that $b_\delta^c(\cos \theta) \le C_\delta$ on $\Sp^2$ and $|v-v_*|^\gamma \lesssim |v'-v_*|^\gamma$, we get
\begin{align*}
I^{\delta,c} &\lesssim\int_{\R^3 \times \R^3 \times \Sp^2} |f_*| \la v_*\ra^\gamma \, G^2
\,\la v \ra^\gamma \,\d \sigma \,\dv_* \,\dv + \int_{\R^3 \times \R^3 \times \Sp^2} |f_*|
\la v_*\ra^\gamma\, G'^2 \, \la v' \ra^\gamma  \,d\sigma \,\dv_* \,\dv.
\end{align*}
The first term is directly bounded from above by $\|f\|_{L^1_v(\la v \ra^\gamma)} \|G\|^2_{L^2_v(\la v \ra^{\gamma/2})}$ and for the second one, we use the regular change of variable $v \to v'$ explained in the proof of (i). We thus get
$$
I^{\delta,c} \lesssim \|f\|_{L^1_v(\la v \ra^\gamma)} \|G\|^2_{L^2_v(\la v \ra^{\gamma/2})}\lesssim \|f\|_{L^2_v(\la v \ra^\ell)}\|g\|^2_{L^2_v(\la v \ra^{\gamma/2}m)}.
$$
Concerning the grazing collisions part, we write
\begin{align*}
&I^\delta = \int_{\R^3 \times \R^3 \times \Sp^2} B_\delta(v-v_*,\sigma) f_* \, G \,(G'-G) \,\d \sigma \,\dv_* \,\dv \\
&\qquad = -{1 \over 2}  \int_{\R^3 \times\R^3 \times \Sp^2} B_\delta(v-v_*,\sigma) f_* \, (G'-G)^2 \,\d \sigma \, \dv_* \,\dv \\
&\qquad \quad+{1 \over 2} \int_{\R^3 \times\R^3 \times \Sp^2} B_\delta(v-v_*,\sigma) f_* \, ((G')^2-G^2) \,\d \sigma \, \dv_* \,\dv
=: I^\delta_1+I^\delta_2.
\end{align*}
The second term $I^\delta_2$ is treated thanks to the cancellation lemma~\cite[Lemma~1]{ADVW-ARMA} { (recalled in Appendix~\ref{app:canc+Carl})}:
		$$
		I_2^\delta  = \int_{\R^3} (S_\delta*G^2) \, f   \,\dv,
		$$
	where (for details, see~\cite[proof of Lemma~2.2]{Trist-JSP})
		\beqn \label{eq:Sdelta}
			S_\delta(z) %:= 2 \pi \, \int_0^{\pi/2} \sin\theta \, b_\delta(\cos \theta) \left( \frac{|z|^\gamma}{\cos^{\gamma+3} (\theta/2)} - |z|^\gamma \right) \, d\theta
  			\lesssim \delta^{2-2s} \, |z|^\gamma.
		\eeqn
We deduce that
	$$
	I^\delta_2 \lesssim \|f\|_{L^1_v(\langle v \rangle^\gamma)} \|G\|^2_{L^2_v(\langle v \rangle^{\gamma/2})}\lesssim \|f\|_{L^2_v(\la v \ra^\ell)}\|g\|^2_{L^2_v(\la v \ra^{\gamma/2}m)}.
	$$
It now remains to handle $I^\delta_1$. First, using that $|v-v_*| \lesssim |v'-v_*|$, we have
\begin{align*}
I^\delta_1 &\lesssim \int_{\R^3 \times\R^3 \times \Sp^2} b_\delta(\cos \theta) |v-v_*|^\gamma |f_*| \, (G'-G)^2 \,\d \sigma \, \dv_* \,\dv \\
&\lesssim \int_{\R^3 \times\R^3 \times \Sp^2} b_\delta(\cos \theta) |v'-v_*|^\gamma |f_*| \, (G'-G)^2 \,\d \sigma \, \dv_* \,\dv \\
&\lesssim \int_{\R^3 \times\R^3 \times \Sp^2} b_\delta(\cos \theta)  |f_*| \langle v_* \rangle^\gamma \, (G'\langle v' \rangle^{\gamma/2} -G \langle v \rangle^{\gamma/2})^2 \,\d \sigma \, \dv_* \,\dv \\
&\quad + \int_{\R^3 \times\R^3 \times \Sp^2} b_\delta(\cos \theta)  |f_*| \langle v_* \rangle^\gamma \, G^2 \, (\langle v \rangle^{\gamma/2} - \langle v' \rangle^{\gamma/2})^2\,\d \sigma \, \dv_* \,\dv =: I^\delta_{11} + I^\delta_{12}.
\end{align*}
To deal with $I^\delta_{12}$, we first note that
\begin{align*}
&|\langle v \rangle^{\gamma/2} - \langle v' \rangle^{\gamma/2}| \lesssim |v'-v| \int_0^1 \langle v'+\tau (v-v') \rangle^{\gamma/2-1} \, \d\tau \\
&\qquad \lesssim |v-v_*| \sin(\theta/2) \int_0^1 \la v_\tau\ra^{\gamma/2-1} \, \d\tau
\end{align*}
where $v_\tau := v'+\tau (v-v')$. Moreover, for any $\tau \in [0,1]$, we have
$$
\langle v \rangle \le \la v-v_* \ra + \la v_* \ra \le \sqrt{2} \langle v_\tau - v_* \ra + \la v_* \ra \lesssim \la v_\tau \ra \la v_* \ra
$$
which implies (since $\gamma/2-1 \le 0$)
$$
\la v_\tau \ra^{\gamma/2-1} \lesssim \la v \ra^{\gamma/2-1} \la v_* \ra^{1-\gamma/2}.
$$
Consequently, we deduce
\beqn \label{eq:vgamma/2}
(\langle v \rangle^{\gamma/2} - \langle v' \rangle^{\gamma/2})^2 \lesssim |v-v_*|^2\sin^2(\theta/2) \la v \ra^{\gamma-2} \la v_* \ra^{2-\gamma}
\lesssim \sin^2(\theta/2)\la v \ra^{\gamma} \la v_* \ra^{4-\gamma}
\eeqn
so that
$$
I^\delta_{12} \lesssim \|f\|_{L^1_v(\la v \ra^{4-\gamma})} \|G\|^2_{L^2_v(\la v \ra^{\gamma/2})}\lesssim \|f\|_{L^2_v(\la v \ra^\ell)}\|g\|^2_{L^2_v(\la v \ra^{\gamma/2}m)}.
$$
For the analysis of $I_{11}^\delta$, we introduce the following notations: $\tilde f := f \langle \cdot \ra^\gamma$, $\tilde \mu := \mu \la v \ra^{-\gamma}$ and~$\GG := G \la v \ra^{\gamma/2}$ so that
$$
I^\delta_{11} = \int_{\R^3 \times \R^3 \times \Sp^2} b_\delta(\cos \theta) |\tilde f| (\GG'-\GG)^2 \,\d \sigma \, \dv_* \,\dv.
$$
We then use Bobylev formula~\cite{Bob-SSR} (see also~\cite[Proposition~2]{ADVW-ARMA}), denoting $\xi^{\pm}=(\xi \pm|\xi|\sigma)/2$, we have:
\begin{align*}
I_{11}^\delta &= {1 \over (2\pi)^3} \int_{\R^3 \times \Sp^2} b_\delta\left({\xi \over |\xi|} \cdot \sigma\right) \Bigg(\widehat{|\tilde{f}|}(0) |\widehat{\GG}(\xi) - \widehat{\GG}(\xi^+)|^2 \\
&\quad \qquad \qquad \quad+ 2 \, \Re  \left(\widehat{|\tilde{f}|}(0) - \widehat{|\tilde{f}|}(\xi^-)\right)
\widehat{\GG}(\xi^+) \overline{\widehat{\GG}}(\xi) \Bigg)\,\d \sigma \, d\xi.
\end{align*}
Similarly, we have
\begin{align*}
\|G\|^2_{\dot H^{s,*}_v} &=  {1 \over (2\pi)^3} \int_{\R^3 \times \Sp^2} b_\delta\left({\xi \over |\xi|} \cdot \sigma\right) \Bigg(\widehat{\tilde\mu}(0) |\widehat{\GG}(\xi) - \widehat{\GG}(\xi^+)|^2 \\
&\quad \qquad \qquad \quad+ 2 \, \Re  \left(\widehat{\tilde\mu}(0) - \widehat{\tilde\mu}(\xi^-)\right)
\widehat{\GG}(\xi^+) \overline{\widehat{\GG}}(\xi) \Bigg)\,\d \sigma \, d\xi.
\end{align*}
Since $\widehat{|\tilde{f}|}(0) = \|\tilde f \|_{L^1_v}$ and $\widehat{\tilde \mu}(v) = \|\tilde\mu\|_{L^1_v}$, we deduce that
\begin{align*}
I_{11}^\delta &= {1 \over (2\pi)^3}  \int_{\R^3 \times \Sp^2} b_\delta\left(\xi \over |\xi| \cdot \sigma\right)2 \, \Re \left(\widehat{|\tilde{f}|}(0) - \widehat{|\tilde{f}|}(\xi^-)\right)
\widehat{\GG}(\xi^+) \overline{\widehat{\GG}}(\xi) \,\d \sigma \, d\xi \\
&\quad -{1 \over (2\pi)^3} \frac{\|\tilde f\|_{L^1_v}}{\|\tilde\mu\|_{L^1_v}} \int_{\R^3 \times \Sp^2}   b_\delta\left(\xi \over |\xi| \cdot \sigma\right)2 \, \Re \left(\widehat{\tilde\mu}(0) - \widehat{\tilde\mu}(\xi^-)\right)
\widehat{\GG}(\xi^+) \overline{\widehat{\GG}}(\xi) \,\d \sigma \, d\xi \\
&\quad + \frac{\|\tilde f\|_{L^1_v}}{\|\tilde\mu\|_{L^1_v}} \|G\|^2_{\dot H^{s,*}_v} =: I_{111}^\delta+I_{112}^\delta+I_{113}^\delta.
\end{align*}
Using then results from the proof of~\cite[Lemma~2.8]{AMUXY-CMP} combined with Lemma~\ref{lem:anis}, we get that
$$
I_{111}^\delta \lesssim \|\tilde f\|_{L^1_v(\la v \ra^{2s})} \|\GG\|^2_{H^s_v} \lesssim  \|f\|_{L^2_v(\la v \ra^{\ell})} \|g\|^2_{H^{s,*}_v(m)}
$$
and
$$
I_{112}^\delta \lesssim \|\tilde f\|_{L^1_v}\|\GG\|^2_{H^s_v} \lesssim \|f\|_{L^2_v(\la v \ra^\ell)} \|g\|^2_{H^{s,*}_v(m)}.
$$
We also clearly have
$$
I_{113}^\delta \lesssim \|f\|_{L^2_v(\la v \ra^\ell)} \|g\|^2_{H^{s,*}_v(m)}.
$$
Gathering all the previous estimates, we are able to deduce that~\eqref{eq:Q(f,g),g} holds.
}
%\noindent {\it Proof of (iii).} The result is immediately obtained taking $g=f$ in~\eqref{eq:Q(f,g),g}.
\qed

\medskip
\subsection{Non homogeneous estimates} \label{subsec:nonhom}
We now state non homogeneous estimates on the trilinear form $\la Q(f,g),h \ra$ (the proof, which is a consequence of Lemma~\ref{lem:nonlinearhom} and Sobolev embeddings in $x$, is given in Appendix~\ref{app:nonlinear}) in order to get some accurate estimates on the terms coming from the nonlinear part of the equation. Basically, we give a non homogeneous version of Lemma~\ref{lem:nonlinearhom}.
We introduce the spaces
{ \beqn \label{eq:XY}
\left\{
		\bal
			&\mathbf X := \HH^3_x L^2_v(m) \\
			&\mathbf Y := \HH^{3,s}_{x,v}( \la v \ra^{\gamma/2}m) \\
			&\mathbf Y^* := \HH^{3,s,*}_{x,v}(m) \\
			&\bar {\mathbf Y} := \HH^{3,s}_{x,v}(\la v \ra^{\gamma/2+2s} m)
		\eal
\right.
\eeqn}that are defined through their norms by~\eqref{eq:HHnxL2v} and~\eqref{eq:anis}.
We also introduce $\mathbf Y'$ the dual space of $\mathbf Y$ with respect to the pivot space $\mathbf X$, meaning that the $\mathbf Y'$-norm is defined through:
	\beqn \label{eq:Y'}
	\|f\|_{\mathbf Y'} := \sup_{\|\phi\|_{\mathbf Y} \le 1} \langle f, \phi \rangle_{\mathbf X} = \sup_{\|\phi\|_{\mathbf Y} \le 1}  \sum_{0 \le j \le 3} \langle \nabla^j_x f, \nabla^j_x \phi \rangle_{L^2_{x,v}(\langle v \rangle^{-2js}  m)}.
	\eeqn

\begin{lem} \label{lem:nonlinearnonhom}
The following estimates hold:
	\begin{enumerate}[leftmargin=*]
		\item[(i)] For {$k > \gamma/2 +3 + 8s$},
			$$
			\la Q(f,g),h \ra_{\mathbf X} \lesssim \|f\|_{\mathbf X} \, \|g\|_{\bar {\mathbf Y}} \, \|h\|_{\mathbf Y} + \| f\|_{\mathbf Y} \, \|g\|_{\mathbf X} \, \|h\|_{\mathbf Y};
			$$
		therefore,
			$$
			\| Q(f,g) \|_{\mathbf Y'} \lesssim \|f\|_{\mathbf X} \, \|g\|_{\bar {\mathbf Y}} +  \| f\|_{\mathbf Y} \, \|g\|_{\mathbf X}.
			$$
		\item[(ii)] { For {$k > 4-\gamma+3/2+6s$},
			$$
			\la Q(f,g),g\ra_{\mathbf X} \lesssim \|f\|_{\mathbf X} \, \|g\|^2_{\mathbf Y^*} + \|f\|_{\mathbf Y} \,  \|g\|_{\mathbf X} \, \|g\|_{\mathbf Y}.
			$$}
		\item[(iii)] { For { $k > 4-\gamma+3/2+6s$},
			$$
			\la Q(f,f),f\ra_{\mathbf X} \lesssim \|f\|_{\mathbf X} \, \|f\|^2_{\mathbf Y^*}.
			$$}
	\end{enumerate}
\end{lem}

%----------------------------------%----------------------------------%----------------------------------%----------------------------------%----------------------------------%----------------------------------%----------------------------------%----------------------------------%
\medskip
\subsection{Some estimates on the linearized operator} \label{subsec:estimlin}
Let us now introduce another type of splitting for the collision kernel (which will be used in Section~\ref{sec:reg} where we study the regularization properties of the Boltzmann linearized operator).
We denote $\widetilde Q_1$ the operator associated to the kernel:
$$
\widetilde B_1 (v-v_*,\sigma) := \chi(|v'-v|) \, b(\cos \theta) \, |v-v_*|^\gamma
$$
and $\widetilde Q^c_1$ the one associated to the remainder part of the kernel:
$$
\widetilde B^c_1 (v-v_*,\sigma):=(1-\chi(|v'-v|)) \, b(\cos \theta) \, |v-v_*|^\gamma.
$$
In the next lemma, we only give estimates on parts of the linearized Boltzmann operator (one of the variable is the Maxwellian $\mu$) which are ``almost bounded'' in the sense that there is no loss of regularity in terms of derivative. Denote
\begin{align*}
&\Lambda_2 f:=
K \langle v \rangle^{\gamma+2s} f + \int_{\R^3 \times \mathbb{S}^2} \widetilde B_1(v-v_*,\sigma) (\mu'_* - \mu_*)(f'+f) \, \d\sigma \, \dv_*+ \widetilde Q_1^c(\mu,f) + Q(f,\mu)
\end{align*}
where $K$ is a positive parameter to be chosen later on (the notation used here is the one used in Paragraph~\ref{subsubsec:lambda1m*}).

\begin{lem} \label{lem:lambda_2}
Let {$k\ge 0$}. For any $K>0$ and for any $\ell>3/2$, we have the following estimate:
\begin{equation} \label{eq:lambda_2}
\|\Lambda_2 f\|_{H^\varsigma_{x,v}(m)} \lesssim \|f\|_{H^\varsigma_{x,v} (\langle v \rangle^{\gamma + 1 + \ell} m)}, \quad \forall \, \varsigma \in \R^+.
\end{equation}
\end{lem}
\begin{proof}
We only look at the case $\varsigma \in \N$ and conclude that the result also holds for~$\varsigma \in \R^+$ by an interpolation argument. Let us begin with the case $\varsigma=0$ i.e. the $L^2$-case. We have
$$
\begin{aligned}
\Lambda_2 f &= K \langle v \rangle^{\gamma+2s} f + \int_{\R^3 \times \mathbb{S}^2} \widetilde B_1(v-v_*,\sigma) (\mu'_* - \mu_*)f' \, \d\sigma \, \dv_* \\
&\quad +  \int_{\R^3 \times \mathbb{S}^2} \widetilde B_1(v-v_*,\sigma) (\mu'_* - \mu_*)\, \d\sigma \, \dv_* \, f + \int_{\R^3\times \mathbb{S}^2} \widetilde B^c_1(v-v_*,\sigma) \mu'_* f' \, \d\sigma \, \dv_*  \\
&\quad -  \int_{\R^3\times \mathbb{S}^2} \widetilde B^c_1(v-v_*,\sigma) \mu_*\, \d\sigma \, \dv_* \, f + Q(f,\mu)\\
&=: \Lambda_{21}f + \Lambda_{22} f + \Lambda_{23}f + \Lambda_{24} f + \Lambda_{25} f + \Lambda_{26} f.
\end{aligned}
$$
The estimate on $\Lambda_{21}$ is obvious:
$$
\|\Lambda_{21}f\|_{L^2_{x,v}(m)} \lesssim \|f\|_{L^2_{x,v} (\langle v \rangle^{\gamma+2s}m)}.
$$
{ The analysis of $\Lambda_{23}$ is also easy to perform using the cancellation lemma from~\cite{ADVW-ARMA} (see Appendix~\ref{app:canc+Carl}), we have:
$$
\Lambda_{23}f = (\widetilde S * \mu) f
$$
with $\widetilde S$ satisfying the estimate $|\widetilde S(z)| \lesssim |z|^{\gamma+2s-2}$ (see Lemma~2.3 from~\cite{AHL*}). We deduce that $|\widetilde S * \mu|(v) \lesssim \langle v \rangle^{\gamma+2s-2}$ and thus
$$
\|\Lambda_{23}f\|_{L^2_{x,v}(m)} \lesssim \|f\|_{L^2_{x,v} (\langle v \rangle^{\gamma+2s-2}m)}.
$$}
To treat $\Lambda_{24}$ and $\Lambda_{25}$, we use the fact that the kernel $\widetilde B^c_1$ is not singular because the grazing collisions are removed. Since $|v'-v| \sim |v-v_*| \sin(\theta/2)$, we have:
$$
|\widetilde B^c_1(v-v_*,\sigma)| \le b(\cos\theta) |v-v_*|^\gamma \mathds{1}_{|v'-v| \ge 1} \lesssim  b(\cos\theta) |v-v_*|^{\gamma+1} \sin (\theta/2).
$$
%where we do not mention the dependency in $\delta$ of the multicplicative constant.
%{\color{red} Ici, on pourrait traiter $s \ge 1/2$ en majorant par $|v-v_*|^{\gamma+2} \sin^2\left(\theta \over 2\right)$ dans l'estimation.}
Consequently, we obtain using that $m \lesssim m' m'_*$ that for $\ell>3/2$:
$$
\begin{aligned}
&\|\Lambda_{24}f\|^2_{L^2_{x,v}(m)} \lesssim \int_{\T^3 \times \R^3} \left(\int_{\R^3\times \mathbb{S}^2} b(\cos \theta) \sin (\theta/2)|v-v_*|^{\gamma+1} \mu'_* \, |f'| \,  \d\sigma \, \dv_* \right)^2 m^2 \, \dv \, \d x \\
&\lesssim \int_{\T^3 \times \R^3 \times \R^3 \times \mathbb{S}^2} b(\cos \theta) \sin (\theta/2) |v-v_*|^{2(\gamma+1)} (\mu'_* m'_*)^2 \, (f'm')^2 \, \langle v_* \rangle^{2\ell} \, \d\sigma \, \dv_* \,  \dv \, \d x
\end{aligned}
$$
where we have used Jensen inequality with the finite measure $b(\cos \theta) \sin\left(\theta / 2\right){\rm d}\sigma$ and Cauchy-Schwarz inequality with the measure $\langle v_* \rangle^\ell \dv_*$. Then, using the basic inequality~$\langle v_* \rangle \lesssim \langle v' \rangle \langle v'_* \rangle$ and the pre-post collisional change of variable, we get:
$$
\begin{aligned}
&\quad \|\Lambda_{24}f\|^2_{L^2_{x,v}(m)} \\
&\lesssim
\int_{\T^3 \times \R^3 \times \R^3 \times \mathbb{S}^2} b(\cos \theta) \sin (\theta/2) |v-v_*|^{2(\gamma+1)} (\mu_*m_*)^2 \, (fm)^2 \, \langle v \rangle^{2 \ell} \, \langle v_* \rangle^{2\ell} \, \d\sigma \, \dv_* \,  \dv \, \d x \\
&\lesssim
\|f\|^2_{L^2_{x,v} (\langle v \rangle^{\gamma+1+\ell} m)} \quad \text{with} \quad \ell>3/2.
\end{aligned}
$$
The treatment of $\Lambda_{25}$ is easier and we directly obtain:
$$
\|\Lambda_{25}f\|_{L^2_{x,v}(m)} \lesssim \|f\|_{L^2_{x,v} (\langle v \rangle^{\gamma+1} m)}.
$$
Concerning $\Lambda_{26}$, we have for any $\ell>3/2$:
$$
\|Q(f,\mu)\|_{L^2_v(m)} \lesssim  \|f\|_{L^2_v(\langle v \rangle^{\gamma+2s + \ell}m)}
$$
where we used~\cite[Theorem~2.1]{AMUXY-ARMA1}. We deduce that
$$
\|\Lambda_{26}f\|_{L^2_{x,v}(m)} \lesssim \|f\|_{L^2_{x,v} (\langle v \rangle^{\gamma+2s+\ell}m)}, \quad \ell >3/2.
$$
It now remains to deal with $\Lambda_{22}$. We have:
%{\color{red} Je ne sais pas estimer ce terme si $s \ge 1/2$ sans perte de r\'egularit\'e. Probablement qu'il faudrait \^etre plus fin dans le d\'ecoupage et inclure cette partie dans le $\Lambda_1$ mais \c{c}a ne me semble pas \'{e}vident, au moins en ne compliquant pas trop les choses...} We denote $M:= \sqrt \mu$ so that we have:
$$
\begin{aligned}
|\Lambda_{22} f| &\le \int_{\R^3 \times \mathbb{S}^2} \widetilde B_1(v-v_*,\sigma) |\sqrt{\mu}'_*-\sqrt{\mu}_*| (\sqrt{\mu}'_* + \sqrt{\mu}_*) |f'| \, \d\sigma \, \dv_* \\
&\lesssim  \int_{\R^3 \times \mathbb{S}^2} b(\cos \theta) \sin (\theta/2) |v-v_*|^{\gamma+1} (\sqrt{\mu}'_* + \sqrt{\mu}_*) |f'| \, \d\sigma \, \dv_*
\end{aligned}
$$
where we used that the gradient of {$\sqrt \mu$} is bounded on $\R^d$. Then we use that $m \lesssim m' m'_*$ and $m \lesssim  \langle v-v_* \rangle^k m_*$ to get:
$$
\begin{aligned}
\|\Lambda_{22} f\|^2_{L^2_{x,v}(m)} &\lesssim \int_{\T^3 \times \R^3} \left(  \int_{\R^3 \times \mathbb{S}^2} b(\cos \theta) \sin (\theta/2) |v-v_*|^{\gamma+1} \sqrt{\mu}'_*|f'| \, \d\sigma \, \dv_* \right)^2 m^2 \dv \, \d x\\
&\quad+  \int_{\T^3 \times \R^3} \left(  \int_{\R^3 \times \mathbb{S}^2} b(\cos \theta) \sin (\theta/2) |v-v_*|^{\gamma+1} \sqrt{\mu}_*|f'| \, \d\sigma \, \dv_* \right)^2 m^2 \dv \, \d x\\
&\lesssim \int_{\T^3 \times \R^3} \left(  \int_{\R^3 \times \mathbb{S}^2} b(\cos \theta) \sin (\theta/2) |v-v_*|^{\gamma+1} \sqrt{\mu}'_* m'_* |f'| m' \, \d\sigma \, \dv_* \right)^2 \dv \, \d x \\
&\quad + \int_{\T^3 \times \R^3} \left(  \int_{\R^3 \times \mathbb{S}^2} b(\cos \theta) \sin (\theta/2) \langle v-v_* \rangle^{\gamma+1+k} \sqrt{\mu}_* m_*|f'| \, \d\sigma \, \dv_* \right)^2 \dv \, \d x \\
&=: I_1 + I_2.
\end{aligned}
$$
Using Jensen inequality and H\"{o}lder inequality as previously, we obtain for $\ell>3/2$:
$$
\begin{aligned}
&\|\Lambda_{22} f\|^2_{L^2_{x,v}(m)} \\
&\quad\lesssim \int_{\T^3 \times\R^3\times\R^3 \times \mathbb{S}^2} b(\cos \theta) \sin (\theta/2) |v-v_*|^{2(\gamma+1)} \mu'_* (m'_*)^2|f'|^2 (m')^2 \langle v_* \rangle^{2\ell} \, \d\sigma \, \dv_* \, \dv \, \d x\\
&\qquad+  \int_{\T^3 \times\R^3\times\R^3 \times \mathbb{S}^2} b(\cos \theta) \sin (\theta/2) \langle v-v_* \rangle^{2(\gamma+1+k)} \mu_* m_*^2 |f'|^2  \langle v_* \rangle^{2\ell} \,  \d\sigma \, \dv_* \, \dv \, \d x\\
&\quad =: I_1 + I_2.
\end{aligned}
$$
The first term $I_1$ is treated as $\Lambda_{24}$ and we thus have:
$$
I_1 \lesssim \|f\|^2_{L^2_{x,v} (\langle v \rangle^{\gamma+1+\ell} m)} .
$$
Concerning $I_2$, we first look at the integral
$$
J :=  \int_{\R^3 \times \mathbb{S}^2} b(\cos \theta) \sin (\theta/2)  \langle v-v_* \rangle^{2(\gamma+1+k)} |f'|^2 \, \d\sigma \, \dv.
$$
and we use the regular change of variable $v \to v'$ explained in the proof of Lemma~\ref{lem:nonlinearhom}-(i). We get
$$
J \lesssim \int_{\mathbb{S}^2} b(\cos (2\theta)) \sin(\theta) \, \d\sigma \int_{\R^3} f^2 m^2 \langle v \rangle^{2(\gamma+1)} \, \dv \, \langle v_* \rangle^{2(\gamma+1)} m_*^2
$$
and thus
$$
I_2 \lesssim  \|f\|^2_{L^2_{x,v} (\langle v \rangle^{\gamma+1} m)},
$$
which concludes the proof in the case $\varsigma=0$.

Let us now explain briefly how to treat higher order derivatives: We only deal with the $H^1$-case, the other cases being handled similarly. For the derivative in $x$, we have immediately that for any $\ell>3/2$,
$$
\|\nabla_x \Lambda_2 f\|_{L^2_{x,v}(m)} \lesssim \|\nabla_xf\|_{L^2_{x,v} (\langle v \rangle^{\gamma + 1 + \ell} m)}
$$
since the operators $\nabla_x$ and $\Lambda_2$ commute ($\Lambda_2$ acts only in velocity). Concerning the derivative in $v$, we have to be more careful and in what follows, we only give the key points to obtain the final estimate. For the first term, we have:
$$
|\nabla_v \Lambda_{21} f| \lesssim \langle v \rangle^{\gamma+2s-1} |f| + \langle v \rangle^{\gamma+2s} |\nabla_v f|.
$$
{For $\Lambda_{23}$, using the cancellation lemma, we have
$$
\nabla_v (\Lambda_{23} f) = (\widetilde S*\nabla_v \mu) f + (\widetilde S*\mu) \nabla_v f
$$
and we also have $|\widetilde S*\nabla_v \mu| \lesssim \langle v \rangle^{\gamma+2s-2}$.}
For $\Lambda_{26}$ we can use the classical result (see~\cite{Vill-JMPA}) that tells us
$$
\nabla_v Q(f,\mu) = Q(\nabla_v f, \mu) + Q(f,\nabla_v \mu).
$$
In the same spirit that the latter formula is proven, one can show that
$$
\nabla_v \Lambda_{22}f =  \int_{\R^3 \times \mathbb{S}^2} \widetilde B_1(v-v_*,\sigma) ((\nabla_v\mu)'_* - (\nabla_v\mu)_*)f' \, \d\sigma \, \dv_* + \Lambda_{22} (\nabla_vf),
$$
$$
\nabla_v \Lambda_{24} f = \int_{\R^3\times \mathbb{S}^2} \widetilde B_1^c(v-v_*,\sigma) (\nabla_v\mu)'_* f' \, \d\sigma \, \dv_*   + \Lambda_{24} (\nabla_v f)
$$
and
$$
\nabla_v \Lambda_{25} f = -  \int_{\R^3\times \mathbb{S}^2} \widetilde B_1^c(v-v_*,\sigma) (\nabla_v\mu)_*\, \d\sigma \, \dv_* \, f + \Lambda_{25}(\nabla_v f).
$$
The key elements to prove those relations are that $\nabla_v \widetilde B_1 = -\nabla_{v_*} \widetilde B_1$ and that we have for any suitable function $f$:
$$
(\nabla_v + \nabla_{v_*})(f')=\left(\nabla_v f\right)' \quad \text{and} \quad (\nabla_v + \nabla_{v_*})(f'_*)=\left(\nabla_v f\right)'_*.
$$
Gathering the previous remarks, we are then able to obtain that for any $\ell>3/2$:
$$
\|\nabla_v \Lambda_2 f\|_{L^2_{x,v}(m)} \lesssim \|f\|_{L^2_{x,v} (\langle v \rangle^{\gamma + 1 + \ell} m)} + \|\nabla_vf\|_{L^2_{x,v} (\langle v \rangle^{\gamma + 1 + \ell} m)},
$$
which allows us to conclude.
\end{proof}

\bigskip
%%%%%%%%%%%%%%%%%%%%%%%%%%%%%%%%%%%%%%%%%%%%%%%%%%%%%%%%%%%%%%%%%%%%%%%%%%%%%%%%%%%%%%%%%%%%%%%%%%%%%%%%%%%%%%%%%%%%%%%%
\section{Regularization properties} \label{sec:reg}
\setcounter{equation}{0}
\setcounter{theo}{0}
%%%%%%%%%%%%%%%%%%%%%%%%%%%%%%%%%%%%%%%%%%%%%%%%%%%%%%%%%%%%%%%%%%%%%%%%%%%%%%%%%%%%%%%%%%%%%%%%%%%%%%%%%%%%%%%%%%%%%%%%
This section is devoted to the proof of Theorem~\ref{thm:LIBthm}.
We start by making a few comments on this theorem:
\begin{itemize}[leftmargin=*]
%\item Notice that it is part of the results of~\cite{HTT1*} that $\Lambda$ generates a semigroup in a large class of Lebesgue and Sobolev spaces and in particular in $H^{r,0}_{x,v}(\langle v \rangle^k)$ for $r \in \N$ and {$k >\gamma/2 +3+2(r+1)s$}, so that $f$ introduced in Theorem~\ref{thm:LIBthm} is well-defined (we have $f(t) = \SS_\Lambda(t)f_0$ for any $t\in \R^+$ where $\SS_\Lambda(t)$ is the semigroup generated by $\Lambda$).

\item As already mentioned, the result is not optimal in the sense that there is a loss in weight in our estimates. But we strongly believe that one could obtain a better estimate (concerning the weights) carrying out a more careful study of the operator~$\Lambda$. Indeed, in our proof, we perform a rough splitting of it and we  use Duhamel formula to recover an estimate on the whole semigroup $S_{\Lambda}(t)$. We could have not split the operator and study it completely, that would certainly provides us a better result. However, the proof would be much more complicated and we are here interested in the gain of regularity in terms of derivatives (not in terms of weights) and in getting quantitative estimates in time.

\item Another important fact is that Theorem~\ref{thm:LIBthm} provides a {``primal'' and a ``dual'' result of regularization, roughly speaking, from $L^2$ into $H^s$ and from $H^{-s}$ into~$L^2$. The fact that we develop a primal and a dual result is directly related to the use of this theorem that we make in Subsections~\ref{subsec:reg} and~\ref{subsec:reg2}. We will only present the proof of the dual result into full details, we just explain how to adapt it in the primal case (which is easier to handle) in Section~\ref{sec:general}. }

%\item Finally, we notice that since the derivatives in $x$ commute with the Boltzmann operator (which is local in~$x$), we will only deal with the case $r=0$ in the proofs as we did for the fractional Kolmogorov equation.
\end{itemize}

\medskip
\subsection{Steps of the proof of the main regularization result} \label{subsec:thm:LIBthm}
In this part, we give the main steps of the proof of Theorem~\ref{thm:LIBthm}.

%----------------------------------%----------------------------------%----------------------------------%----------------------------------%----------------------------------%----------------------------------%----------------------------------%----------------------------------%
\smallskip
\subsubsection{Splitting of the operator {for the dual result}} \label{subsec:splitting}
We are going to study the regularization properties only of a part of $\Lambda$, we thus start by splitting it into two parts.
Note that in this paper, we consider two types of splittings to separate grazing and non-grazing collisions cutting the small $\theta$ or the small $|v'-v|$. For our purpose in this part, we will work with the second option which is more adapted to the study of hypoelliptic properties of the linearized Boltzmann operator. We recall that $\widetilde Q_1$ is the operator associated to the kernel:
$$
\widetilde B_1 (v-v_*,\sigma) = \chi(|v'-v|) \, b(\cos \theta) \, |v-v_*|^\gamma
$$
and $\widetilde Q^c_1$ the one associated to the remainder part of the kernel:
$$
\widetilde B^c_1 (v-v_*,\sigma)=(1-\chi(|v'-v|)) \, b(\cos \theta) \, |v-v_*|^\gamma.
$$
We then have:
$$
\begin{aligned}
\Lambda h &= - v \cdot \nabla_x h + \widetilde Q_1 (\mu, h) + \widetilde Q^c_1(\mu,h) + Q(h,\mu) \\
&= \Bigg(- K \langle v \rangle^{\gamma+2s} h- v \cdot \nabla_x h + \int_{\R^3 \times \mathbb{S}^2} \widetilde B_1(v-v_*,\sigma) (\mu_* h' - \mu'_*h)\, \d\sigma \, {\rm d} v_* \Bigg)\\
&\quad+\Bigg(  K \langle v \rangle^{\gamma+2s} h + \int_{\R^3 \times \mathbb{S}^2} \widetilde B_1(v-v_*,\sigma) (\mu'_* - \mu_*)(h'+h) \, \d\sigma \, \dv_*\\
&\hskip 4.5cm+ \int_{\R^3\times \mathbb{S}^2} \widetilde B^c_1(v-v_*,\sigma) (\mu'_* h' - \mu_* h) \, \d\sigma \, \dv_* + Q(h,\mu)\Bigg) \\
&=: \Lambda_1 h + \Lambda_2 h
\end{aligned}
$$
where $K$ is a large positive parameter to be fixed later. Notice that $\Lambda_2$ had already been defined in Subsection~\ref{subsec:estimlin} and recall that Lemma~\ref{lem:lambda_2} tells that this part of the linearized operator do not induce a loss of regularity in terms of derivatives. Note also that in $\Lambda_1$, we have a term which is going to provide us some regularization
$$
\int_{\R^3 \times \mathbb{S}^2} \widetilde B_1(v-v_*,\sigma) (\mu_* h' - \mu'_*h) \, \d\sigma \, \dv_*
$$
and another one which provides us some dissipativity:
$
- K \langle v \rangle^{\gamma+2s} h.
$

\smallskip
\subsubsection{Regularization properties of $\Lambda_1$ {in the dual case}} \label{subsec:regdual}
The main result of this Subsection is Proposition~\ref{prop:duallambda1} and is about the regularization features of the semigroup associated to~$\Lambda_1$. Here, we just state the result and we postpone its proof to Subsection~\ref{subsec:quantization} in which we develop pseudo-differential arguments.

\smallskip
\noindent {\it Functional spaces.} In the remainder part of this section, we consider three weights:
\begin{equation} \label{condweight}
\left\{
\begin{aligned}
&\text{$m(v) = \langle v \rangle^l$ with {$l \ge 0$},}\\
&\text{$m_0(v) = \langle v \rangle^{l_0}$ with {$l_0 >\gamma/2 +3+4s$}}\\
&\text{$m_1(v) = \langle v \rangle^{l_1}$ with {$l_1=l_0 +\gamma + 1 + \ell$ and $\ell>3/2$}.}
\end{aligned}
\right.
\end{equation}
We then denote for $i=\emptyset,0,1$:
$$
\left\{
\begin{aligned}
&F_i=L^2_{x,v}(m_i)\\
&G_i = H^{s,0}_{x,v}({ \langle v \rangle^{\gamma/2}} m_i)\\
&H_i=H^{0,s}_{x,v}( \langle v \rangle^{\gamma/2} m_i)) \cap  L^2_{x,v}( \langle v \rangle^{(\gamma+2s)/2} m_i) \\
&\text{$G'_i$ the dual of $G_i$ w.r.t. $F_i$}\\
&\text{$H'_i$ the dual of $H_i$ w.r.t. $F_i$.}
\end{aligned}
\right.
$$
We also introduce the (almost) flat spaces:
$$
\left\{
\begin{aligned}
&\widetilde F=L^2_{x,v} \\
&\widetilde G=H^{s,0}_{x,v}({ \langle v \rangle^{\gamma/2}}) \\
&\widetilde H=H^{0,s}_{x,v}( \langle v \rangle^{\gamma/2}) \cap  L^2_{x,v}( \langle v \rangle^{(\gamma+2s)/2}) \\
&\text{$\widetilde G'$ the dual of $\widetilde G$ w.r.t. $\widetilde F$}\\
&\text{$\widetilde H'$ the dual of $\widetilde H$ w.r.t. $\widetilde F$.}
\end{aligned}
\right.
$$

\smallskip
\noindent {\it Remark on the dual embeddings.}
First, we notice that
\begin{equation} \label{prop:embed}
\forall \, q_1 \le q_2, \, \varsigma \in \R^+, \quad H^\varsigma_v(\langle v \rangle^{q_2}) \hookrightarrow H^\varsigma_v(\langle v \rangle^{q_1}).
\end{equation}
This property is clear in the case $\varsigma \in \N$. Let us now treat the case $\varsigma \in \R^+ \setminus \N$.
%This case can be shown using a pseudo-differential argument or even more simply, using real interpolation.
Since the weighted space $H^\varsigma_v(\langle v \rangle^{q_i})$ is defined through
$$
h \in H^\varsigma_v(\langle v \rangle^{q_i}) \Leftrightarrow h \langle v \rangle^{q_i} \in H^\varsigma_v
$$
and that we have, using the standard real interpolation notations (recalled in the introduction):
$$
H^\varsigma_v = \left[H^{\lfloor \varsigma \rfloor}_v,H^{\lfloor \varsigma \rfloor+1}_v\right]_{{\varsigma-\lfloor \varsigma \rfloor},2},
$$
one can prove that
$$
H^\varsigma_v(\langle v \rangle^{q_i}) = \left[H^{\lfloor \varsigma \rfloor}_v(\langle v \rangle^{q_i}),H^{\lfloor \varsigma \rfloor+1}_v(\langle v \rangle^{q_i})\right]_{{\varsigma-\lfloor \varsigma \rfloor},2}, \quad i=1,2.
$$
From this, since $H^{\ell}_v(\langle v \rangle^{q_2}) \hookrightarrow H^{\ell}_v(\langle v \rangle^{q_1})$ for $\ell \in \N$, we deduce the desired embedding result:~$
H^\varsigma_v(\langle v \rangle^{q_2}) \hookrightarrow H^\varsigma_v(\langle v \rangle^{q_1}).
$

We can now prove that the standard inclusions for dual spaces do not hold here. Indeed, we have for example $G_1 \subset G_0$ and also $G'_1 \subset G'_0$ (the same for ``$H$-spaces'' hold). This is due to the fact that the pivot spaces are $F_i$ and not $L^2_{x,v}$ as usually. Indeed, using that~$k_1 \ge k_0$ and~\eqref{prop:embed}, we have
\begin{align*}
\| h\|_{G'_0} &= \sup_{\|\varphi\|_{G_0} \le 1} \langle h, \phi \rangle_{F_0} \\
&=  \sup_{\|\varphi m_0\|_{\widetilde G} \le 1} \bigg\langle h m_1, \phi {m_0^2 \over m_1} \bigg\rangle_{\widetilde F} \\
&= \sup_{\|\psi m_1^2/m_0\|_{\widetilde G} \le 1} \langle h m_1, \psi m_1 \rangle_{\widetilde F} \\
&\le  \sup_{\|\psi m_1\|_{\widetilde G} \le 1} \langle h m_1, \psi m_1 \rangle_{\widetilde F} = \sup_{\|\phi\|_{G_1} \le 1} \langle h, \phi \rangle_{F_1} = \|h\|_{G'_1}.
\end{align*}

\medskip
\noindent {\it Reduction of the problem to a ``simpler'' framework.}
We start by explaining how to avoid some difficulties coming from the spaces in which we are working.
First, in order to simplify the problem, since we work in weighted spaces, we are going to ``include'' the weight in our operator.
For this purpose, we define the operator $\Lambda^{m}_1$ by
$$
\Lambda^{m}_1 g := m \, \Lambda_1 ({m}^{-1} g).
$$
We notice that if $h$ satisfies $\partial_t h = \Lambda_1 h$, then $g := {m}h$ satisfies $\partial_t g = \Lambda^{m}_1 g$ and we thus have $S_{\Lambda_1^{m}}(t) g_0 = m S_{\Lambda_1}(t) h_0$ if $g_0=mh_0$.
Then, in order to avoid having to work in dual spaces, we introduce formal dual operators for which we prove regularization properties in ``positive'' Sobolev spaces.
To this end, we introduce the (formal) adjoint operator (w.r.t. the scalar product of $L^2_{x,v}$) of $\Lambda_1^m$ that we denote $\Lambda^{{m},*}_{1}$  and which is defined by:
\begin{equation*} %\label{eq:lambda1m*}
\Lambda_{1}^{{m},*} \phi := \int_{\R^3 \times \mathbb{S}^2} \widetilde B_1(v-v_*,\sigma) \, \mu'_* \, ( \phi' {m}'-\phi {m}) \, \d\sigma \, \dv_*\, {m}^{-1}  - K \langle v \rangle^{\gamma+2s} \, \phi + v \cdot \nabla_x \phi.
\end{equation*}
The advantage of working with this operator is that we can work in flat and positive Sobolev spaces.
We now write our main regularization estimate:

\begin{prop} \label{prop:duallambda1}
For $K$ large enough, we have the following estimates for any $\varphi_0 \in \widetilde F$:
\begin{equation} \label{eq:dual}
\forall \, t \in (0,1], \quad \|S_{\Lambda^{{m},*}_1}(t)\phi_0\|_{\widetilde H}\lesssim {1 \over \sqrt{t}} \|\phi_0\|_{\widetilde F} \quad \text{and} \quad  \|S_{\Lambda^{{m},*}_1}(t)\phi_0\|_{\widetilde G}\lesssim {1 \over t^{1/2+s}} \|\phi_0\|_{\widetilde F}.
\end{equation}
\end{prop}
The proof of Proposition~\ref{prop:duallambda1} is to be compared with the one {developed in the article~\cite{HTT2*} to study regularization properties of the fractional Kolmogorov equation}. Indeed, it is the same strategy of proof: We introduce a functional which is going to be an entropy for our equation for small times. However, it is much more complicated in this case and our approach requires refined pseudo-differential tools, Subsection~\ref{subsec:quantization} is dedicated to its proof. Before that, we explain how to use Proposition~\ref{prop:duallambda1} to get our final result in Theorem~\ref{thm:LIBthm}.

%{Let us also mention that one can easily show that, taking $K$ large enough, $\Lambda_1^m$ and $\Lambda_1^{m,*}$ are dissipative in $L^2_{x,v}$ so that the semigroup $S_{\Lambda^{{m},*}_1}(t)$ is well defined in the previous proposition (it is also true if we work in $H^{r,0}_{x,v}$ with $r \in \N^*$ and then with $r \in \R^*$ by interpolation).}
%

%\subsubsection{Pseudo-differential tools}
%Let us introduce the following weights
%\begin{equation} \label{eq:symbols}
%\begin{aligned}
%\lambda_v (v,\eta)& := (\seq{\eta}^2 + \seq{v\wedge \eta}^2 + \seq{v}^2)^{1/2},\\
%\lambda_x (v,\xi) & := (\seq{\xi}^2+\seq{v\wedge \xi}^2 + \seq{v}^2)^{1/2},\\
%p (v,\eta)&: =  \langle v\rangle ^\gamma \lambda_v^{2s} +K \langle v \rangle^{\gamma+2s}, \\
% q (v,\xi)&:= \langle v\rangle ^\gamma \lambda_x^{2s}+K \langle v \rangle^{\gamma+2s} , \\
%\omega (v,\eta,\xi)&:= {-} \langle v\rangle ^\gamma \lambda_x^{s-1} \lambda_v^{s-1}(\eta\cdot \xi+ (v\wedge \eta) \cdot (v\wedge \xi)).
%\end{aligned}
%\end{equation}
%%where $K$ is as in the definition of $ \Lambda_1$.

%----------------------------------%----------------------------------%----------------------------------%----------------------------------%----------------------------------%----------------------------------%----------------------------------%----------------------------------%
\smallskip
\subsubsection{Proof {of the dual result} of Theorem~\ref{thm:LIBthm}}
The goal is first to prove the dual result in Theorem~\ref{thm:LIBthm} in the case $r=0$. The proof will be exactly the same for other values of~$r$ since the operator $( 1 -\Delta_x)^{r/2}$ commutes with the Boltzmann operator. We can thus apply the result obtained for $r=0$ to $( 1 -\Delta_x)^{r/2} h_0$ to recover the result for $r\neq 0$.

From Proposition~\ref{prop:duallambda1}, we can deduce an estimate on the semigroup associated to $\Lambda_1$ in the ``original'' (non flat) spaces:

\begin{cor}
For $K$ large enough, for any $h_0 \in H'$, resp.~$h_0 \in G'$, there holds:
\begin{equation} \label{eq:lambda_1}
\forall \, t \in (0,1], \quad \|S_{\Lambda_1}(t)h_0\|_{F} \lesssim {1 \over \sqrt{t}} \|h_0\|_{H'}, \quad \text{resp.} \quad \|S_{\Lambda_1}(t)h_0\|_{F} \lesssim {1 \over t^{1/2+s}} \|h_0\|_{G'}.
\end{equation}
\end{cor}

\begin{proof}
Let us consider $K$ large enough so that the conclusion of Proposition~\ref{prop:duallambda1} holds. Using~\eqref{eq:dual} and denoting $g_0=mh_0$, we have for any $t \in (0,1]$:
$$
\begin{aligned}
\|S_{\Lambda_1}(t) h_0\|_{F} &{=} \|S_{\Lambda_1^{m}} (t) g_0 \|_{\widetilde F} = \sup_{\|\phi\|_{\widetilde F} \le 1} \langle S_{\Lambda_1^m}(t) g_0, \phi \rangle = \sup_{\|\phi\|_{\widetilde F} \le 1} \langle g_0, S_{\Lambda_1^{{m},*}}(t)\phi \rangle \\
&\lesssim \sup_{\|\phi\|_{\widetilde F} \le 1} \|g_0\|_{\widetilde H'} \|S_{\Lambda_1^{{m},*}}(t)\phi\|_{\widetilde H} \lesssim  {1 \over \sqrt{t}}\|g_0\|_{\widetilde H'} {=} {1 \over \sqrt{t}} \|h_0\|_{H'}
\end{aligned}
$$
which is exactly the first part of~\eqref{eq:lambda_1}. The second one is proven in the same way.
\end{proof}

Let us finally prove that the regularization properties of $\Lambda_1$ are enough to conclude that the whole operator $\Lambda$ has some good regularization properties: Even if we have a loss of weight in the final estimate, $\Lambda$ inherits regularization properties from $\Lambda_1$ in terms of fractional Sobolev norms.

\begin{lem} \label{lem:lambda_1/lambda}
For any $h_0 \in H'_1$, resp.~$h_0 \in G'_1$, we have:
\begin{equation} \label{eq:lambda}
\forall \, t \in (0,1], \quad \|S_{\Lambda}(t)h_0\|_{F_0} \lesssim {1 \over \sqrt{t}} \|h_0\|_{H'_1}, \quad \text{resp.} \quad \|S_{\Lambda}(t)h_0\|_{F_0} \lesssim {1 \over t^{1/2+s}} \|h_0\|_{G'_1}.
\end{equation}
\end{lem}
%{ Ici, sans le dire, on utilise que $k_0>\gamma/2+3+4s$ pour savoir que $\Lambda$ g\'en\`ere bien un semi-groupe dans $F_0$. Pour cela, on utilise implicitement le Th\'eor\`eme~\ref{theo:extension}. C'est tr\`es p\'enible \`a expliquer car on a l'impression que pour prouver le th\'eo~\ref{theo:extension}, on utilise le th\'eo~\ref{thm:LIBthm}. Mais en fait pour prouver le th\'eo~\ref{theo:extension}, on utilise seulement la prop~\ref{resultHs} de r\'egularisation de $S_{\widetilde \Lambda_1}(t)$ pour les transf\'erer \`a $S_\BB(t)$. Et la seule chose que l'on demande est que $\widetilde \Lambda_1$  g\'en\`ere un semi-groupe (ce qui est, il me semble, ok dans tous les espaces qu'on veut, sans condition sur le poids). Finalement, ce qui cr\'ee la confusion est qu'on veut \'ecrire notre r\'esultat principal de r\'egularisation (Theorem~\ref{thm:LIBthm}) sur le semi-groupe g\'en\'er\'e par $\Lambda$ tout entier. Je trouve que c'est plus propre que d'introduire d\`es l'intro les splittings de $\Lambda$...

%En conclusion, il faudrait faire une phrase intelligente qui r\'esume ce que je viens de dire pour expliquer la condition sur le poids $k_0$, sans en dire trop non plus pour ne pas donner le b\^aton pour se faire battre !}
\begin{proof}
We have:
\begin{equation} \label{eq:lambdaHTT1*}
\forall \, t \in (0,1], \quad \|S_{\Lambda}(t) h_0 \|_{F_0} \lesssim \|h_0\|_{F_0}.
\end{equation}

Then, we write Duhamel formula:
$$
S_{\Lambda}(t) = S_{\Lambda_1}(t) + \int_0^t S_{\Lambda}(\tau) \Lambda_2 S_{\Lambda_1}(t-\tau) \, \d \tau
$$
from which we deduce, combining~\eqref{eq:lambdaHTT1*},~\eqref{eq:lambda_1} and~\eqref{eq:lambda_2} applied with the appropriate weights, that for $t \in (0,1]$,
$$
\begin{aligned}
\|S_{\Lambda}(t)h_0\|_{F_0}
&\lesssim \|S_{\Lambda_1}(t) h_0\|_{F_0} + \int_0^t \|S_{\Lambda} (\tau)\Lambda_2 S_{\Lambda_1}(t-\tau) h_0\|_{F_0} \, \d \tau \\
&\lesssim {1 \over \sqrt{t}} \|h_0\|_{H'_0} + \int_0^t \|\Lambda_2 S_{\Lambda_1}(t-\tau) h_0\|_{F_0} \, \d \tau \\
&\lesssim {1 \over \sqrt{t}} \|h_0\|_{H'_0} + \int_0^t \|S_{\Lambda_1}(t-\tau) h_0\|_{F_1} \, \d \tau \\
&\lesssim {1 \over \sqrt{t}} \|h_0\|_{H'_0} + \int_0^t {1 \over \sqrt{t-\tau}} \|h_0\|_{H'_1} \, \d \tau \\
&\lesssim {1 \over \sqrt{t}} \|h_0\|_{H'_0} + \int_0^1 {1 \over \sqrt{\tau}} \|h_0\|_{H'_1} \, \d \tau \lesssim {1 \over \sqrt{t}} \|h_0\|_{H'_1}.
\end{aligned}
$$
This concludes the proof of the first part of~\eqref{eq:lambda}. Concerning the second one, we proceed as before using that $1/2+s<1$ since $s<1/2$ and we obtain for any $t \in (0,1]$:
$$
\begin{aligned}
\|S_{\Lambda}(t)h_0\|_{F_0}
%&\lesssim \|S_{\Lambda_1}(t) f\|_{F_0} + \int_0^{t} \|S_{\Lambda}(s)\Lambda_2 S_{\Lambda_1}(t-s) f\|_{F_0} \, \d s \\
%&\lesssim {1 \over t^{1/2+s}} \|f\|_{Y'_0} + \int_0^t \|\Lambda_2 S_{\Lambda_1}(t-s) f\|_{F_0} \, \d s \\
%&\lesssim {1 \over t^{1/2+s}} \|f\|_{Y'_0} + \int_0^t \|S_{\Lambda_1}(t-s) f\|_{X_1} \, \d s \\
%&\lesssim {1 \over t^{1/2+s}} \|f\|_{Y'_0} + \int_0^t {1 \over (t-s)^{1/2+s}} \|f\|_{Y'_1} \, \d s \\
%&\lesssim {1 \over t^{1/2+s}} \|f\|_{Y'_0} + \int_0^1 {1 \over s^{1/2+s}} \|f\|_{Y'_1} \, \d s
\lesssim {1 \over t^{1/2+s}} \|h_0\|_{G'_1}.
\end{aligned}
$$
\end{proof}
\begin{rem}
In the previous results, we skipped the proof that the operators we consider generate continuous semigroups. In Proposition~\ref{prop:duallambda1}, the fact that $\Lambda_1^{m,*}$ generates a semigroup in the large space $\widetilde F$ could be either proved directly either using the general strategy of enlargement proposed in~\cite{GMM*}. Similarly, in Lemma~\ref{lem:lambda_1/lambda}, we skipped the proof of the fact that $\Lambda$ also generates a semigroup, let us just note that the conditions on the weights entering in the definitions of the functional spaces in~\eqref{condweight} are here needed to close the enlargement argument.
\end{rem}

%----------------------------------%----------------------------------%----------------------------------%----------------------------------%----------------------------------%----------------------------------%----------------------------------%----------------------------------%
\medskip
\subsection{Pseudodifferential study} \label{subsec:quantization}

The aim of this Subsection is the proof of Proposition~\ref{prop:duallambda1} about the regularization properties of the operator
\begin{equation*} %\label{eq:lambda1m*}
\Lambda_{1}^{{m},*} \phi = \int_{\R^3 \times \mathbb{S}^2} \widetilde B_1(v-v_*,\sigma) \, \mu'_* \, ( \phi' {m}'-\phi {m}) \, \d\sigma \, \dv_*\, {m}^{-1}  - K \langle v \rangle^{\gamma+2s} \, \phi + v \cdot \nabla_x \phi.
\end{equation*}
This will be done with a pseudodifferential version of the Lyapunov trick developed in the fractional Fokker-Planck case {in~\cite{HTT2*}} and special classes of symbols that we recall in the Appendix {\ref{app:pseudo}}.

%----------------------------------%----------------------------------%----------------------------------%----------------------------------%----------------------------------%----------------------------------%----------------------------------%----------------------------------%
\smallskip
\subsubsection{Pseudodifferential formulation of the operator $\Lambda_1^{m,*}$} \label{subsubsec:lambda1m*}

The operator $\Lambda_{1}^{{m},*}$ is very similar to the operator $\mathcal{L}_{1,2,\delta}$ defined in \cite[Proposition 3.1]{AHL*}. We shall thus take advantage of the analysis of the pseudo-differential operator $\mathcal{L}_{1,2,\delta}$ and its symbol in \cite{AHL*}.
If we extract the collision part of the operator $\Lambda_1^{m,*}$ (forgetting the transport one and the addition of the multiplicative term), we obtain
$$
\Lambda_{1}^{{m},*, {\hbox{\footnotesize collision}}} \phi := \int_{\R^3 \times \mathbb{S}^2} \widetilde B_1(v-v_*,\sigma) \, \mu'_* \, ( \phi' {m}'-\phi {m}) \, \d\sigma \, \dv_*\, {m}^{-1}
$$
In the case $m= 1$, this operator is actually the main one studied in~\cite{AHL*}:
$$
\Lambda_{1}^{{1},*, {\hbox{\footnotesize collision}}} = \mathcal{L}_{1,2,1} =: -\tilde{a}_0(v,D_v) ,
$$
where $\tilde{a}_0$ is a real symbol in $(v,\eta)$ defined through
$$
\tilde{a}_0(v,\eta) := \int_{\R^3_\vartheta} {{\d \vartheta}\over |\vartheta|^{3+2s}}  \int_{E_{0,\vartheta}} \d\alpha \, \widetilde{b}(\alpha,\vartheta) \, \mathds{1}_{|\alpha| \ge |\vartheta|} \, \chi(\vartheta) \, \mu(\alpha+v) \, |\alpha+\vartheta|^{\gamma+1+2s} \,
(1-\cos(\eta \cdot \vartheta))
$$
thanks to Carleman representation (see Lemma~\ref{lem:Carleman}).  We  recall below the main result from~\cite{AHL*} concerning the symbol $\tilde{a}_0$ (be careful, this symbol is denoted without tilde there). The notations are those from Appendix~\ref{app:pseudo} where the definitions of objects concerning the pseudo-differential calculus are recalled.

\begin{prop}[Propositions~3.1 and~3.4 in \cite{AHL*}] \label{prop:AHL}
The symbol $\tilde{a}_0$ satisfies the following properties:
\begin{enumerate}[leftmargin=1.1cm]
\item[(i)] $   \tilde{a}_0 \in S( \seq{v}^\gamma (1+ |\eta|^2 + |v\wedge\eta|^2)^s, \Gamma),$
\item[(ii)] $ \forall \,\eps>0, \, \,  \nabla_\eta \tilde{a}_0 \in S( \eps \seq{v}^\gamma (1+ |\eta|^2 + |v\wedge\eta|^2)^s + \eps^{-1} \seq{v}^{\gamma + 2s},\Gamma),$
\item[(iii)]  $\exists \, c>0, \, -c\seq{v}^{\gamma + 2s}+ \seq{v}^\gamma \sep{ 1+ |\eta|^2 + |v\wedge\eta|^2}^s \lesssim \tilde{a}_0 \lesssim \seq{v}^\gamma \sep{ 1+ |\eta|^2 + |v\wedge\eta|^2}^s,$
\end{enumerate}
where $\Gamma:=|\dv|^2+|\d \eta|^2$ is the flat metric.
\end{prop}

For convenience we denote by $a_0$ the Weyl symbol of operator $\tilde{a}_0(v,D_v)$, so that
$$
a_0^w = \tilde{a}_0(v,D_v).
$$
Everywhere in what follows, any symbol with a tilde will refer to a classical quantization, and when no tilde is present, the symbol will refer to the Weyl quantization. Both quantizations are recalled in the beginning of Subsection \ref{pseud} in the Appendix.
 Note that $a_0$ is not real anymore, anyway we shall see later that it conserves good ellipticity properties. Denoting then
$$
a(v,\eta) := \sep{  m^{-1} \sharp {a}_0 \sharp m } (v,\eta)+ K \seq{v}^{\gamma + 2s},
$$
where $\sharp$ denotes the usual Weyl composition and we omit the dependency of $a$ with respect to $K$ in our notation, we have:
$$
\Lambda_{1}^{{m},*} = -a^w + v \cdot \nabla_x.
$$
For sake of simplicity, we introduce the following notation
$$
A := a^w,
$$ so that  the collision part of
operator $\Lambda_{1}^{{m},*}$ writes
$$
 \Lambda_{1}^{{m},*} = -A + v \cdot \nabla_x
$$
(recall that they depend on $K$).
In order to study the symbolic properties of $a$, %and $\tilde{a}$,
we now introduce
the main weights. We pose for $(v,\eta) \in \R^6$
$$
\lambda_v^2(v,\eta) :=  \seq{\eta}^2 + \seq{v\wedge \eta}^2 +  \seq{v}^2
\quad \text{and} \quad
p(v,\eta) := \seq{v}^\gamma \lambda_v^{2s} + K \seq{v}^{\gamma + 2s}
$$
which will be the main reference symbol of our study (note that this symbol is denoted $\tilde{a}_K$ in~\cite{AHL*}). Although $p$ depends on $K$, we will omit in the following any subscript or reference to this dependence. It will be shown in the next subsection that $p$ is a good weight in the sense of Appendix~\ref{app:pseudo}. The following Lemma shows that
$a$ has good properties in the class $S_K(p)$, the main class of symbols whose definition is recalled in full generality in Appendix~\ref{app:pseudo}.

 \begin{lem} \label{aposell} Let $m(v) = \seq{v}^k$ for $k\ge 0$. Then
 uniformly in $K$ sufficiently large, we have that  $\Re a \geq 0$,  $a \in S_K(p)$ and $\Re a$ is elliptic positive in this class. %, and the same is true for $\tilde{a}$.
  \end{lem}
\begin{proof} We shall take profit of the estimates from~\cite{AHL*} recalled above in Propostion~\ref{prop:AHL}. We first note that
because of the symbolic estimates on $\tilde{a}_0$ we can take $\eps= K^{-1/2}$ in $(ii)$ and,  using Lemma \ref{allsk}, we get that
$\tilde{a}_0 \in S_K(p)$ and then ${a}_0 \in S_K(p)$.
Adding $K \seq{v}^{\gamma + 2s}$ does not change the computation and we also get
that
$$
a_0 + K \seq{v}^{\gamma + 2s} \in S_K(p).
$$
Now we can do the conjugation with $m$. We first note that clearly, with the same notations as before, we have
$ m \in S_K(m)$ and $m^{-1} \in S_K(m^{-1})$.
This can be checked directly by noticing that the derivatives of $m$ in $\eta$ are zero. The stability of the class $S_K$ from Lemma~\ref{stab} implies then  that
$$
a  = m^{-1} \sharp  a_0   \sharp  m + K \seq{v}^{\gamma + 2s} = m^{-1} \sharp \sep{ a_0 + K \seq{v}^{\gamma + 2s} } \sharp  m \in S_K( p ).
$$
We can also notice that looking at the main terms in the asymptotic development of the~$\sharp$~product { (see in particular Lemma \ref{allsk} and its proof)},  we have
$$
a =  {a}_0  + K \seq{v}^{\gamma + 2s} + r =  \tilde{a}_0  + K \seq{v}^{\gamma + 2s} + r'
$$
 with  $r$ and $r' \in K^{-1/2} S(p)$ (note that $r$ is exactly  the { Weyl} symbol of $m^{-1} [ a_0^w ,  m ]$).
%\sc {\Rd  est-ce que \c{c}a change quelque-chose avec $\tilde{a}_0$ si c'est bien \c{c}a ?} {\Gn reponse :  ca ne change rien si on met $\tilde{a_0}$,  la place, mais c'est plus pratique sans le tilde parce $a_0$ est rel et donc $a_0 + K \seq{v}^{\gamma + 2s}$ elliptic rel}. \rm
Since from Propostion~\ref{prop:AHL}-$(iii)$, we have   $\tilde{a}_0 + K \seq{v}^{\gamma + 2s} \gtrsim p $ (uniformly in $K$), we get that
 $$
 \Re a \gtrsim p
 $$
 so that $\Re a$ is non-negative and elliptic for $K$ large  (note that this proof is very close to the one of Lemma~\ref{allsk} in Appendix~\ref{app:pseudo}).
%The proof is complete.
\end{proof}

%----------------------------------%----------------------------------%----------------------------------%----------------------------------%----------------------------------%----------------------------------%----------------------------------%----------------------------------%
\smallskip
\subsubsection{Reference weights} \label{subsec:refweights}

We now introduce some weights involving the constant $K$ where~$K$ is a large constant to be defined later. Formally, $1/\sqrt{K}$ plays the role of a small semiclassical parameter.
We recall that for $(v,\eta) \in \R^6$
$$
\lambda_v^2(v,\eta) =  \seq{\eta}^2 + \seq{v\wedge \eta}^2 +  \seq{v}^2
\quad \text{and} \quad
p(v,\eta) = \seq{v}^\gamma \lambda_v^{2s} + K \seq{v}^{\gamma + 2s}.
$$
We shall need their counterparts in the $\xi$ variable (considered as a parameter) instead of~$\eta$ and thus also introduce
$$
\lambda_x^2(v,\eta) :=  \seq{\xi}^2 + \seq{v\wedge \xi}^2 +  \seq{v}^2
$$
and
$$
q(v,\eta) := \seq{v}^\gamma \lambda_x^{2s} + K \seq{v}^{\gamma + 2s},
$$
where we omit the dependance on $K$ and $\xi$ again in the notations. We eventually introduce a mixed symbol
$$
\omega (v,\eta): = { - }\langle v\rangle ^\gamma \lambda_x^{s-1} \lambda_v^{s-1}(\eta\cdot \xi+ (v\wedge \eta) \cdot (v\wedge \xi))
$$
which will be crucial in the analysis.
Following Appendix~\ref{app:pseudo}, we have in particular:
\begin{lem}
The symbols $p$, $q$ and more generally $\seq{v}^\zeta p^\varrho q^\varsigma$ for $\zeta$, $\varrho$ and $\varsigma \in \R$ are temperate with respect to $\Gamma$ uniformly w.r.t. $K$ and $\xi$.
\end{lem}
\begin{proof} These computations are done for e.g. in~\cite[Section 3.3]{AHL*}.
\end{proof}

The symbols $p$, $q$, and $\omega$ are then good symbols w.r.t. these classes, as the  following lemma shows.

\begin{lem} \label{stabsymb} We have $p \in S_K(p)$, $q\in S_K(q)$, $\omega \in S_K(\sqrt{pq})$ and more generally, we also have  $\seq{v}^\zeta p^\varrho q^\varsigma \in S_K(\seq{v}^\zeta p^\varrho q^\varsigma )$ for $\zeta$, $\varrho$ and $\varsigma \in \R$, all this uniformly in $K$ and $\xi$. \end{lem}

\begin{proof}
We only do the proof for $p$, the other being similar. We just have to differentiate the symbol $p$. We study first the gradient with respect to~$\eta$. We notice that
\begin{equation*}
\begin{split}
\nabla_\eta p & = s \seq{v}^\gamma \lambda_v^{2s-2} \nabla_\eta (\lambda_v^2).
\end{split}
\end{equation*}
%we check that
%\begin{equation*}
%\begin{split}
%\abs{\nabla_\eta (\lambda_v^2)} &
% \leq 2 \lambda_v \seq{v}
% \leq 2 K^{-1/2} \lambda_v (K^{1/2})\seq{v}
% \leq 2 K^{-1/2} \lambda_v^2.
%\end{split}
%\end{equation*}
%As a consequence we get that
%\begin{equation*}
%\begin{split}
%\abs{ \nabla_\eta p } & \leq  2s K^{-1/2} \seq{v}^\gamma \lambda_v^{2s}
%\end{split}
%\end{equation*}
%which is the desired result. {\Rd les deux dernires estimations ne me semblent pas claires, je rcris les choses plus clairement pour moi :
We also have that
$$
\abs{\nabla_\eta (\lambda_v^2)}
 \leq 2 \lambda_v \seq{v}
$$
from which we deduce that
\begin{align*}
\abs{ \nabla_\eta p } &\le 2s \seq{v}^{\gamma+1} \lambda_v^{2s-1} \\
&= 2s K^{-1/2} \left(K^{1/2} \seq{v}^{\gamma/2+s}\right) \seq{v}^{\gamma/2+1-s} \lambda_v^{2s-1} \\
&\le 2s K^{-1/2} p^{1/2} \seq{v}^{\gamma/2} \lambda_v^{s} \le 2s K^{-1/2} p
\end{align*}
which is the desired result.
We skip the other similar computations.
\end{proof}

%----------------------------------%----------------------------------%----------------------------------%----------------------------------%----------------------------------%----------------------------------%----------------------------------%----------------------------------%
\smallskip
\subsubsection{Technical lemmas} \label{subsec:technical}

The main idea in the proof of the regularization result in Proposition~\ref{prop:duallambda1} is to
use the positivity preserving property of the Wick quantization.
%In the following we  give four Lemmas in this spirit

In what follows, we state a series of lemmas (from~\ref{lemma1} to~\ref{lemma4}) which are crucial to be able to ``compare'' our operator $A$ with quantizations of the simpler symbols $p$ and $q$ we introduced in the preceding subsection. The following statements are given for sufficiently large and fixed $K$ (see~\cite{AHL*} and Appendix~\ref{app:pseudo}).

\begin{lem}\label{lemma1}
There exists $c_a>0$ such that
$$
2 \Re \sep{ Ah, h}\geq c_a
 \sep{ p ^{\rm Wick}h, h}.$$
\end{lem}

\begin{proof}  We first notice that
$$
 \Re \sep{ Ah, h} = \Re \sep{ a^w  h,h}
= \sep{ ( \Re a)^w  h,h}
$$
thanks to the properties of the Weyl quantization.
%
%{\Rd \sc il me semble qu'avec la d\'efinition de $A$, on a plut\^ot $\Re \sep{ Af, f} =
%\Re \sep{ a^w f, f}$ ?  Ca rejoint une de mes questions prcdentes. Et est-ce qu'on a
%$\Re \sep{ a^w f, f} = \sep{(\Re a)^w f,f}$ ?}\Gn oui et oui, corrig \Bk \rm
%Since $\Re(a)$ is elliptic
%positive from Lemma~\ref{aposell}
%
%
Using \eqref{mainscalar} for $\Re a$, we therefore get that
$$
 \Re \sep{ Ah, h} = \sep{( \Re a)^w h, h} \simeq    \sep{ (\Re a)^\wick h, h} =
\Re  \sep{ a^\wick h, h}.
 $$
% \sc {\Rd j'imagine qu'il y a un $2$ en trop ou en moins ? \Gn enlev (c'etait pas faux avec...) \Rd  et est-ce qu'on a $\Re  \sep{ a^\wick f, f} =   \sep{ (\Re a)^\wick f, f}$ ?} \Gn oui \Bk \rm
%
%
 Moreover, $\Re a \simeq p$ uniformly in $K$ from Lemma~\ref{aposell}. This implies that there exists $c_a >0$ such that $\Re a - c_a p \geq 0$. Using the positivity property of the Wick quantization gives  $\Re (a)^\wick  - c_a p^\wick \geq 0$ in the sense of operators.
 %
% \sc {\Rd Est-ce que ce n'est pas plutt $\Re a$ qui vrifie ces proprits plutt que $a$ lui-mme ? \Gn coquille corrige } \rm
%
  This proves the result.
 \end{proof}
%
%
%\sc {\color{blue} Il  me semble que des fois on utilise elliptic positive et d'autres elliptic positive ! \`a verifier} \Gn j'ai a peu pres vrifi :  je prefere elliptic positive \Bk \rm
%
%

\begin{lem}\label{lemma2}
There exists $c_p>0$ such that
$$
\sep{p^{\rm Wick} A h + A^* p^{\rm Wick} h,h}   \geq c_p \ \sep{(p^2)^{\rm Wick}h,h}.
 $$
\end{lem}

\begin{proof} %\sc \Gn dsol, il y avait des reliquats de l'ancienne version dans cette preuve, j'ai nettoy :-) \Bk \rm
%Denoting $\tilde{a}$ the Weyl symbol of $A$ {\Rd pourquoi introduire $\tilde{a}$, il me semble que le symbole de Weyl de $A$ est $a$ avec nos d\'efinitions ?} {\color{blue} est-ce que ce n'est pas plutot le symbole de Weyl de $a(v,Dv)$ que on veut considerer ici?}
We have from the definition of the Wick quantization (see \eqref{wickN2})
 $$
 p^\wick A + A^* p^\wick = \sep{ (p\star  N ) \sharp  a + \bar{a}  \sharp (p\star  N ) }^w.
 $$
% {\color{blue} Je crois que $G$ est $N$ comme  dans la suite et dans l'appendice, peut-etre il faudrait aussi rappeler que il est d\'efini dans l'appendice)}
Using now Lemma \ref{allsk}, we have that  $p \in S_K(p)$ implies $p\star  N  \in S_K(p)$ and
 %{\Rd $a \in S_K(p)$ implies also $\tilde{a} \in S_K(p)$ ncessaire ? }.
 from the second point in Lemma \ref{stab}, we get that
 $(p\star  N ) \sharp a + \bar{a} \sharp (p\star  N )$ is elliptic, real and positive (from selfadjointness) in $S_K(p^2)$. We therefore get from \eqref{mainscalar} that
 $$
 \sep{\sep{ (p\star N) \sharp a + \bar{a} \sharp (p\star N) }^w h,h} \simeq
 \sep{\sep{ (p\star N) \sharp a + \bar{a} \sharp (p\star N) }^\wick h,h}
$$
Since $(p\star N) \sharp a + \bar{a} \sharp (p\star N) \simeq p^2$ (uniformly in $K$), the positivity properties of the Wick quantization imply the result.
\end{proof}

\begin{lem}\label{lemma3} There exists $c_q>0$  such that
$$
\sep{q^{\rm Wick} A h + A^* q^{\rm Wick} h,h}\geq c_q \sep{(p q)^{\rm Wick} h,h}.
$$
\end{lem}

\begin{proof} The proof is almost the same as the one of Lemma \ref{lemma2}, the main difference being that
the symbol $q$ now depends on a parameter $\xi$,  with respect to which all estimates have to be uniform. We write
 $$
 q^\wick A + A^* q^\wick = \sep{ (q\star N) \sharp {a} + \bar{{a}} \sharp (q\star N) }^w
 $$
 where again ${a}$ denotes the Weyl symbol of $A$.
 We have that $q \in S_K(q)$ uniformly in~$K$ and~$\xi$ and this implies  $q\star N \in S_K(q)$. From  ${a} \in S_K(p)$
 and the second point in Proposition~\ref{stab}, we get that
 $(q\star N) \sharp {a} + \bar{{a}} \sharp (q\star N)$ is elliptic, real and positive  in $S_K(pq)$. Together with~\eqref{mainscalar},  this implies that there exists $c_q>0$ s.t.
 $$
 \sep{\sep{ (q\star N) \sharp {a} + \bar{{a}} \sharp (q\star N) }^w h,h} \simeq
 \sep{\sep{ (q\star N) \sharp {a} + \bar{{a}} \sharp (q\star N) }^\wick h,h} \geq c_q \sep{ (pq)^\wick h, h}
$$
 where the last inequality comes from the positivity properties of the Wick quantization.
 \end{proof}

\begin{lem} \label{lemma3bis} There exist  $c_\omega>0$ such that
$$
\abs{ \sep{\omega^{\rm Wick} A h + A^* \omega^{\rm Wick} h,h}} \leq c_\omega  \sep{ (p^{3/2} q^{1/2})^\wick h,h}.
$$
\end{lem}

\begin{proof}
We begin by denoting $n := p^{3/4} q^{1/4}$ so that $n^2 = p^{3/2} q^{1/2}$. Using Lemma \ref{stabsymb}, we get that $n$ is elliptic positive in $S_K(n)$. Note also
 that
 $$
 \omega^{\rm Wick} A  + A^* \omega^{\rm Wick} = \sep{ (\omega\star N ) \sharp {a} + \bar{{a}} \sharp (\omega\star N ) }^w
 $$
 using the definitions of the Wick quantization and still denoting again ${a}$ the Weyl symbol of operator $A$. From Lemma \ref{stabsymb}, $\omega \in S_K(\sqrt{pq})$ so that $\omega\star N$ is also in $S_K(\sqrt{pq})$ by Lemma~\ref{allsk}. On the other hand, ${a} \in S_K(p)$ and using the stability Proposition \ref{stab}, we therefore get that
 \begin{equation} \label{theta2}
  (\omega\star N ) \sharp {a} + \bar{{a}} \sharp (\omega\star N ) \in S_K(p^{3/2} q^{1/2}) = S_K(n^2).
 \end{equation}
 We then write
 \begin{equation*}
 \begin{split}
 & \abs{ \sep{\omega^{\rm Wick} A h + A^* \omega^{\rm Wick} h,h}} \\
 & = \bigg| \bigg(
    \underbrace{  (n^{-1})^\wick \sep{ (\omega\star N ) \sharp {a} + \bar{{a}} \sharp (\omega\star N )}^w
    (n^{-1})^\wick}_{ \textrm{ Operator } \Omega } ((n^{-1})^\wick)^{-1} h, ((n^{-1})^\wick)^{-1} h \bigg) \bigg|.
    \end{split}
    \end{equation*}
Let us prove that operator $\Omega$ is bounded. For this, we first note that
  $(n^{-1})^\wick = (n^{-1} \star N)^w$  and recall that $n$ is elliptic positive. Lemma \ref{stab} implies that $n^{-1}$ is positive elliptic in~$S_K(n^{-1})$ too and from Lemma \ref{allsk}, the same is true for $n^{-1} \star N$. The Weyl symbol of~$\Omega$ can be written
  $$
  \hbox{symb}(\Omega) = (n^{-1} \star N) {\sharp} \sep{ (\omega\star N ) \sharp {a} + \bar{{a}} \sharp (\omega\star N )} {\sharp}(n^{-1} \star N)
  $$
    and from the stability Lemma \ref{stab} and \eqref{theta2}, this symbol is in  $S_K(1)$. In particular, the operator $\Omega$ is bounded on $L^2$. We have that
    \begin{equation} \label{secondcomplique}
    \begin{split}
    \left|\sep{\Omega ((n^{-1})^\wick)^{-1}h,((n^{-1})^\wick)^{-1}h}\right| & \leq C \norm{  \big((n^{-1})^\wick\big)^{-1} h}^2 \\
    &\leq C \norm{  n^\wick h}^2 \leq C \sep{ (n^2)^\wick h,h}.
 \end{split}
 \end{equation}
 The first inequality comes from the fact that $\Omega$ is bounded. The last inequality is just a consequence of \eqref{mainnorm}. Let us precise the arguments used for proving the second inequality:
  we have
   \begin{equation} \label{trucdebut}
    \begin{split}
     \norm{  \big((n^{-1})^\wick\big)^{-1} h}^2  =  \norm{  ((n^{-1}\star N)^w)^{-1} h}^2    \simeq \norm{  ((n^{-1}\star N)^{-1})^w h}^2
      \end{split}
      \end{equation}
      using the definition of the Wick quantization and  \eqref{maingamma}. We also check by direct computation that $(n^{-1}\star N)^{-1}$ is elliptic positive in
      in $S_K(n)$ using Lemmas \ref{allsk} (see also Remark~\ref{allskrem} and Lemma~\ref{stab}-$(ii)$. This implies by \eqref{mainequiv} applied with $\tau = (n^{-1}\star N)^{-1}$ that
      \begin{equation} \label{trucmilieu}
      \norm{  ((n^{-1}\star N)^{-1})^w h}^2  \simeq  \norm{  n^w h}^2,
      \end{equation}
      and we get then by \eqref{mainnorm}
      \begin{equation} \label{trucfin}
      \norm{  n^w h}^2 \simeq \sep{ (n^2)^\wick h,h}.
      \end{equation}
      The estimates (\ref{trucdebut})-(\ref{trucfin}) yield the second inequality in \eqref{secondcomplique}.
\end{proof}
      To conclude this subsection, we state a lemma which will be useful in the sequel, and whose proof is direct using positivity properties of the Wick quantization.
\begin{lem} \label{lemma4}
We have the following estimates:
 $$
  \sep{(\langle v\rangle^{2\gamma}\lambda_v^{4s})^{\rm Wick} h,h} \le  \sep{(p^2)^{\rm Wick} h,h} \leq 2(1+K^2)   \sep{(\langle v\rangle^{2\gamma}\lambda_v^{4s})^{\rm Wick} h,h} ,
 $$
 $$
 \sep{p^{\rm Wick} h,h}= \sep{(\langle v\rangle^{\gamma}\lambda_v^{2s})^{\rm Wick} h,h}+K\sep{(\langle v\rangle^{\gamma+2s})^{\rm Wick} h,h},
 $$
  $$
 \sep{(\langle v\rangle^{2\gamma}\lambda_v^{2s}\lambda_x^{2s})^{\rm Wick} h,h}\le  \sep{(pq)^{\rm Wick} h,h} \le (1+K)^2 \sep{(\langle v\rangle^{2\gamma}\lambda_v^{2s}\lambda_x^{2s})^{\rm Wick} h,h}.
 $$
% J'ai rajout d'autres ingalits qu'on utilise pour pouvoir crire tout en terme de $\lambda_v$, $\lambda_x$ etc... dans la preuve sur la fonctionnelle de Lyapunov. Ce n'est pas toujours indispensable mais c'est peut-tre plus clair ?
\end{lem}
\medskip

%----------------------------------%----------------------------------%----------------------------------%----------------------------------%----------------------------------%----------------------------------%----------------------------------%----------------------------------%
\smallskip
\subsubsection{The Lyapunov functional} %functional and proof of Proposition~\ref{prop:duallambda1}}
\label{subsec:propduallambda1}
%We now come to the main part of the proof of Proposition~\ref{prop:duallambda1}.
From now on, we fix once and for all the constant~$K$ so that the conclusions of Lemmas \ref{lemma1} to \ref{lemma4} are true.
We build below a Lyapunov functional corresponding to the following equation
$$
\partial_t \varphi = v \cdot \nabla_x \varphi - A \varphi,
$$
and we consider  $\varphi$ a solution.
Then, since $A$ acts only on the velocity variable, we can take the Fourier transform of our equation in $x \in \T^3$ and see the associated Fourier variable~$\xi\in \Z^3$ as a parameter in our equation. We thus consider $\psi=\mathcal{F}_x \varphi$ to be a solution of
$$
\partial_t \psi { -} i v\cdot \xi \psi + A \psi =0
$$
 with initial data $\psi_0$. We introduce an adapted entropy
functional defined  for all $t\geq 0$ by
\begin{equation} \label{hhh1}
\HH(t) := C\norm{\psi}^2 + D t \sep{p ^{\rm Wick} \psi,\psi}+ E t^{1+s}  \sep{ \omega^{\rm Wick} \psi,  \psi} + t^{1+2s} \sep{q^{\rm Wick} \psi,\psi}
\end{equation}
for large constants $C$, $D$, $E$ to be chosen later, where $\norm{\cdot}$ is the usual $L^2_{x,v}$ norm and $\sep{ \cdot,\cdot}$ is the usual (complex) $L^2_{x,v}$ scalar product.

\begin{lem} \label{hh1bis}
If $E \leq \sqrt{D}$ then for all  $t\geq 0$, we have $\HH(t) \geq 0$. Precisely, we have
\begin{equation*}
0 \leq  C\norm{\psi}^2 + \frac{D}{2} t \sep{p ^{\rm Wick} \psi,\psi}  + \frac{1}{2} t^{1+2s} \sep{q^{\rm Wick} \psi,\psi}
\leq \HH(t).% \leq  C\norm{f}^2 + \frac{3}{2} D t \norm{\Lambda_v^{s-1} \D_v f}^2 + \frac{3}{2} t^{1+2s} \norm{ \Lambda_x^{s-1} \D_x f}^2
\end{equation*}
\end{lem}
\begin{proof} The first part of the inequality comes from the positivity property~\eqref{wick2}. For the bound on $\mathcal{H}(t)$, we start by noticing that using Cauchy-Schwarz inequality:
$$
|\eta\cdot \xi+ (v\wedge \eta) \cdot (v\wedge \xi)| \le \lambda_x \lambda_v.
$$
Then,  the time-dependent Cauchy-Schwarz inequality gives
$$
- Et^{s} \langle v\rangle ^\gamma \lambda_x^{s-1} \lambda_v^{s-1}(\eta\cdot \xi+ (v\wedge \eta) \cdot (v\wedge \xi))
\le \frac{E^2}{2} \langle v \rangle^\gamma \lambda_v^{2s} + \frac{1}{2} t^{2s} \langle v \rangle^{\gamma} \lambda_x^{2s}.
$$
The positivity of the Wick quantization and the fact that $E^2 \le D$  imply that
$$
E t^{1+s}  \sep{ \omega^{\rm Wick} \psi,  \psi} \ge -{D \over 2} t \sep{p ^{\rm Wick} \psi,\psi} - {1 \over 2} t^{1+2s} \sep{q^{\rm Wick} \psi,\psi}
$$
which proves the statement.
\end{proof}

We now show that $\HH$ is indeed a Lyapunov function (entropy functional).

\begin{lem} \label{derivB}  For well chosen (arbitrarily large) constants $C$, $D$ and $E$, we have
$$
\ddt \HH(t) \leq 0, \quad \forall \, t \in (0,1].
$$
\end{lem}

\begin{proof}
Let us define
$$
\mathcal P := p ^{\rm Wick} A + A^* p ^{\rm Wick},\quad \mathcal {\it\Omega} := \omega^{\rm Wick} A + A^* \omega^{\rm Wick}, \quad  \mathcal Q := q^{\rm Wick} A + A^* q^{\rm Wick}.
$$
Then, we have
\begin{equation}\label{C}
\ddt C \norm{\psi}^2= - 2C \, \Re \sep{ A\psi, \psi},
\end{equation}

\begin{equation}\label{D}
\ddt \left(D t  \sep{p ^{\rm Wick} \psi,\psi}\right)= D\sep{p^{\rm Wick} \psi,\psi} -   Dt \sep{\mathcal P \psi,\psi}{ +}D t \sep {\{p,v\cdot \xi\}^{\rm Wick} \psi,\psi},
\end{equation}

\begin{equation}\label{E}
\begin{aligned}
&\ddt\left( E t^{1+s} \sep{ \omega^{\rm Wick} \psi,  \psi} \right) \\
&\qquad =  (1+s)Et^{s}\sep{\omega^{\rm Wick} \psi,\psi} - Et ^{1+s}\sep{{\it \Omega} \psi,\psi}{ +} Et^{1+s}\sep{\{\omega,v\cdot \xi\}^{\rm Wick} \psi,\psi},
\end{aligned}
\end{equation}

\begin{equation}\label{F}
\begin{aligned}
&\ddt \left( t^{1+2s} \sep{ q^{\rm Wick} \psi,  \psi} \right)\\
&\qquad =(1+2s) t^{2s}\sep{q^{\rm Wick} \psi,\psi} - t^{1+2s} \sep{\mathcal Q \psi,\psi}{ +} t^{1+2s} \sep {\{q,v\cdot \xi\}^{\rm Wick} \psi,\psi},
\end{aligned}
\end{equation}
where, in the first term we used the skew-adjointness of the transport operator and in the last term of \eqref{D}, \eqref{E}, \eqref{F}, we used \eqref{wick}.

The right hand side in \eqref{C} is non-positive (thanks to the property of positivity of the Wick quantization~\eqref{wick2}) and using Lemma \ref{lemma1} and Lemma \ref{lemma4}, it can be estimated as
\begin{align*}
-2 C\Re \sep{ A\psi, \psi}&\leq - c_a C
 \sep{ p^{\rm Wick} \psi, \psi}\\ &\leq -\underbrace{c_a C\sep{(\langle v\rangle^{\gamma}\lambda_v^{2s})^{\rm Wick} \psi,\psi}}_{I}-\underbrace{ c_a CK\sep{(\langle v\rangle^{\gamma+2s})^{\rm Wick} \psi,\psi}}_{II}.
 \end{align*}

 Analogously, we can  deduce a bound for the first term in \eqref{D}. Indeed, we recover two non-negative terms
 $$
 D\sep{ p^{\rm Wick} \psi, \psi}\leq \underbrace{D\sep{(\langle v\rangle^{\gamma}\lambda_v^{2s})^{\rm Wick} \psi,\psi}}_{i}+\underbrace{D K\sep{(\langle v\rangle^{\gamma+2s})^{\rm Wick} \psi,\psi}}_{ii}.
 $$
 Moreover,  using the positivity of the Wick quantization~\eqref{wick2}, the second  term in \eqref{D} is non-positive and, using Lemma \ref{lemma2} and  Lemma \ref{lemma4}, it can be estimated as
 $$
 -Dt\sep{\mathcal P  \psi,\psi}\leq -c_p D t\sep{(p ^2)^{\rm Wick} \psi,\psi} \leq -\underbrace{c_p D t \sep{(\langle v\rangle^{2\gamma}\lambda_v^{4s})^{\rm Wick} \psi,\psi}}_{III}.
 $$
Concerning  the third term in  \eqref{D},  let us compute $ \{p,v\cdot \xi\}$:
 \begin{align*}
 \{p,v\cdot \xi\}&= \nabla_\eta p \cdot \nabla_v (v\cdot \xi) - \nabla_v p\cdot \nabla_\eta(v\cdot \xi)= \langle v \rangle^\gamma (\nabla_\eta \lambda_v^{2s}) \cdot  \xi\\& = 2s \langle v  \rangle^\gamma  \lambda_v^{2s-2} (\eta\cdot\xi +(v\wedge \eta)\cdot(v\wedge \xi))\\
 &\leq 2 s  \langle v  \rangle^\gamma \lambda_x \lambda_v^{2s-1}
 ,\end{align*}
 where we used  the fact that $|\eta\cdot\xi +(v\wedge \eta)\cdot(v\wedge \xi)|\leq  \lambda_x \lambda_v$.
Hence,  for any $\varepsilon_{1}>0$, we obtain two non-negative terms
\begin{align*}
&Dt\sep {\{p,v\cdot \xi\}^{\rm Wick} \psi,\psi} \\
&\qquad \leq \underbrace{ 2s {\varepsilon_{1}}^{-1} D\sep{(\langle v  \rangle^\gamma\lambda_v^{2s})^{\rm Wick} \psi,\psi }}_{iii} +\underbrace{2s\varepsilon_{1}^s
D t^{1+s} \sep{ (\langle v  \rangle^\gamma\lambda_x^{s+1} \lambda_v^{s-1})^{\rm Wick} \psi,\psi}}_{iv}.
\end{align*}

Let us now consider \eqref{E}. Using the fact that $\omega \leq \langle v\rangle^\gamma \lambda_x^s \lambda_v^s $, we can bound the first term in \eqref{E}, for any $\varepsilon_{2}>0$, with two non-negative terms
 $$
E t^s \sep{\omega^{\rm Wick} \psi,\psi}  \leq \underbrace{\varepsilon_{2}^{-1}E \sep{ (\langle v\rangle^\gamma \lambda_v^{2s} )^{\rm Wick}\psi,\psi}}_{v}+ \underbrace{\varepsilon_{2} ^{1 /s} E t^{1+s} \sep{(\langle v\rangle^\gamma \lambda_x^{s+1} \lambda_v^{s-1} )^{\rm Wick}\psi,\psi}}_{vi}.
 $$
For the second term in \eqref{E},   Lemma \ref{lemma3bis} implies
 $$
 \sep{\mathcal {\it \Omega} \psi,\psi}\leq c_\omega  \sep{ (p^{3/2} q^{1/2})^\wick \psi, \psi}
 $$
 %Hence, computing
%\begin{align*}
%\omega p&=\langle v\rangle^{2\gamma} \lambda_x^{s-1}\lambda_v^{3s-1}(\eta\cdot \xi + (v\wedge\eta)\cdot( v\wedge\xi))\\
%&\quad + K \langle v\rangle^{2\gamma+2s} \lambda_x^{s-1}\lambda_v^{s-1}(\eta\cdot \xi + (v\wedge\eta)\cdot( v\wedge\xi)),
% \end{align*}
% we obtain
% $$
% \omega p\geq - \langle v\rangle^{2\gamma} \lambda_x^{s}\lambda_v^{3s}- K \langle v\rangle^{2\gamma+2s} \lambda_x^{s}\lambda_v^{s}.
% $$
%Since $ \langle v\rangle^{2s} \leq\lambda_v^{2s} $, there holds
% $$
% \langle v\rangle^{2\gamma+2s} \lambda_x^{s}\lambda_v^{s}\leq  \langle v\rangle^{2\gamma} \lambda_x^{s}\lambda_v^{3s}
% $$
% and, for any $\varepsilon_{3}>0$, we have
% $$
% t^{1+s}\langle v\rangle^{2\gamma} \lambda_x^{s}\lambda_v^{3s}\leq \varepsilon_{3}^{-1} t \langle v\rangle^{2\gamma}\lambda_v^{4s}+\varepsilon_{3}t^{1+2s}\langle v\rangle^{2\gamma} \lambda_x^{2s}\lambda_v^{2s}.
% $$
%Therefore, we can bound  the second term in \eqref{E}, for any $\varepsilon_{3}>0$, by
%\begin{align*}
% &-Et ^{s+1}\sep{\mathcal M \psi,\psi}\\
% &\qquad \leq  \underbrace{c_\omega\varepsilon_{3}^{-1}E  (K+1) t  \sep{(\langle v\rangle^{2\gamma}\lambda_v^{4s})^{\rm Wick} \psi,\psi }}_{vii}+ \underbrace{c_\omega\varepsilon_{3}E { (K+1)} t^{1+2s} \sep{ (\langle v\rangle^{2\gamma} \lambda_x^{2s}\lambda_v^{2s})^{\rm Wick}\psi,\psi }}_{viii},
% \end{align*}
and, for any $\varepsilon_{3}>0$, we have
 $$
 t^{1+s}p^{3/2} q^{1/2}\leq \varepsilon_{3}^{-1} t p^2 +\varepsilon_{3}t^{1+2s}pq.
 $$
Therefore, we can bound  the second term in \eqref{E}, using Lemma~\ref{lemma4}, for any $\varepsilon_{3}>0$, by
\begin{align*}
 &-Et ^{s+1}\sep{{\it \Omega} \psi,\psi}\\
 &\qquad \leq  c_\omega\varepsilon_{3}^{-1}E t  \sep{(p^2)^{\rm Wick} \psi,\psi }+ c_\omega\varepsilon_{3}E  t^{1+2s} \sep{ (pq)^{\rm Wick}\psi,\psi }\\
 &\qquad \leq \underbrace{2(1+K^2)c_\omega\varepsilon_{3}^{-1}E t  \sep{(\langle v \rangle^{2\gamma} \lambda_v^{4s})^{\rm Wick} \psi,\psi }}_{vii}\\
 &\qquad \quad + \underbrace{(1+K)^2c_\omega\varepsilon_{3}E  t^{1+2s} \sep{ (\langle v \rangle^{2\gamma} \lambda_v^{2s} \lambda_x^{2s})^{\rm Wick}\psi,\psi }}_{viii}
 \end{align*}
where $(vii)$ and $(viii)$ are non-negative.

\noindent Let us now observe that
$$
(\nabla_\eta \lambda_v^2)\cdot \xi=2 (\eta\cdot \xi+ (v\wedge \eta) \cdot( v\wedge \xi)),
$$
and
$$
\nabla_\eta(\eta\cdot \xi
 + (v\wedge \eta) \cdot (v\wedge \xi))\cdot \xi= \lambda_x^2- \langle v\rangle^2.
 $$
 We then compute
 \begin{align*}
&  \{\omega,v\cdot \xi\} \\
 &\quad=  \nabla_\eta \omega \cdot \nabla_v (v\cdot \xi) - \nabla_v \omega\cdot \nabla_\eta(v\cdot \xi)= \nabla_\eta \omega \cdot  \xi\\
 &\quad = - \langle v\rangle^{\gamma}\lambda_x^{s-1}\lambda_v^{s-1}\nabla_\eta(\eta\cdot \xi
 + (v\wedge \eta) \cdot (v\wedge \xi))\cdot \xi\\
 &\qquad -\langle v\rangle^{\gamma}\lambda_x^{s-1}(\eta\cdot \xi+ (v\wedge \eta) \cdot (v\wedge \xi))(\nabla_\eta\lambda_v^{s-1})\cdot \xi \\
 &\quad = - \langle v\rangle^{\gamma}\lambda_x^{s+1}\lambda_v^{s-1}
 +\langle v\rangle^{\gamma+2}\lambda_x^{s-1}\lambda_v^{s-1}
 - (s-1)\langle v\rangle^{\gamma}\lambda_x^{s-1}\lambda_v^{s-3}(\eta\cdot \xi + (v\wedge\eta)\cdot( v\wedge\xi))^2.
\end{align*}
In the last expression of $\{\omega,v\cdot \xi\}$, we first notice that since $s-1<0$ and $\min(\lambda_x,\lambda_v) \ge \langle v \rangle$, the second term is bounded as follows:
$$
\langle v\rangle^{\gamma+2}\lambda_x^{s-1}\lambda_v^{s-1} \le \langle v \rangle^{\gamma+2s}.
$$
Gathering the first and third terms, we use Cauchy-Schwarz inequality and  $s<1$ to find:
\begin{align*}
&- \langle v\rangle^{\gamma}\lambda_x^{s+1}\lambda_v^{s-1}
-(s-1)\langle v\rangle^{\gamma}\lambda_x^{s-1}\lambda_v^{s-3}(\eta\cdot \xi + (v\wedge\eta)\cdot( v\wedge\xi))^2 \\
&\quad \le - \langle v\rangle^{\gamma}\lambda_x^{s+1}\lambda_v^{s-1}
+ (1-s) \langle v\rangle^{\gamma}\lambda_x^{s-1}\lambda_v^{s-3} (\lambda_x^2-\langle v \rangle^2) (|\eta|^2 +|v \wedge \eta|^2) \\
&\quad =  - \langle v\rangle^{\gamma}\lambda_x^{s+1}\lambda_v^{s-1} + (1-s)  \langle v\rangle^{\gamma}\lambda_x^{s+1}\lambda_v^{s-3} (|\eta|^2 +|v \wedge \eta|^2) \\
&\qquad- (1-s) \langle v\rangle^{\gamma+2}\lambda_x^{s-1}\lambda_v^{s-3}(|\eta|^2 +|v \wedge \eta|^2) \\
&\quad \le  - \langle v\rangle^{\gamma}\lambda_x^{s+1}\lambda_v^{s-1} + (1-s)  \langle v\rangle^{\gamma}\lambda_x^{s+1}\lambda_v^{s-1} -
(1-s)  \langle v\rangle^{\gamma+2}\lambda_x^{s+1}\lambda_v^{s-3} \\
&\quad \le -s  \langle v\rangle^{\gamma}\lambda_x^{s+1}\lambda_v^{s-1}.
\end{align*}
Thus we have:
$$
 \{\omega,v\cdot \xi\}\le -s \langle v\rangle^{\gamma}\lambda_x^{s+1}\lambda_v^{s-1}+ \langle v\rangle^{\gamma+2s  }.
$$
Hence, the third term in  \eqref{E} can be estimated as
\begin{align*}
&E t^{s+1}\sep{\{\omega,v\cdot \xi\}^{\rm Wick} \psi,\psi} \\
&\quad \leq -\underbrace{ s Et^{s+1} \sep{(\langle v\rangle^{\gamma}\lambda_x^{s+1}\lambda_v^{s-1})^{\rm Wick} \psi,\psi})}_{IV}+\underbrace{ E t^{s+1} \sep{(\langle v\rangle^{\gamma+2s} )^{\rm Wick}\psi,\psi}}_{ix},
\end{align*}
  where $(-IV)$ is non-positive and $(ix)$ is non-negative.

  It remains to consider \eqref{F}. Observing that, for any $\varepsilon_4>0$,
$$
t^{2s}\langle v\rangle^\gamma \lambda_x^{2s}\leq \varepsilon_4 ^{-1} \langle v\rangle^\gamma \lambda_v^{2s} + \varepsilon_4^{\frac{1-s}{2s}}t^{1+s}\langle v\rangle^\gamma \lambda_v^{s-1}\lambda_x^{s+1},
$$
we have that the first term in \eqref{F} can be bounded for any $\varepsilon_4>0$, by
 \begin{align*}
 &(1+2s) t^{2s}\sep{q^{\rm Wick} \psi,\psi} \\
 &\quad \leq \underbrace{(1+2s)\varepsilon_4 ^{-1} \sep{(\langle v\rangle^\gamma \lambda_v^{2s})^{\rm Wick} \psi,\psi }}_{x} +\underbrace{(1+2s)\varepsilon_4^{\frac{1-s}{2s}}  t^{1+s}\sep{(\langle v\rangle^\gamma \lambda_v^{s-1}\lambda_x^{s+1}) ^{\rm Wick}\psi,\psi}}_{xi}\\
 &\qquad + \underbrace{K (1+2s)t^{2s}\sep{(\langle v\rangle ^{\gamma+2s})^{\rm Wick} \psi,\psi}}_{xii}
 \end{align*}
  where $(x),(xi),(xii)$ are non-negative terms.

 \noindent Moreover, using Lemma  \ref{lemma3} and Lemma \ref{lemma4}, the second term in \eqref{F} can be estimated as
 \begin{align*}
-t^{1+2s}\sep{\mathcal Q \psi,\psi}&\leq -c_q t^{1+2s}\sep{(pq)^{\rm Wick} \psi,\psi}
\le - \underbrace{c_q t^{1+2s}\sep{(\langle v \rangle^{2\gamma} \lambda_v^{2s} \lambda_x^{2s})^{\rm Wick} \psi,\psi}}_{V}
%+c_q t^{1+2s}\sep{(\langle v\rangle^{\gamma+2s} )^{\rm Wick}\psi,\psi}\\
%&\leq -\underbrace{c_q t^{1+2s}\sep{(\langle v\rangle^{2\gamma}\lambda_v^{2s}\lambda_x^{2s})^{\rm Wick} \psi,\psi}}_{V} +\underbrace{c_q t^{1+2s}\sep{(\langle v\rangle^{\gamma+2s} )^{\rm Wick}\psi,\psi}}_{xiii},
 \end{align*}
 where $(-V)$ is non-positive.
Finally, since $q$ does not depend on~$\eta$,  we deduce that the Poisson bracket~$\{q,v\cdot \xi\}$ vanishes, hence the third term in \eqref{F} is null.

We now show that with a good choice of the constants $C,D$ and $E$ the sum of the terms in \eqref{C}, \eqref{D}, \eqref{E} and \eqref{F} is non-positive. Indeed, we have to choose $C,D$ and $E$ so that:
 \begin{align*}
 - I + i+iii+v+ x &\leq -\frac{1}{10}I, \\
 -II + ii+ ix+ xii %+ xiii
                    &\leq -\frac{1}{10} II, \\
 -III + vii &\leq -\frac{1}{10}III,\\
 -IV + iv+vi+xi  &\leq -\frac{1}{10}IV,\\
 -V +viii &\leq -\frac{1}{10}V.
 \end{align*}
Restricting the study to  $t\in(0,1]$, and thanks to the fundamental posivity preserving property \eqref{wick2} of the Wick quantization,  the above conditions are satisfied if
\begin{align*}
 D  +  2s {\varepsilon_{1}}^{-1} D+ \varepsilon_{2}^{-1}E +(1+2s)\varepsilon_4 ^{-1}& \leq \frac{9}{10} c_A C,\\
 D K+ E   + K (1+2s) %+ c_q
& \leq \frac{9}{10} c_A C K,\\
2(1+K^2)c_\omega \varepsilon_{3}^{-1}E  %(K+1)
&\leq \frac{9}{10} c_p D ,\\
2s\varepsilon_{1}^s D + \varepsilon_{2} ^{1 /s} E +(1+2s)\varepsilon_4^{\frac{1-s}{2s}} & \leq \frac{9}{10} s E,\\
(1+K)^2c_\omega \varepsilon_{3}E &\leq \frac{9}{10} c_q .
\end{align*}
The above are satisfied  if the constants $C,D,E$ and $\varepsilon_1$, $\varepsilon_2$, $\varepsilon_3$ and $\varepsilon_4$ verify
\begin{align*}
 D \leq \frac{1}{10} c_A C,  \quad    2s {\varepsilon_{1}}^{-1} D \leq \frac{1}{10} c_A C,  \quad \varepsilon_{2}^{-1}E\leq \frac{1}{10} c_A C,  \quad (1+2s)\varepsilon_4 ^{-1} \leq \frac{1}{10} c_A C,\\
  E \leq  \frac{1}{10} c_A C K,   \quad (1+2s)\leq  \frac{1}{10} c_A C, \\
   %\quad c_\mathcal Q  \leq \frac{1}{10} c_A C K,\\
c_\omega \varepsilon_{3}^{-1}E  %(K+1)
                              \leq \frac{1}{10} c_p D ,\\
2s\varepsilon_{1}^s D  \leq \frac{1}{10} s E,\quad  \varepsilon_{2} ^{1 /s}
\leq \frac{1}{10} s, \quad  (1+2s)\varepsilon_4^{\frac{1-s}{2s}}  \leq \frac{1}{10} s E,\\
c_\omega \varepsilon_{3}E \leq \frac{1}{10} c_q.
\end{align*}
This is possible by choosing first $E$, then $\eps_4$, $\eps_3$ and $\eps_2$ small enough,  then $D$ large enough, then $\eps_1$ small enough and finally $C$ as large as needed. Once this choice is done we get that
\begin{equation*}
{\d\over \d t} \mathcal{H}(t) \leq -\frac{1}{10} (I + II + III + IV+ V)  \leq 0
\end{equation*}
and the proof is complete.
%We conclude the proof as we did for Theorem~\ref{FKEthm}, checking that we can choose (in order  of reverse appearance) the constants $C$, $D$, $E$ and the small constants
%$\eps_{j}$, $j=1,\dots,4$ such that for $t \in (0,1]$,
% \begin{align*}
% - I + i+iii+v+ x &\leq -\frac{1}{10}I, \\
% -II + ii+ ix+ xii %+ xiii
%                    &\leq -\frac{1}{10} II, \\
% -III + vii &\leq -\frac{1}{10}III,\\
% -IV + iv+vi+xi  &\leq -\frac{1}{10}IV,\\
% -V +viii &\leq -\frac{1}{10}V.
% \end{align*}
Note that $D$ and  $C$  can be taken arbitrarily large at the end of this procedure. %We obtain therefore that
%\begin{equation*}
%\ddt \hhh(t) \leq -\frac{1}{10} (I + II + III + IV+ V)  \leq 0
%\end{equation*}
%This ends the proof.
\end{proof}

%----------------------------------%----------------------------------%----------------------------------%----------------------------------%----------------------------------%----------------------------------%----------------------------------%----------------------------------%
\smallskip
\subsubsection{Proof of Proposition~\ref{prop:duallambda1}}

We can now prove Proposition \ref{prop:duallambda1}. Consider $\varphi$ the solution of
$$
\partial_t \varphi = v \cdot \nabla_x \varphi - A \varphi,
$$
with initial data $\varphi_0$
and $\psi=\mathcal{F}_x \varphi$ to be the  solution of
$$
\partial_t \psi { -} i v\cdot \xi \psi + A \psi =0
$$
 with initial data $\psi_0= \mathcal{F}_x \varphi_0$. From Lemma \ref{derivB}, we know that
 $$
 \HH(t) \leq \HH(0) = C \norm{\psi_0}^2,
 $$
 and using Lemma \ref{hh1bis}, this gives for all $t\in (0,1]$
 \begin{equation} \label{majorpq}
 \sep{p ^{\rm Wick} \psi,\psi} \leq \frac{2C}{D} \frac{1}{t} \norm{\psi_0}^2 \qquad \textrm{ and }  \qquad \sep{q^{\rm Wick} \psi,\psi} \leq  \frac{2C}{t^{1+2s}} \norm{\psi_0}^2,
 \end{equation}
 where we used the fact that both left members are non-negative according to Proposition~\ref{mainpseudo}. Working in the class $S_K(p)$ again, gives through Proposition~\ref{mainpseudo} and Lemma~\ref{AHLinverse} (see there the definition of $H_R$)
 \begin{equation*}
 \begin{split}
 \norm{ \seq{v}^{\gamma/2} \seq{D_v}^s \psi}^2
& =  \norm{ \seq{v}^{\gamma/2} \seq{D_v}^s ({ (p^{1/2})}^w)^{-1} { (p^{1/2})}^w \psi}^2 \\
& = \Big\| \underbrace{\seq{v}^{\gamma/2} \seq{D_v}^s ({ (p^{1/2})}^{-1})^w}_{\textrm{bounded operator}} H_R \, { (p^{1/2})}^w \psi \Big\|^2 \\
& \lesssim \norm{  { (p^{1/2})}^w \psi}^2, \\
\end{split}
\end{equation*}
where we used that the  operator $ \seq{v}^{\gamma/2} \seq{D_v}^s $ has its Weyl symbol in $S_K(p^{1/2})$ (this Weyl symbol is $\seq{v}^{\gamma/2} \sharp \seq{\eta}^s$), and that $ {(p^{1/2})}^{-1} \in S_K(p^{-1/2})$ , so that $\seq{v}^{\gamma/2} \seq{D_v}^s ({ (p^{1/2})}^{-1})^w$ is a bounded operator.
Using then \eqref{mainnorm} and  \eqref{majorpq}, we get
\begin{equation*}
 \begin{split}
& \norm{ \seq{v}^{\gamma/2} \seq{D_v}^s \psi}^2  \lesssim \norm{  { (p^{1/2})}^w \psi}^2
 \simeq \sep {p^\wick \psi,\psi}
 \lesssim \frac{1}{t} \norm{\psi_0}^2.
\end{split}
\end{equation*}
Similarly,
$$
\norm{ \seq{v}^{\gamma/2+s} \psi}^2  \lesssim \frac{1}{t} \norm{\psi_0}^2,
$$
and working in $S_K(q)$ gives, in the same way,
$$
\norm{ \seq{v}^{\gamma/2}\seq{\xi}^s \psi}^2  \lesssim \frac{1}{t^{1+2s}} \norm{\psi_0}^2.
$$
Taking the inverse Fourier transform in the $x$ variable finally yields
\begin{multline*}
\norm{ \seq{v}^{\gamma/2} \seq{D_v}^s \varphi}^2 \lesssim \frac{1}{t} \norm{\varphi_0}^2,  \qquad \norm{ \seq{v}^{\gamma/2+s} \varphi}^2  \lesssim \frac{1}{t} \norm{\varphi_0}^2 \\
\textrm{ and } \qquad \norm{ \seq{v}^{\gamma/2}\seq{D_x}^s \varphi}^2  \lesssim \frac{1}{t^{1+2s}} \norm{\varphi_0}^2.
\end{multline*}
This is exactly the statement of Proposition \ref{prop:duallambda1}, the proof is thus complete. \qed

%----------------------------------%----------------------------------%----------------------------------%----------------------------------%----------------------------------%----------------------------------%----------------------------------%----------------------------------%
\medskip
\subsection{Adaptation of the proof for the primal result and generalization} \label{sec:general}

{\subsubsection{Adaptation of the proof for the primal result} \label{adaptprimal}
If we want to prove the ``primal'' regularization property in Theorem~\ref{thm:LIBthm}, as in Subsection~\ref{subsec:splitting}, we split $\Lambda$ into two parts:
\begin{equation} \label{decomptilde}
\begin{aligned}
\Lambda h &=\left(- K \langle v \rangle^{\gamma+2s}h - v \cdot \nabla_x h + \int_{\R^3 \times \mathbb{S}^2} \widetilde B_1(v-v_*,\sigma) \mu'_* (h'-h) \, \d\sigma \, \dv_* \right)\\
&\qquad + \Bigg(K \langle v \rangle^{\gamma+2s} h+ \int_{\R^3 \times \mathbb{S}^2} \widetilde B_1(v-v_*,\sigma) (\mu'_*-\mu_*) h\, \d\sigma \, \dv_*\\
&\hskip 4.5cm + \int_{\R^3\times \mathbb{S}^2} \widetilde B_1^c(v-v_*,\sigma) (\mu'_* h' - \mu_* h) \, \d\sigma \, \dv_* + Q(h,\mu) \Bigg) \\
&=: \widetilde{\Lambda}_1h + \widetilde{\Lambda}_2h,
\end{aligned}
\end{equation}
note that this splitting will also be used in Subsection~\ref{subsec:reg}.
Then, the study of $\widetilde{\Lambda}_1^m$ is totally similar to the one of $\Lambda_1^{*,m}$ (the only differences being in the fact that the roles of~$m$ and~$m^{-1}$ are inverted and the sign in front of the transport operator is opposite). We thus just have to adapt the signs in the Lyapunov functional: The sign of $\omega$ has to be changed in Paragraph~\ref{subsec:refweights}. The other part $\widetilde{\Lambda}_2$ is controlled as well as $\Lambda_2$. The proof is thus done in the same way and we do not enter into details.}

%----------------------------------%----------------------------------%----------------------------------%----------------------------------%----------------------------------%----------------------------------%----------------------------------%----------------------------------%
\smallskip
\subsubsection{Generalization to higher order estimates}
Theorem \ref{thm:LIBthm} deals with regularization in close to $L^2$ spaces: For example, it says that that the semigroup associated to $\Lambda$
goes from~$L^2$ to $H^{s}$ type spaces, with suitable weights and explicit norms. One can wonder if an higher order quantitative regularization is also available. This is the aim of the following Theorem, for which we give a condensed statement in the primal case and in~$H^{\ell s}$ spaces  (see notation \eqref{eq:Hnfrac} and below).

\begin{theo} \label{thm:LIBthmhigh}
 Let $\ell\in \N^*$, {$k' \ge 0$ and $k>\max(\gamma/2+3+2(\ell+1) s, k' + \gamma+5/2)$}. Consider also~{$h_0 \in H^{\ell s}_{x,v}(\langle v \rangle^k)$.} Then, there exists ${C_\ell}>0$ independent of~$h_0$ such that, we have:
 {$$
\forall \, t \in (0,1], \quad   \|S_\Lambda(t)h_0\|_{H^{\ell s}_{x,v}(\langle v \rangle^{k'})}\le \frac{C_{\ell}}{t^{1/2+s}} \|h_0\|_{H^{(\ell-1)s}_{x,v}(\langle v \rangle^{k})}.
$$}
%
%$$
%{ \|f(t)\|_{H^{\ell s,0}_{x,v}(\langle v \rangle^k)}\le \frac{C_{r,\ell}}{t^{\ell(1/2+s)}} \|f_0\|_{H^{0,0}_{x,v}(\langle v \rangle^{k'})}
%\quad \text{or} \quad \|f(t)\|_{H^{0,0}_{x,v}(\langle v \rangle^k)} \leq \frac{C_r}{t^{\ell(1/2+s)}} \|f_0\|_{(H^{\ell s,0}_{x,v}(\langle v \rangle^{k'}))'}}
% $$
% where $(H^{ \ell s,0}_{x,v}(\langle v \rangle^{k'}))'$ is the dual space of $H^{\ell s ,0}_{x,v}(\langle v \rangle^{k'})$ with respect to $H^{r,0}_{x,v}(\langle v \rangle^{k'})$.
\end{theo}

In this Section we shall not give the complete proof of this result, since this is very close to the  one of Theorem \ref{thm:LIBthm}, but only elements of it. The remaining of this Section is devoted to these elements.

%In all the following, we consider $\ell \in \N^*$ given by the theorem as well as $k$ and $k'$ given there.
As a first step we  split the operator $\Lambda$ into two parts following  \eqref{decomptilde}:
$
\Lambda = \widetilde{\Lambda}_1 + \widetilde{\Lambda}_2.
$
Adapting the proof of Lemma \ref{lem:lambda_2}, we have for suitable functions $h$
\begin{equation} \label{eq:lambda_2bis}
\|\widetilde{\Lambda}_2  h \|_{H^{(\ell-1) s}_{x,v}( \seq{v}^{k'})} \lesssim \|h\|_{H^{(\ell-1) s}_{x,v} (\langle v \rangle^{k} )}
\end{equation}
where $k$ and $k'$ are given in the statement of Theorem \ref{thm:LIBthmhigh}. We also have the following result:

\begin{prop} \label{resultHs}
We have for all $k \ge 0$ and all $t \in (0,1]$,
$$
\norm{S_{\widetilde{\Lambda}_1}(t) h_0}_{H^{\ell s}_{x,v}(\seq{v}^k)} \lesssim
\frac{1}{t^{1/2+s}} \|h_0\|_{H^{(\ell-1)s}_{x,v}(\langle v \rangle^{k})}.
$$
\end{prop}

\noindent {\it Elements of proof of Proposition \ref{resultHs}.} \ \
Similarly to the beginning of Paragraph \ref{subsec:regdual}, we define for $\ell \in \N$
(here in the primal case):
$$
\left\{
\begin{aligned}
&\widetilde F_\ell=H^{\ell s}_{x,v} \\
&\widetilde G_\ell =H^{\ell s,0}_{x,v}({ \langle v \rangle^{\ell \gamma/2}}) \\
&\widetilde H_\ell=H^{0,\ell s}_{x,v}( \langle v \rangle^{\ell \gamma/2}) \cap  L^2_{x,v}( \langle v \rangle^{\ell (\gamma+2s)/2})
%\\
%&\text{$G'$ the dual of $G$ w.r.t. $F$}\\
%&\text{$H'$ the dual of $H$ w.r.t. $F$.}
\end{aligned}
\right.
$$
and $\widetilde{\Lambda}_1^{m^{-1}} = m^{-1} \widetilde{\Lambda}_1 m$.
We notice that it is sufficient to prove the following two estimates:
\begin{equation} \label{FGHl}
\norm{S_{\widetilde{\Lambda}_1^{m^{-1}} } {(t)} h_0}_{\widetilde G_\ell} \lesssim  \frac{1}{t^{1/2+s}} \norm{h_0}_{\widetilde F_{\ell-1}}, \qquad \norm{S_{\widetilde{\Lambda}_1^{m^{-1}} } {(t)}h_0}_{\widetilde H_\ell} \lesssim  \frac{1}{\sqrt{t}} \norm{h_0}_{\widetilde F_{\ell-1}}.
\end{equation}
In fact by interpolation, estimates \eqref{FGHl} are direct consequences of the following estimates:
\begin{equation} \label{FGHl2}
\norm{S_{\widetilde{\Lambda}_1^{m^{-1}} }{(t)} h_0}_{\widetilde G_\ell} \lesssim  \frac{1}{t^{\ell(1/2+s)}} \norm{h_0}_{\widetilde F_0}, \qquad \norm{S_{\widetilde{\Lambda}_1^{m^{-1}}}{(t)} h_0}_{\widetilde H_\ell} \lesssim  \frac{1}{t^{{\ell}/2}} \norm{h_0}_{\widetilde F_0}.
\end{equation}
The proof is very close to the one given in the dual case: As already mentioned, we essentially have to replace $m$ there by $m^{-1}$ here, change the sign in front of the transport term~$v\cdot\nabla_x$, we also have to work in $\widetilde G_\ell$ or $\widetilde H_\ell$ instead of~$\widetilde G (= \widetilde G_1)$ and $\widetilde H (=\widetilde H_1)$ introduced in Paragraph \ref{subsec:regdual} for getting Proposition \ref{prop:duallambda1}.
%We shall in fact give an idea on how to prove the preceding result using the same tools as in Subsection \ref{subsec:quantization}.
To be more precise, let us recall that a fundamental large parameter $K$ is involved there and enters here in the definition of $\widetilde{\Lambda}_1^{m^{-1}}$. Following the strategy of Subsection
\ref{subsec:quantization}, we get that
$$
\widetilde{\Lambda}_1^{m^{-1}} = -b^w - v\cdot \nabla_x
$$
where $b$ has exactly the same properties as $a$ in Subsection~\ref{subsec:quantization}. In particular as in Lemma~\ref{aposell}, $\Re b \geq 0$ and $\Re(b)$ is elliptic positive in the class $S_K(p)$ as there.
We then pose $B = b^w$  and recall the definitions of the symbols in Paragraph \ref{subsec:refweights}:
$$
p(v,\eta) = \seq{v}^\gamma \lambda_v^{2s} + K \seq{v}^{\gamma + 2s},
\quad
q(v,\eta) = \seq{v}^\gamma \lambda_x^{2s} + K \seq{v}^{\gamma + 2s},
$$
and
$$
\omega (v,\eta) = { - }\langle v\rangle ^\gamma \lambda_x^{s-1} \lambda_v^{s-1}(\eta\cdot \xi+ (v\wedge \eta) \cdot (v\wedge \xi)).
$$
Since we are in the primal and not dual case (the sign in front of the transport term is opposite), we have to take the opposite of $\omega$ that we call $\tilde{\omega}:=-\omega$.

The main point of the analysis is then to introduce, such as in Paragraph~\ref{subsec:propduallambda1}, a suitable functional which is here:
%\begin{multline} \label{hhhl}
%\hhh_l(t) = C\norm{\psi}^2 +
%\sum_{0 \leq \alpha+ \beta \leq  2\ell-2}
%D_{\alpha,\beta} t \sep{\sep{p^{1+\alpha/2}q^{\beta/2}  }^{\rm Wick} \psi,\psi} \\
%+ E_{\alpha,\beta} t^{1+s}  \sep{ \sep{p^{\alpha/2}q^{\beta/2} \tilde{\omega}} ^{\rm Wick} \psi,  \psi} \\
%+ F_{\alpha,\beta}t^{1+2s} \sep{\sep{p^{\alpha/2}q^{1+\beta/2}}^{\rm Wick} \psi,\psi}
%\end{multline}
\begin{multline} \label{hhhl}
\HH_\ell(t) := C\norm{\psi}^2 +
\sum_{0 \leq \alpha+ \beta \leq  \ell-1}
D_{\alpha,\beta} t ^{1+\alpha + \beta (1+2s)}\sep{\sep{p^{1+\alpha}q^{\beta}  }^{\rm Wick} \psi,\psi} \\
+ E_{\alpha,\beta} t^{1/2+ \alpha + (1/2+\beta)(1+2s)}  \sep{ \sep{p^{\alpha}q^{\beta} \tilde{\omega}} ^{\rm Wick} \psi,  \psi} \\
+ F_{\alpha,\beta}t^{\alpha + (1+\beta)(1+2s)} \sep{\sep{p^{\alpha}q^{1+\beta}}^{\rm Wick} \psi,\psi}
\end{multline}
 for well chosen constants $C$, $D_{\alpha,\beta}$, $ E_{\alpha,\beta}$ and $F_{\alpha,\beta}$.
 We note that for $\ell=1$, we get $\HH_1 = \HH$ defined in \eqref{hhh1}. The computations exactly follow the ones done in Subsection~\ref{subsec:propduallambda1} using estimates
similar to the ones given in Paragraph~\ref{subsec:technical}, with the same roles of each term as there in the preceding decomposition. Note that we were note able to restrict the analysis to $\alpha+\beta = \ell-1$ due to too high order terms after time derivation, this explains that the full range of $\alpha$ and $\beta$ is needed to close the estimates and conclude that for a good choice of constants,
$$
\frac{\d}{ \d t} \HH_\ell(t) \leq 0.
$$
We omit the details of the computation as well as the last parts of the proof of \eqref{FGHl2} which leads to Proposition \ref{resultHs}, since it follows the end of Subsection \ref{subsec:quantization} . \qed

\bigskip
It is now straightforward to come back to the proof of Theorem \ref{thm:LIBthmhigh}.

\smallskip
\noindent {\it Elements of proof of Theorem \ref{thm:LIBthmhigh}.}
Taking this result into account and together with~\eqref{eq:lambda_2bis}  we can write
$$
 S_{\Lambda}(t) = S_{\widetilde{\Lambda}_1}(t) + \int_0^t \SS_\Lambda(t-\tau)  (\widetilde{\Lambda}_2 S_{\widetilde{\Lambda}_1})  (\tau) \, \d\tau
 $$
  for $t\in (0,1]$. Arguing as in the proof of Lemma \ref{lem:lambda_1/lambda} we easily get the Theorem  (this strongly uses $s<1/2$). We omit the details.
 \qed

\bigskip
%%%%%%%%%%%%%%%%%%%%%%%%%%%%%%%%%%%%%%%%%%%%%%%%%%%%%%%%%%%%%%%%%%%%%%%%%%%%%%%%%%%%%%%%%%%%%%%%%%%%%%%%%%%%%%%%%%%%%%%%
\section{Exponential decay of the linearized semigroup} \label{sec:lin}
\setcounter{equation}{0}
\setcounter{theo}{0}
%%%%%%%%%%%%%%%%%%%%%%%%%%%%%%%%%%%%%%%%%%%%%%%%%%%%%%%%%%%%%%%%%%%%%%%%%%%%%%%%%%%%%%%%%%%%%%%%%%%%%%%%%%%%%%%%%%%%%%%%

{
% \subsection{Functional spaces}
%We recall that $m$ is a polynomial weight $m(v) = \la v \ra^k$. We introduce the spaces $H^n_{x}\HH^\ell_{v}(m)$ and $\HH^n_x \HH^\ell_v(m)$, $(n,\ell) \in \N^2$ which are respectively associated to the following norms:
%	\beqn \label{eq:norm}
%		\|h\|^2_{H^n_{x}\HH^\ell_{v}(m)} := \sum_{|\alpha|\le \ell, \, \, |\beta| \le n, \, \, |\alpha| + |\beta| \le \max(\ell,n)} \|\partial^\alpha_v \partial^\beta_x  h\|^2_{L^2_{x,v}(m {\langle v \rangle^{- 2 |\alpha| s}})},
%	\eeqn
%and
%	\beqn \label{eq:norm2}
%		\|h\|^2_{{\HH^n_{x}\HH^\ell_{v}}(m)} := \sum_{|\alpha|\le \ell, \, \, |\beta| \le n, \, \, |\alpha| + |\beta| \le \max(\ell,n)} \|\partial^\alpha_v \partial^\beta_x  h\|^2_{L^2_{x,v}(m {\langle v \rangle^{- 2 |\alpha| s - 2 |\beta| s}})}.
%	\eeqn
%We want to establish exponential decay of the semigroup $\SS_{\Lambda}(t)$ in various Lebesgue and Sobolev spaces that we will denote $\EE$:
%	\begin{equation} \label{def:EE}
%	\begin{gathered}
%		\EE :=
%		\left\{ \bal
%			%&L^2_{x,v}(m) \qquad \qquad \qquad \qquad \qquad \quad  \text{with} \quad {k > \gamma/2 + 3+2s} \\
%			&H^n_x \HH^\ell_v(m), \, \, (n, \ell) \in \N^2, \; n \ge \ell \\
%			&\HH^n_x \HH^\ell_v(m),  \, \,  (n, \ell) \in \N^2, \; n \ge \ell 		\eal \right.
%			\quad \text{with}\quad
%	{ k>{\gamma\over 2}+3+2( \max(1,n)+1)s .}
%	\end{gathered}
%	\end{equation}
%Notice that those definitions include the case $L^2_{x,v}(m)$ which can be obtained in one or the other type of space taking $n=\ell=0$.

{
We recall here that $m$ is a polynomial weight $m(v) = \la v \ra^k$ and that we want to establish exponential decay of the semigroup $\SS_{\Lambda}(t)$ in various Lebesgue and Sobolev spaces $\EE$ introduced in \eqref{def:EE2}. For the reader convenience we recall their definition:
\begin{equation} \label{def:EE}
	\begin{gathered}
		\EE :=
		\left\{ \bal
			%&L^2_{x,v}(m) \qquad \qquad \qquad \qquad \qquad \quad  \text{with} \quad {k > \gamma/2 + 3+2s} \\
			&H^n_x \HH^\ell_v(m), \, \, (n, \ell) \in \N^2, \; n \ge \ell \\
			&\HH^n_x \HH^\ell_v(m),  \, \,  (n, \ell) \in \N^2, \; n \ge \ell 		\eal \right.
			\quad \text{with}\quad
	{ k>{\gamma\over 2}+3+2( \max(1,n)+1)s .}
	\end{gathered}
	\end{equation}
}
%----------------------------------%----------------------------------%----------------------------------%----------------------------------%----------------------------------%----------------------------------%----------------------------------%----------------------------------%
\medskip

\subsection{Main result on the linearized operator}	
The main result on the linearized equation is a precise version of Theorem~\ref{theo:main2} and reads
\begin{theo}\label{theo:extension}
	Let us consider $\EE$ be one of the admissible spaces defined in~\eqref{def:EE} and introduce $E=H^{\max(1,n)}_{x,v}(\mu^{-1/2})$ where $n  \in \N$ is the order of $x$-derivatives in the definition of~$\EE$.
	{ Then, for any $\lambda < \lambda_0$,
	where we recall that $\lambda_0>0$ is the spectral gap of $\Lambda$ on~$E$ (see~\eqref{eq:spectralgapE})}, there is a constructive constant $C\ge 1$ such that the operator $\Lambda$ satisfies on~$\EE$:
		\begin{enumerate}[leftmargin=1.1cm]
			\item[(i)] $\Sigma (\Lambda) \subset \{ z \in \C \mid \Re  \, z \le - \lambda \} \cup \{ 0 \}$;
			\item[(ii)] The null-space $N(\Lambda)$ is given by \eqref{eq:kerlambda} and the projection $\Pi_0$ onto $N(\Lambda)$ by \eqref{eq:Pi0};
			\item[(iii)] $\Lambda$ is the generator of a strongly continuous semigroup $\SS_{\Lambda}(t)$ on $\EE$ that verifies
				$$
				\forall \, t \ge 0, \;   \forall\, h_0 \in \EE, \quad \| \SS_{\Lambda}(t) h_0 - \Pi_0 h_0 \|_{\EE} \le C \, e^{-\lambda t} \, \| h_0 - \Pi_0 h_0 \|_{\EE}.
				$$
		\end{enumerate}
\end{theo}
To prove this theorem, we exhibit a splitting of the linearized operator into two parts,
one which is regular and the second one which is dissipative. {We shall also study the regularization properties of the semigroup. The latter point is based on Section~\ref{sec:reg} in which a precise study of the short time regularization properties of the linearized operator is performed.}
We can then use the abstract theorem of enlargement of the functional space of the semigroup decay
from Gualdani et al.~\cite{GMM*} using the result of Mouhot and Neumann~\cite{MN-Nonlinearity} (Theorem~\ref{theo:gapE}) as a starting point.

%----------------------------------%----------------------------------%----------------------------------%----------------------------------%----------------------------------%----------------------------------%----------------------------------%----------------------------------%
\medskip
\subsection{Splitting of the linearized operator} %\label{subsec:splitting}

We recall that $\chi \in \DD(\R)$ is a truncation function which satisfies $\mathds{1}_{[-1,1]} \le \chi \le \mathds{1}_{[-2,2]}$ and that we denote $\chi_{a} (\cdot) = \chi(\cdot/a)$ for $a>0$.
We then introduce
\begin{equation} \label{eq:A+B}
\AA h := M \chi_R h \quad \text{and} \quad \BB h := \Lambda h - \AA h = - v \cdot \nabla_x h + \LL h - \AA h
\end{equation}
for some positive constants $M$ and $R$ to be chosen later. In the next subsection, we are going to prove a coercivity-type inequality of the following form: For $\delta$ small enough,
$$
\langle \LL h, h \rangle_{L^2_v(m)} \le -c_{0,\delta} \|h\|^2_{*} + c_{1,\delta} \|h\|^2_{L^2_v}
$$
where $\|\cdot\|_{*}$ is a stronger norm than the $L^2_v(m)$-norm and $c_{0,\delta}$, $c_{1,\delta}$ are positive constants depending on~$\delta$. Then, choosing suitable constants $M$ and $R$, we will be able to deduce that our operator $\BB$ is indeed dissipative in $L^2_{x,v}(m)$ and that it provides us a gain of regularity coming from the term $-c_{0,\delta} \|h\|^2_{*}$.

}
%
%We linearize the equation around the equilibrium $\mu$. If we set $f = \mu + h$, $h$ satisfies the equation
%	$$
%	\left\{
%		\begin{aligned}
%		&\partial_t h = Q(\mu,h) + Q(h, \mu) - v \cdot \nabla_x h + Q(h,h) \\
%		&h_{|t=0} = h_0 = f_0 - \mu.
%		\end{aligned}
%	\right.
%	$$
%We recall the notations
%	\beqn \label{eq:linearop}
%		\LL h = Q(h,\mu) + Q(\mu,h) \quad \text{and} \quad \Lambda h = \LL h - v \cdot \nabla_x h.
%	\eeqn
%In this section, we state our main result on the linearized operator: the semigroup associated to $\Lambda$ enjoys exponential decay properties in various Sobolev spaces. We also introduce a splitting of the linearized operator that we shall use to prove this result.

%----------------------------------%----------------------------------%----------------------------------%----------------------------------%----------------------------------%----------------------------------%----------------------------------%----------------------------------%
\medskip

%We also introduce the so-called collision frequency
%	$$
%	\nu_\delta(v) :=  \int_{\R^3 \times \Sp^2}  \mu_* \, b^c_\delta(\cos \theta) |v-v_*|^\gamma  \,\d \sigma \,\dv_*,
%	$$
%so that we have the following splitting: $\LL_\delta^c = \AA_{\delta,\varepsilon} + \BB_{\delta,\varepsilon}^c - \nu_\delta$.
%
%Moreover, $\nu_\delta$ satisfies
%	$$
%	\nu_\delta(v) = K_\delta \,(\mu \ast |\cdot|^\gamma)(v)
%	$$
%with
%	\beqn \label{eq:Kdelta}
% 	K_\delta := \int_{\Sp^2} b^c_\delta(\cos \theta) \,\d \sigma \approx \int_\delta^{\pi/2} b(\cos \theta) \, \sin \theta \, d\theta \approx 	\delta^{-2s} - \left(\frac{\pi}{2}\right)^{-2s} \xrightarrow[\delta \rightarrow 0]{} +\infty
%	\eeqn
%using the spherical coordinates to get the second equality and~(\ref{eq:angularsing}) to get the final one; and
%	$$
%	(\mu \ast | \cdot|^\gamma)(v) \approx \langle v \rangle^\gamma.
%	$$
%
%We then define
%	\beqn \label{eq:B0}
%		\BB^0_{\delta,\varepsilon} := \LL_\delta + \BB^c_{\delta,\varepsilon} - \nu_\delta
%	\eeqn
%so that $\LL = \AA_{\delta,\varepsilon} + \BB^0_{\delta,\varepsilon}$. Finally, we denote
%	$$
%	\BB_{\delta,\varepsilon} := - v \cdot \nabla_x + \BB^0_{\delta,\varepsilon}
%	$$
%so that $\Lambda = \AA_{\delta,\varepsilon} + \BB_{\delta,\varepsilon}$.

%----------------------------------%----------------------------------%----------------------------------%----------------------------------%----------------------------------%----------------------------------%----------------------------------%----------------------------------%
\medskip
\subsection{Dissipativity properties} \label{subsec:dissip}
In this subsection, we focus on dissipativity properties of some well chosen part of the linearized operator. Let us highlight the fact that the main difficulties are already here in the homogeneous case (Lemma~\ref{lem:coercive}). To go from there to the inhomogeneous case (Lemma~\ref{lem:BdissipHk}) just consists in introducing an equivalent norm to the usual one in inhomogeneous Sobolev spaces and is thus relatively simpler.
\begin{lem} \label{lem:coercive}
Let { $k>\gamma/2 + 3+2s$}.
For $\delta>0$ small enough, we have:
$$
\langle \LL h, h \rangle_{L^2_v(m)} \le -c_{0}\, \delta^{2-2s} \|h\|^2_{{\dot H}^{s,*}_v(m)} - c_0\,\delta^{-2s} \|h\|^2_{L^2_v(\langle v \rangle^{\gamma/2} m)} + C_\delta \|h\|^2_{L^2_v}.
$$
where $c_0$ is a universal positive constant and $C_\delta$ is a positive constant depending on $\delta$.
\end{lem}

\begin{proof}
In what follows, we denote $H:=hm$. We start by spliting the scalar product~$\langle Q(\mu, h), h \rangle_{L^2_v(m)}$ into two parts:\begin{align*}
&\langle Q(\mu, h), h \rangle_{L^2_v(m)}
= \int_{\R^3 \times \R^3 \times \Sp^2} B(v-v_*,\sigma) \left[\mu'_*\,  h' \, - \mu_* \,  h \right] \, h\, m^2  \,\d \sigma \,\dv_* \,\dv  \\
&\qquad =  \int_{\R^3 \times \R^3 \times \Sp^2}B(v-v_*,\sigma) \left[\mu'_*\,  H' \, - \mu_* \,  H \right] \, H  \,\d \sigma \,\dv_* \,\dv \\
&\qquad \quad +  \int_{\R^3 \times \R^3 \times \Sp^2} B(v-v_*,\sigma)\,\mu'_* \, h' \, h\,m \, (m-m') \,\d \sigma \,\dv_* \,\dv \\
&\qquad =: \langle Q(\mu, H), H \rangle_{L^2_v} + R.
\end{align*}
We recall that for $\delta>0$, $b_\delta$ and $b_\delta^c$ are given by
	$$
	b_\delta (\cos \theta) = \chi_{\delta}(\theta) \, b(\cos \theta) \quad  \text{and} \quad b^c_\delta(\cos \theta) = (1-\chi_\delta(\theta)) \, b(\cos \theta)
	$$
and we denote $B_\delta$, $B_\delta^c$ (resp. $Q_\delta$, $Q_\delta^c$) the associated kernels (resp. operators).
We then write that
\beqn \label{eq:split}
\langle Q(\mu, h), h \rangle_{L^2_v(m)} =
 \langle Q_\delta(\mu, H), H \rangle_{L^2_v} +  \langle Q_\delta^c(\mu, H), H \rangle_{L^2_v} + R
\eeqn
and we are going to estimate each part of this decomposition. First, concerning grazing collisions, using the pre-post change of variables, we have:
\begin{align*}
& \langle Q_\delta(\mu, H), H \rangle_{L^2_v} =  \int_{\R^3 \times \R^3 \times \Sp^2} B_\delta(v-v_*,\sigma) \, \mu_* \, H \, (H'-H) \,\d \sigma \,\dv_* \,\dv \\
 & \qquad  = -{1 \over 2}  \int_{\R^3 \times \R^3 \times \Sp^2} B_\delta(v-v_*,\sigma)\, \mu_* \left( H' - H \right)^2 \,\d \sigma \,\dv_* \,\dv\\
&\qquad \quad +{1 \over 2} \int_{\R^3 \times \Sp^2}B_\delta(v-v_*,\sigma) \, \mu_*  \left( (H')^2 - H^2 \right)\,\d \sigma \,\dv_* \,\dv =: - I_{1} + I_{2}.
\end{align*}
Using the cancellation lemma~\cite[Lemma~1]{ADVW-ARMA}, we have that
$$
I_2 = {1 \over 2} \int_{\R^3} (S_\delta*H^2) \, \mu \,\dv
$$
with 	$S_\delta$ defined in~\eqref{eq:Sdelta} which satisfies $S_\delta(z) \lesssim \delta^{2-2s}|z|^\gamma$. We deduce that
\beqn \label{eq:I_2}
I_2 \lesssim \delta^{2-2s} \|h\|^2_{L^2_v(\la v \ra^{\gamma/2}m)}.
\eeqn
We now treat $I_1$. To do that, we first notice that for $\eps \in (0,1/2)$, we have
$$
|v-v_*|^\gamma \ge \eps \la v-v_* \ra^\gamma - \eps \, \mathds{1}_{|v-v_*|^\gamma\le \eps/(1-\eps)}.
$$
Together with the fact that
$$
\la v-v_* \ra^\gamma \gtrsim \la v'-v_* \ra^\gamma \gtrsim \la v_* \ra^{-\gamma} \la v' \ra^\gamma,
$$
we deduce that
\begin{align*}
I_1 &\ge \eps \int_{\R^3 \times \R^3 \times \Sp^2} b_\delta(\cos \theta) \, \la v-v_* \ra^\gamma \, \mu_* (H'-H)^2 \,\d \sigma \,\dv_* \,\dv \\
&\quad - \eps \int_{\R^3 \times \R^3 \times \Sp^2} b_\delta(\cos \theta) \,\mathds{1}_{|v-v_*|^\gamma\le \eps/(1-\eps)} \, \mu_* (H'-H)^2 \,\d \sigma \,\dv_* \,\dv \\
&\ge C \eps \int_{\R^3 \times \R^3 \times \Sp^2} b_\delta(\cos \theta) \, \mu_* \la v_* \ra^{-\gamma} (H'\la v' \ra^{\gamma/2}-H\la v' \ra^{\gamma/2})^2 \,\d \sigma \,\dv_* \,\dv \\
&\quad - \eps \int_{\R^3 \times \R^3 \times \Sp^2} b_\delta(\cos \theta) \,\mathds{1}_{|v-v_*|^\gamma\le \eps/(1-\eps)} \, \mu_* (H'-H)^2 \,\d \sigma \,\dv_* \,\dv \\
&\ge C {\eps \over 2} \int_{\R^3 \times \R^3 \times \Sp^2} b_\delta(\cos \theta) \, \mu_* \la v_* \ra^{-\gamma} (H'\la v' \ra^{\gamma/2}-H\la v \ra^{\gamma/2})^2 \,\d \sigma \,\dv_* \,\dv \\
&\quad - C \eps\int_{\R^3 \times \R^3 \times \Sp^2} b_\delta(\cos \theta) \, \mu_* \la v_* \ra^{-\gamma} \, H^2 \, (\la v \ra^{\gamma/2} - \la v' \ra^{\gamma/2})^2 \,\d \sigma \,\dv_* \,\dv \\
&\quad - \eps \int_{\R^3 \times \R^3 \times \Sp^2} b_\delta(\cos \theta) \,\mathds{1}_{|v-v_*|^\gamma\le \eps/(1-\eps)} \, \mu_* (H'-H)^2 \,\d \sigma \,\dv_* \,\dv =: I_{11}-I_{12}-I_{13}.
\end{align*}
First, we clearly have
$$
I_{11} \gtrsim \eps \|h\|^2_{\dot H^{s,*}_v(m)}.
$$
For $I_{12}$, we can use~\eqref{eq:vgamma/2} to get
$$
I_{12} \lesssim \eps \,\delta^{2-2s} \|h\|^2_{L^2_v(\la v \ra^{\gamma/2}m)}.
$$
Concerning $I_{13}$, we use that for $\eps \le 1/2$, $\mathds{1}_{|v-v_*|^\gamma\le \eps/(1-\eps)} \le \mathds{1}_{|v-v_*|\le1}$ so that
\begin{align*}
I_{13} &\lesssim \eps \int_{\R^3 \times \R^3 \times \Sp^2} b_\delta(\cos \theta) \, \mathds{1}_{|v-v_*|\le1} \, \mu_* (H'-H)^2\,\d \sigma \,\dv_* \,\dv \\
&\lesssim \eps \int_{\R^3 \times \R^3 \times \Sp^2} b(\cos \theta) \, \mathds{1}_{|v-v_*|\le1} \, \mu_* (H'-H)^2\,\d \sigma \,\dv_* \,\dv.
\end{align*}
From the proof of~\cite[Theorem~1.2]{CH-ARMA}, we get
$$
I_{13} \lesssim \eps \, \|h\|^2_{H^s_v(m)}.
$$
We thus have obtained
$$
{1 \over 2} I_1 \ge c_1 \,\eps \|h\|^2_{\dot H^{s,*}_v(m)} - c_2 \, \eps \|h\|^2_{H^s_v(\la v \ra^{\gamma/2}m)}, \quad c_1, \, c_2>0.
$$
On the other hand, as already mentioned in the proof of Lemma~\ref{lem:anis}, adapting the proof of~\cite[Theorem~3.1]{He-JSP}, we can get that
$$
{1 \over 2} I_1 \ge c_3 \, \delta^{2-2s} \|h\|^2_{H^s_v(\langle v \rangle^{\gamma/2}m)} - c_4 \, \delta^{2-2s} \|h\|^2_{L^2_v(\langle v \rangle^{\gamma/2}m)}, \quad c_3, \, c_4>0.
$$
Combining the two previous inequalities, we get that there exist positive constants $c_i$ for~$i =1, \dots, 4$ such that
\beqn \label{eq:I_1}
\bal
I_1 &\ge c_1\, \eps \,\|h\|^2_{\dot H^{s,*}_v(m)}
+ (c_3\,\delta^{2-2s} -  c_2 \,\eps) \|h\|^2_{H^s_v(\la v \ra^{\gamma/2}m)} \\
& \quad - c_4 \, \delta^{2-2s} \|h\|^2_{L^2_v(\la v \ra^{\gamma/2}m)}.
\eal
\eeqn
Gathering~\eqref{eq:I_2} and~\eqref{eq:I_1}, up to changing the value of $c_4$, we have obtained:
	\beqn \label{eq:Qdeltamu}
	\bal
	&\langle Q_\delta(\mu, H), H \rangle_{L^2_v} \\
	&\qquad \le - c_1 \, \eps \,\|h\|^2_{\dot H^{s,*}_v(m)}
- (c_3\,\delta^{2-2s} - c_2 \,\eps) \|h\|^2_{H^s_v(\la v \ra^{\gamma/2}m)} +c_4\, \delta^{2-2s} \,\|h\|^2_{L^2_{v}(\la v \ra^{\gamma/2} m)}.
	\eal
	\eeqn

We now deal with the cut-off part $\langle Q_\delta^c(\mu, H), H \rangle_{L^2_v}$. In this term, grazing collisions are removed, we can thus separate gain and loss terms:
\begin{align*}
\langle Q_\delta^c(\mu, H), H \rangle_{L^2_v}
&\le \int_{\R^3 \times \R^3 \times \Sp^2} B_\delta^c(v-v_*,\sigma) \,\mu_* \, |H'| \, |H|\,\d \sigma \,\dv_* \,\dv \\
&\quad- \int_{\R^3 \times \R^3 \times \Sp^2} B_\delta^c(v-v_*,\sigma) \,\mu_* \,\d \sigma \,\dv_* \, H^2 \,\dv.
\end{align*}
The loss term is multiplicative and can be rewritten as
	$$
	\int_{\R^3 \times \R^3 \times \Sp^2} B_\delta^c(v-v_*,\sigma) \,\mu_* \,\d \sigma \,\dv_* \, H^2 \,\dv
	= K_\delta \int_{\R^3} (\mu \ast |\cdot|^\gamma) \, H^2 \,\dv
	$$
with
	\beqn \label{eq:Kdelta}
 	K_\delta := \int_{\Sp^2} b^c_\delta(\cos \theta) \,\d \sigma \approx \int_\delta^{\pi/2} b(\cos \theta) \, \sin \theta \, d\theta \approx 	\delta^{-2s} - \left(\frac{\pi}{2}\right)^{-2s} \xrightarrow[\delta \rightarrow 0]{} +\infty
	\eeqn
using the spherical coordinates to get the second equality and~(\ref{eq:angularsing}) to get the final one. Since we also have
	$$
	(\mu \ast | \cdot|^\gamma)(v) \approx \langle v \rangle^\gamma,
	$$
we can deduce that there exists $\nu_0>0$ such that
	\beqn \label{eq:loss}
	-\int_{\R^3 \times \R^3 \times \Sp^2} B_\delta^c(v-v_*,\sigma) \,\mu_* \,\d \sigma \,\dv_* \, H^2 \,\dv
	\le - \nu_0 \, \delta^{-2s} \, \|h\|^2_{L^2_v(\la v \ra^{\gamma/2} m)}.
	\eeqn
Concerning the gain term, following ideas from~\cite{MischlerBook*}, we are going to split it into two parts. To do that, we denote $w:=v+v_*$ and $\widehat w := w/|w|$. We then have
	\begin{align*}
	 &\int_{\R^3 \times \R^3 \times \Sp^2} B_\delta^c(v-v_*,\sigma)\, \mu_* \, |H'| \, |H|\,\d \sigma \,\dv_* \,\dv \\
	 &\qquad =  \int_{\R^3 \times \R^3 \times \Sp^2} \mathds{1}_{|\widehat w \cdot \sigma| \ge 1-\delta^3} \, B_\delta^c(v-v_*,\sigma) \,\mu_* \, |H'| \, |H|\,\d \sigma \,\dv_* \,\dv \\
&\qquad \quad	 +  \int_{\R^3 \times \R^3 \times \Sp^2} \mathds{1}_{|\widehat w \cdot \sigma| \le 1-\delta^3}\, B_\delta^c(v-v_*,\sigma) \,\mu_* \, |H'| \, |H|\,\d \sigma \,\dv_* \,\dv =: J_1+J_2.
	\end{align*}
We first deal with $J_1$: Using Young inequality, we have
	\begin{align*}
	J_1 &\lesssim \delta^{-1/2} \int_{\R^3 \times \R^3 \times \Sp^2} \mathds{1}_{|\widehat w \cdot \sigma| \ge 1-\delta^3} \, B_\delta^c(v-v_*,\sigma)\, \mu_* \, H^2 \,\d \sigma \,\dv_* \,\dv\\
	&\quad + \delta^{1/2}  \int_{\R^3 \times \R^3 \times \Sp^2} \mathds{1}_{|\widehat w \cdot \sigma| \ge 1-\delta^3} \, B_\delta^c(v-v_*,\sigma)\, \mu'_* \,H^2\,\d \sigma \,\dv_* \,\dv =: J_{11} + J_{12}
	\end{align*}
where we have used the pre-post collisional change of variables noticing that $\widehat {w'} = \widehat w$ (with obvious notations).
Using that $b_\delta^c(\cos \theta) \lesssim \delta^{-2-2s}$ on the sphere and $(\mu * |\cdot|^\gamma)(v) \lesssim \langle v \rangle^\gamma$ , we get
	$$
	J_{11} \lesssim {\delta^{-5/2-2s}}  \int_{\R^3} \int_{\Sp^2} \mathds{1}_{|\widehat w \cdot \sigma| \ge 1-\delta^3} \,d \sigma \, H^2 \, \langle v \rangle^\gamma \,\dv.
	$$
Then, since for any $z \in\Sp^2$, we have $\int_{\Sp^2} \mathds{1}_{|z \cdot \sigma| \ge 1-\delta^3} \,d \sigma \lesssim \delta^3$, we obtain
	\beqn \label{eq:J11}
	J_{11} \lesssim {\delta^{1/2-2s}} \|h\|^2_{L^2_v(\la v \ra^{\gamma/2} m)}.
	\eeqn
As far as $J_{12}$ is concerned, we roughly bound it from above as:
	$$
	J_{12} \lesssim \delta^{1/2}  \int_{\R^3 \times \R^3 \times \Sp^2} B_\delta^c(v-v_*,\sigma)\, \mu'_* \,H^2 \,\d \sigma \,\dv_* \,\dv.
	$$
We then perform the regular change of variable $v_* \to v'_*$. { Note that due to the symmetry between
the roles played by $v$ and $v_*$, this change of variable is similar to the one $v \to v'$, shown in the proof of Lemma~\ref{lem:nonlinearhom}. Moreover} notice that $|v-v_*|^\gamma \lesssim |v-v'_*|^\gamma$ to obtain:
	\beqn \label{eq:J12}
	J_{12} \lesssim \delta^{1/2} \int_{\Sp^2} b_\delta^c(\cos \theta)\,d\sigma \int_{\R^3 \times \R^3}  \mu_* |v-v_*|^\gamma \,H^2 \, \dv_* \,\dv
	\lesssim  \delta^{1/2-2s} \|h\|^2_{L^2_v(\la v \ra^{\gamma/2} m)}.
	\eeqn
The analysis of $J_2$ starts similarly as the one of $J_1$ using Young inequality:
\begin{align*}
	J_2 &\lesssim {\delta^{1/2}} \int_{\R^3 \times \R^3 \times \Sp^2} \mathds{1}_{|\widehat w \cdot \sigma| \le 1-\delta^3} \, B_\delta^c(v-v_*,\sigma)\, \mu_* \, H^2 \,\d \sigma \,\dv_* \,\dv\\
	&\quad + {\delta^{-1/2}}  \int_{\R^3 \times \R^3 \times \Sp^2} \mathds{1}_{|\widehat w \cdot \sigma| \le 1-\delta^3} \, B_\delta^c(v-v_*,\sigma)\, \mu'_* \,H^2\,\d \sigma \,\dv_* \,\dv =: J_{21} + J_{22}.
	\end{align*}
The treatment of $J_{21}$ is simple and similar as the one of $J_{12}$, we get:
	\beqn \label{eq:J21}
	J_{21}
	\lesssim  \delta^{1/2-2s} \|h\|^2_{L^2_v(\la v \ra^{\gamma/2} m)}.
	\eeqn
For $J_{22}$, we are going to use the following computation: Denoting $u:=v-v_*$ the relative velocity, we have
	$$
	|v'_*|^2 = {1 \over 4} (|w|^2+|u|^2) - \frac{|w||u|}{2} \widehat w \cdot \sigma
	$$
so that if $|\widehat w \cdot \sigma| \le 1-\delta^3$, then
	$$
	|v'_*|^2 \ge  {1 \over 4} (|w|^2+|u|^2) - (1-\delta^3) \frac{|w||u|}{2} \ge {\delta^3 \over 4} (|w|^2+|u|^2)
	= {\delta^3 \over 2} (|v|^2 + |v_*|^2).
	$$
From this, we deduce that
	$$
	\mu'_* \le e^{- \delta^3 |v|^2/4} e^{- \delta^3 |v_*|^2/4}.
	$$
Consequently,
	\beqn \label{eq:J22}
	J_{22} \lesssim  {\delta^{-5/2-2s}} \int_{\R^3 \times \R^3} |v-v_*|^\gamma\, e^{- \delta^3 |v_*|^2/4} \,H^2 \, e^{- \delta^3 |v|^2/4} \,\dv_* \,\dv \lesssim  {C_\delta} \|h\|^2_{L^2_v}.
	\eeqn
Combining~\eqref{eq:loss},~\eqref{eq:J11},~\eqref{eq:J12},~\eqref{eq:J21} and~\eqref{eq:J22}, we obtain
	\beqn \label{eq:Qdeltamuc}
	\langle Q_\delta^c(\mu,H),H \rangle_{L^2_v} \le
	 \delta^{-2s} \left(c_5\delta^{1/2} - \nu_0\right) \|h\|^2_{L^2_v(\la v \ra^{\gamma/2} m)} + {C_\delta} \|h\|^2_{L^2_v}, \quad c_5>0.
	\eeqn
	
Coming back to~\eqref{eq:split}, it remains to analyse the rest term:
$$
R= \int_{\R^3 \times \R^3 \times \Sp^2} B(v-v_*,\sigma)\,\mu'_* \, h' \, h\,m \, (m-m') \,\d \sigma \,\dv_* \,\dv.
$$
First, let us remark that
		$$
		|m'-m| \leq \Bigg(\sup_{z \in \text{B}(v, |v'-v|)} \left|\nabla m\right|(z) \Bigg) \, |v'-v|,
		$$
	with
		$$
		|v'-v| \lesssim \, |v-v_*|\, \sin (\theta/2).
		$$
	Then, we use the fact that
		$$
		\sup_{z \in \text{B}(v, |v'-v|)} \left|\nabla m\right|(z)  \lesssim \langle v \rangle ^{k-1} + \langle v' \rangle^{k-1} \lesssim \langle v' \rangle^{k-1} \, \langle v'_* \rangle^{k-1},
		$$
	which implies that
		$$
		|m' -m| \lesssim \, \sin (\theta/2) \,|v-v_*| \,\langle v' \rangle^{k-1} \, \langle v'_* \rangle^{k-1}.
		$$
		Consequently, we have:
		\begin{align*}
		R &\lesssim
		 \int_{\R^3 \times \R^3 \times \Sp^2} b(\cos \theta) \, \sin(\theta/2) \,\mu'_* \, \langle v'_* \rangle^{k-1}\, |v-v_*|^{\gamma+1} \, |h'| \,\langle v' \rangle^{k-1} \, |h|\,m \,\d \sigma \,\dv_* \,\dv \\
		 &\lesssim \int_{\R^3 \times \R^3 \times \Sp^2} b(\cos \theta) \, \sin(\theta/2) \,\mu'_* \, \langle v'_* \rangle^{k-1}\,|v-v_*|^{\gamma+2} \, (h')^2 \langle v' \rangle^{2k-2} \,\d \sigma \,\dv_* \,\dv \\
		 &\quad + \int_{\R^3 \times \R^3 \times \Sp^2} b(\cos \theta) \, \sin(\theta/2) \,\mu'_* \, \langle v'_* \rangle^{k-1}\,|v-v_*|^{\gamma} \, h^2 \, m^2 \,\d \sigma \,\dv_* \,\dv.
		\end{align*}
For the first part, we use the pre-post collisional change of variables and for the second one, we use the regular change of variable $v_* \to v'_*$ explained in the proof of Lemma~\ref{lem:nonlinearhom}. It gives us
		\beqn \label{eq:R}
		R \le c_6 \|h\|^2_{L^2_v(\la v \ra^{\gamma/2} m)}, \quad c_6>0.
		\eeqn

Gathering~\eqref{eq:Qdeltamu},~\eqref{eq:Qdeltamuc} and~\eqref{eq:R} yields
	\begin{align*}
	&\langle Q(\mu,h),h \rangle_{L^2_v(m)} \le  - (c_3 \, \delta^{2-2s} - c_2  \, \eps) \|h\|^2_{{H}^{s}_{v}(\la v \ra^{\gamma/2} m )} - c_1 \, \eps \|h\|^2_{\dot H^{s,*}_v(m)}
	\\
	&\qquad \qquad \qquad + \left(c_6 + \delta^{-2s} \left(c_4 \,  \delta^2 + c_5 \,\delta^{1/2} - \nu_0\right)\right)  \|h\|^2_{L^2_v(\la v \ra^{\gamma/2} m)}+ {C_\delta} \|h\|^2_{L^2_v}.
	\end{align*}

We also have from Lemma~\ref{lem:nonlinearhom}-(i) applied with $\varsigma_1=2s$, $\varsigma_2=0$, $N_1=\gamma+2s$ and $N_2=0$:
	$$
	\langle Q(h,\mu),h \rangle_{L^2_v(m)} \le c_7 \|h\|^2_{L^2_v(\la v \ra^{\gamma/2} m)}, \quad c_7>0.
	$$

The two previous inequalities imply
	\begin{align*}
	&\langle \LL h, h \rangle_{L^2(m)}  \le  - (c_3 \, \delta^{2-2s} - c_2  \, \eps) \|h\|^2_{{H}^{s}_{v}(\la v \ra^{\gamma/2} m )} - c_1 \, \eps \|h\|^2_{\dot H^{s,*}_v(m)}
	\\
	&\qquad \qquad \qquad + \left(c_6 + c_7 + \delta^{-2s} \left(c_4 \,  \delta^2 + c_5 \,\delta^{1/2} - \nu_0\right)\right)  \|h\|^2_{L^2_v(\la v \ra^{\gamma/2} m)}+ {C_\delta} \|h\|^2_{L^2_v}.
	\end{align*}
	Taking $\delta$ small enough and then $\eps$ small enough of the order of $\delta^{2-2s}$, we obtain the wanted estimate:
	\begin{align*}
	&\langle \LL h, h \rangle_{L^2_v(m)}  \\
	&\quad \le  - c_0 \, \delta^{2-2s} \|h\|^2_{{\dot H}^{s,*}_{v}(m)} - c_0 \, \delta^{2-2s} \|h\|^2_{{H}^{s}_{v}(\la v \ra^{\gamma/2} m)} -c_0 \, \delta^{-2s}  \|h\|^2_{L^2_v(\langle v \rangle^{\gamma/2}m)} + C_\delta \|h\|^2_{L^2_v}
	\end{align*}
	for some $c_0>0$.
\end{proof}
We can now prove the dissipativity properties of $\BB = - v \cdot \nabla_x + \LL-M\chi_R$ in~$L^2_{x,v}(m)$.
\begin{lem} \label{lem:BdissipL2}
	Let us consider {$k> \gamma/2 + 3+2s$} and $a<0$. There exist $M$ and $R>0$ such that
	$\BB- a$ is dissipative in $L^2_{x,v}(m)$, namely
		$$
		\forall \, t \geq0, \quad \|\SS_{\BB}(t)h_0\|_{L^2_{x,v}(m)} \leq e^{at}\|h_0\|_{L^2_{x,v}(m)}.
		$$
	We even have the following estimate (which is better that simple dissipativity as stated above), for any ${h_0} \in L^2_{x,v}(m)$:
		$$
		\forall \, t \geq0, \quad {1 \over 2} {\d \over \d t} \|\SS_\BB(t)h_0\|^2_{L^2_{x,v}(m)} \le
		- c_1 \, \| \SS_{\BB}(t) h_0\|^2_{L^2_x {H}^{s,*}_v (m)} + a \, \| \SS_{\BB}(t) h_0\|^2_{L^2_{x,v} (\langle v \rangle^{\gamma /2} m)}
		$$
	for some constant $c_1>0$.
\end{lem}

\begin{proof}
	Consider $a<0$ and $\delta>0$ small enough so that the conclusion of Lemma~\ref{lem:coercive} holds and such that $c_0 \, \delta^{-2s}>-a$. We are going to estimate the integral $\int_{\R^3 \times \R^3} (\BB h) \, h \, m^2 \,\dv \,\d x$. We first notice that the term coming from the transport operator gives no contribution:
	$$
	\int_{\T^3 \times \R^3} (v \cdot \nabla_x h) \, h \, m^2 \,\dv \,\d x = {1 \over 2} \int_{\T^3 \times \R^3} (v \cdot \nabla_x h^2)  \, m^2 \,\dv \,\d x =0.
	$$
	Then, using Lemma~\ref{lem:coercive} and integrating in $x$, we obtain
		$$
		\begin{aligned}
			&\int_{\T^3 \times \R^3} (\LL h) \, h \, m^2 \,\dv \,\d x \\
			& \qquad \le -c_{0}\, \delta^{2-2s} \|h\|^2_{L^2_x{H}^{s,*}_v(m)}
			- c_0\,\delta^{-2s} \|h\|^2_{L^2_{x,v}(\langle v \rangle^{\gamma/2} m)} + C_\delta \|h\|^2_{L^2_{x,v}}.
		\end{aligned}
		$$
In summary, we have obtained
	\begin{align}\label{mainq}
	&{ \int_{\T^3 \times \R^3}} (\BB h) \, h \, m^2 \,\dv \,\d x
	\le  -c_{0}\, \delta^{2-2s} \|h\|^2_{L^2_x{H}^{s,*}_v(m)} \\
	&\nonumber \qquad \qquad + \int_{\T^3 \times \R^3} h^2 \, m^2  \langle v \rangle^{\gamma} \left(-c_0 \, \delta^{-2s} + C_\delta \langle v \rangle^{-\gamma} - M \chi_R(v)\right) \,\dv \,\d x.
	\end{align}
Since $-c_0 \, \delta^{-2s} + C_\delta \langle v \rangle^{-\gamma}$ goes to $-c_0 \, \delta^{-2s}<a$ as $|v|$ goes to infinity, we can choose $M$ and~$R$ large enough so that for any $v \in \R^3$, $-c_0 \, \delta^{-2s} + C_\delta \langle v \rangle^{-\gamma} - M \chi_R \le a$, which concludes the proof.
\end{proof}

The goal of the next lemma is to generalize previous dissipativity results to higher order derivatives spaces of type $H^n_{x}\HH^\ell_{v}(m)$ and $\HH^n_x \HH^\ell_v(m)$
defined through their norms in~\eqref{eq:norm} and~\eqref{eq:norm2}.
Notice that,  in order to get our dissipativity result, it is necessary to have less weight on $v$-derivatives (which is induced by the weight $\la v \ra^{-2|\alpha|s}$ in the definitions
of the norms of $H^n_{x}\HH^\ell_{v}(m)$ and $\HH^n_x \HH^\ell_v(m)$).
However, the introduction of the weight $\la v \ra^{-2|\beta|s}$ in order to have less weight on the $x$-derivatives in the space $\HH^n_x \HH^\ell_v(m)$ is not needed at this point
but dissipativity results still hold true doing that and we will make use of it in the nonlinear study in Section~\ref{sec:nonlinear}.

\begin{lem} \label{lem:BdissipHk}
	Let us consider $(n,\ell) \in \N^2$ with $n \ge \ell$. In what follows, $\EE = H^n_{x}\HH^\ell_{v}(m)$ with~{$k> \gamma/2 +3+2(n+1) s$}
	or $\EE = \HH^n_x \HH^\ell_v(m)$ with {$k> \gamma/2 +3+2(n+1)s$}. Then for any $a<0$,
	there exist $M$, $R>0$ such that $\BB - a$ is hypodissipative in $\EE$ in the sense that %there exists a constant $C_a>0$ such that
		$$
		\forall\, t \ge0, \quad \| \SS_{\BB}(t) h_0\|_{\EE} \lesssim e^{a t} \|h_0\|_\EE.
		$$
\end{lem}

{\begin{proof}
The case $n=\ell=0$ is nothing but Lemma~\ref{lem:BdissipL2}. Let us notice that the operator~$\nabla_x$ commutes with the operator $\BB$,
	the treatment of $x$-derivatives is thus simple and one can always reduce to the case $n = \ell$. Moreover, we only handle the case $\EE=H^{n}_{x} \HH^{\ell}_{v}(m)$, the other case being similar.
	We now deal with the case $n=\ell=1$, the higher-order derivatives being treatable in the same way. To do that, we introduce the following norm on $H^1_x \HH^1_v(m)$:
		$$
		\Nt h\Nt^2_{H^1_x \HH^1_v(m)} := \|h\|^2_{L^2_{x,v}(m)} + \|\nabla_x h\|^2_{L^2_{x,v}(m)} +\zeta \, \|\nabla_v h\|^2_{L^2_{x,v}(m_0)}
		$$
	where $\zeta>0$ is a positive constant to be chosen later and $m_0(v) := \langle v \rangle^{-2s} \, m(v) = \la v \ra^{k_0}$ with $k_0:=-2s+k$.
	This norm is equivalent to the classical norm on $H^1_x \HH^1_v(m)$ defined through~\eqref{eq:norm}.
	
	 \noindent In the subsequent proof, $\eta$ is a positive constant that  will be fixed later on. Let us introduce~$h_t := \SS_{\BB}(t) h_0$ with $h_0 \in H^1_x \HH^1_v(m)$.

	Coming back to the proof of Lemma~\ref{lem:BdissipL2}, { thanks to \eqref{mainq}, }we have that
		\beqn \label{eq:ht}
		\bal
		\forall \, t \geq0, \quad &{1 \over 2} {\d \over \d t} \|h_t\|^2_{L^2_{x,v}(m)} \le
		- c_0 \, \delta^{2-2s} \,\| h_t\|^2_{L^2_x { H}^{s,*}_v (m)} \\
		&\quad +\int_{\T^3 \times \R^3} \left(-c_0 \, \delta^{-2s} +C_\delta \langle v \rangle^{-\gamma} - M \chi_R(v)\right) h_t^2 \, m^2 \langle v \rangle^\gamma \,\dv \,\d x.
		\eal
		\eeqn
	Moreover, since the $x$-derivatives commute with $\BB$,
		\beqn \label{eq:dxht}
		\bal
		\forall \, t \geq0, \quad &{1 \over 2} {\d \over \d t} \|\nabla_x h_t\|^2_{L^2_{x,v}(m)} \le
		- c_0 \, \delta^{2-2s} \,\| \nabla_x h_t\|^2_{L^2_x { H}^{s,*}_v (m)} \\
		&\quad +\int_{\T^3 \times \R^3} \left(-c_0 \, \delta^{-2s} +C_\delta \langle v \rangle^{-\gamma} - M \chi_R(v)\right) |\nabla_x h_t|^2 \, m^2 \langle v \rangle^\gamma \,\dv \,\d x.
		\eal
		\eeqn

	Therefore, it remains to consider the $v$-derivatives. In what follows $\partial_x$ and $\partial_v$ stand for $\partial_{x_1}$,$\partial_{x_2}$
	or $\partial_{x_3}$ and $\partial_{v_1}$,$\partial_{v_2}$ or $\partial_{v_3}$, respectively.
	
	\noindent We have
		$$
		\partial_t(\partial_vh_t)= \mathcal B (\partial_vh_t) -\partial_xh_t - M \left(\partial_v \chi_R\right) h_t + Q(h_t, \partial_v \mu)+ Q(\partial_v \mu, h_t),
		$$
	thus, we can split  ${1 \over 2} {\d \over \d t} \|\partial_v h_t\|^2_{L^2_{x,v}(m_0)}$ into five terms, according to the previous computation,
		$$
		 \quad {1 \over 2} {{\d \over \d t}} \|\partial_v h_t\|^2_{L^2_{x,v}(m_0)}:= I_1+\dots+I_5.
		$$
	For the first term we can use again { \eqref{mainq}}, obtaining
		\beqn \label{eq:I1}
		\bal
		\forall \, t \geq0, \quad &I_1 \le
		- c_0 \, \delta^{2-2s} \,\|\partial_v h_t\|^2_{L^2_x { H}^{s,*}_v (m_0)} \\
		&\qquad +\int_{\T^3 \times \R^3} \left(-c_0 \, \delta^{-2s} +C_\delta \langle v \rangle^{-\gamma}
		- M \chi_R(v)\right) |\partial_v h_t|^2 \, m_0^2 \, \langle v \rangle^\gamma \,\dv \,\d x.
		\eal
		\eeqn
	For the second term, we have
		\beqn \label{eq:I2}
		I_2= -  \int_{\T^3 \times \R^3}( \partial_x h_t)\,(\partial_v h_t )\,m_0^2\,\dv \,\d x
		\leq \frac 1 2  \|\partial_v h_t\|^2_{L^2_{x,v}(m_0)} + \frac 1 {2} \|\partial_x h_t\|^2_{L^2_{x,v}(m_0)}.
		\eeqn
	The term $I_3$ is simply handled as follows:
		\beqn \label{eq:I3}
		\begin{aligned}
		I_3 &\lesssim {M \over R} \int_{\T^3 \times \R^3} \mathds{1}_{R \le |v| \le 2R} \, h_t \, (\partial_v h_t) \, m_0^2  \,\dv \,\d x\\
		&\lesssim {M \over R} \int_{\T^3 \times \R^3} \mathds{1}_{R \le |v| \le 2R} \, h_t^2\, m_0^2  \,\dv\,\d x
		+ {M \over R} \int_{\T^3 \times \R^3} \mathds{1}_{R \le |v| \le 2R} \, (\partial_v h_t)^2 \, m_0^2  \,\dv\,\d x.
		\end{aligned}
		\eeqn
	Let us now consider $I_4$. Using Lemma~\ref{lem:nonlinearhom}-(i), we have
		\beqn \label{eq:I4}
		\begin{aligned}
		I_4 &= \int_{\T^3} \langle Q(h_t,\partial_v \mu), \partial_v h_t \rangle_{L^2_v(m_0)} \,\d x
		\lesssim \|h_t\|_{L^2_{x,v}(\langle v \rangle^{\gamma/2} m_0)} \|\partial_v h_t\|_{L^2_{x,v}(\langle v \rangle^{\gamma/2} m_0)} \\
		&\lesssim {1 \over \eta}  \|h_t\|_{L^2_{x,v}(\langle v \rangle^{\gamma/2} m_0)}^2 + \eta \|\partial_v h_t\|_{L^2_{x,v}(\langle v \rangle^{\gamma/2} m_0)}^2.
		\end{aligned}
		\eeqn
	Concerning $I_5$, still using Lemma~\ref{lem:nonlinearhom}-(i), we have:
		\beqn \label{eq:I5}
		\begin{aligned}
			I_5 &= \int_{\T^3} \langle Q(\partial_v \mu, h_t), \partial_v h_t \rangle_{L^2_v(m_0)} \,\d x
			\lesssim \|h_t \|_{L^2_xH^s_v(\langle v \rangle^{\gamma/2+2s} m_0)} \|\partial_v h_t\|_{L^2_xH^s_v(\langle v \rangle^{\gamma/2} m_0)} \\
			&\lesssim {1 \over \eta} \|h_t \|_{L^2_xH^s_v(\langle v \rangle^{\gamma/2} m)}^2 + \eta \|\partial_v h_t\|_{L^2_xH^s_v(\langle v \rangle^{\gamma/2} m_0)}^2.
		\end{aligned}
		\eeqn
	Before concluding, let us remark that from Lemma~\ref{lem:anis},
	$$
	\|h\|^2_{L^2_xH^{s,*}_v(m)} \gtrsim \delta^{2-2s} \|h\|^2_{L^2_xH^s_v(\la v \ra^{\gamma/2}m)}.
	$$
	Combining this fact with estimates~\eqref{eq:ht},~\eqref{eq:dxht} and~\eqref{eq:I1} to~\eqref{eq:I5}, we get:
		$$
		\bal
			&\frac 12\frac d {dt }\Nt h_t\Nt^2_{H^1_x \HH^1_v(m)} =
			\frac 12\frac d {dt } \|h_t\|^2_{L^2_{x,v}(m)} +\frac 12\frac d {dt } \|\nabla_x h_t\|^2_{L^2_{x,v}(m)} +\zeta \,\frac 12\frac d {dt } \|\nabla_v h_t\|^2_{L^2_{x,v}(m_0)}\\
			&\qquad \leq -{c_0 \over 2} \, \delta^{2-2s} \left( \|h_t\|^2_{L^2_xH^{s,*}_v(m)} + \|\nabla_x h_t\|^2_{L^2_xH^{s,*}_v(m)} + \zeta \|\nabla_v h_t\|^2_{L^2_xH^{s,*}_v(m_0)} \right) \\
			&\qquad \quad + \left(- {c_0 \over 2} \, \delta^{4-4s}+ \frac{C\zeta}{\eta} \right) \, \| h_t\|^2_{L^2_x {H}^s_v (\langle v \rangle^{\gamma /2} m)} \\
			 &\qquad \quad +\zeta \left(- {c_0 \over 2} \, \delta^{4-4s}+ C\eta \right) \, \| \nabla_vh_t\|^2_{L^2_x {H}^s_v (\langle v \rangle^{\gamma /2} m_0)} \\
			&\qquad \quad +\int_{\T^3 \times \R^3} \bigg(-c_0 \, \delta^{-2s} +C_\delta \langle v \rangle^{-\gamma} %+ \frac{C\zeta}{\eta} \langle v \rangle^{-4s}
			\\
			&\qquad \qquad \qquad \qquad \qquad \qquad \qquad+ {C \zeta M \over R} \mathds{1}_{R \le |v| \le 2R} \langle v \rangle^{-\gamma-4s} - M \chi_R(v)\bigg) h_t^2 \, m^2 \langle v \rangle^\gamma \,\dv \,\d x \\
			&\qquad \quad +\int_{\T^3 \times \R^3} \bigg(-c_0 \, \delta^{-2s} +C_\delta \langle v \rangle^{-\gamma} + C \zeta \langle v \rangle^{-\gamma-4s} - M \chi_R(v)\bigg) |\nabla_x h_t|^2 \, m^2 \langle v \rangle^\gamma \,\dv \,\d x \\
			&\qquad \quad +\zeta \int_{\T^3 \times \R^3} \bigg(-c_0 \, \delta^{-2s} +C_\delta \langle v \rangle^{-\gamma} + C \langle v \rangle^{-\gamma}%+ C \eta
			\\
			&\qquad \qquad \qquad \qquad \qquad \qquad \qquad+ {C M \over R} \mathds{1}_{R \le |v| \le 2R} \langle v \rangle^{-\gamma}- M \chi_R(v)\bigg) |\nabla_v h_t|^2 \, m_0^2 \, \langle v \rangle^\gamma \,\dv \,\d x
		\eal
		$$
	for a constant $C>0$. Consider now $a<0$ and $\delta$ small enough such that $c_0 \, \delta^{-2s} >-a$. We can then choose, in this order, $\eta$ and $\zeta$ small enough and then $M$ and $R$ large enough such that
		\begin{align*}
		\frac 12\frac d {dt }\Nt h_t\Nt^2_{H^1_x \HH^1_v(m)} &\le   a \, \| h_t\|^2_{L^2_{x,v} (\langle v \rangle^{\gamma /2} m)}
				   + a \, \| \nabla_x h_t\|^2_{L^2_{x,v} (\langle v \rangle^{\gamma /2} m)} + \zeta a \, \| \partial_v h_t\|^2_{L^2_{x,v} (\langle v \rangle^{\gamma /2} m_0)} \\
				   &\quad  -c_1  \left( \|h_t\|^2_{L^2_xH^{s,*}_v(m)} + \|\nabla_x h_t\|^2_{L^2_xH^{s,*}_v(m)} + \|\nabla_v h_t\|^2_{L^2_xH^{s,*}_v(m_0)} \right)
	     \end{align*}
	     for some $c_1>0$, which concludes the proof.
	\end{proof}}

{ \begin{rem}
Notice that if the constants $M$ and $R$ are chosen so that the conclusion of the lemma holds in $\EE = H^n_{x}\HH^\ell_{v}(m)$ or $\EE = \HH^n_x \HH^\ell_v(m)$, then  the conclusion also holds in the spaces $\EE' = H^{n'}_{x}\HH^{\ell'}_{v}(m)$ or $\EE = \HH^{n'}_x \HH^{\ell'}_v(m)$ for any $n'$, $\ell' \le \ell$ and $\ell' \le n'$ with the same constants $M$ and $R$.
\end{rem}}
%----------------------------------%----------------------------------%----------------------------------%----------------------------------%----------------------------------%----------------------------------%----------------------------------%----------------------------------%
\medskip
{\subsection{Regularization properties of $\AA S_\BB(t)$} \label{subsec:reg}

{Recall that $\AA$ and $\BB$ are defined in~\eqref{eq:A+B}. In this part, we focus on the regularization properties of the semigroup $\SS_\BB$ which are crucial in order to get a result on the linearized equation. To do that, we first introduce some notations and tools.

 \smallskip
We define the convolution of two semigroups $\SS_1 * \SS_2$ by
$$
(\SS_1 * \SS_2)(t) := \int_0^t \SS_1(\tau) \, \SS_2(t-\tau) \, \d\tau,
$$
and, for $p \in \N^*$, we define $\SS^{(*p)}$ by $ \SS^{(*p)} = \SS * \SS^{(*(p-1))}$ with $\SS^{(*1)} = \SS$. For $\varsigma \in \R^+$ and~$\nu$ a polynomial weight, we also introduce intermediate spaces
$$
X_\varsigma(\nu) := \left[ H^{\lfloor \varsigma \rfloor}_x \HH^{\lfloor \varsigma \rfloor}_v(\nu), H^{\lfloor \varsigma \rfloor+1}_x \HH^{\lfloor \varsigma \rfloor+1}_v(\nu)\right]_{\varsigma-\lfloor \varsigma \rfloor,2}.
$$
Notice that by standard results of interpolation, if $\BB-a$ is hypodissipative in both spaces~$H^{\lfloor \varsigma \rfloor}_x \HH^{\lfloor \varsigma \rfloor}_v(\nu)$ and $H^{\lfloor \varsigma \rfloor+1}_x \HH^{\lfloor \varsigma \rfloor+1}_v(\nu)$, it is also in $X_\varsigma(\nu)$. Notice also that we have the following continuous embeddings:
\beqn \label{eq:embedX}
X_\varsigma(\nu\langle v \rangle^{2(\lfloor \varsigma \rfloor+1)s}) \hookrightarrow H^\varsigma_{x,v}(\nu) \hookrightarrow X_\varsigma(\nu).
\eeqn

Let us now state a lemma on the regularization properties of the semigroup $S_\BB(t)$.
\begin{lem} \label{lem:regSB}
Let $r \in \N^*$, {$k'>(1-\gamma)/2$} and {$k>k'+\gamma+5/2+ 2(\lfloor(r-1)s\rfloor+2)s$}.
We consider $a<0$ and the operator $\BB$ is defined such that the conclusion of Lemma~\ref{lem:BdissipHk} is satisfied in $H^{\lfloor (r-1)s\rfloor+1}_x \HH^{\lfloor (r-1)s\rfloor+1}_v(\langle v \rangle^{k})$.
%Consider $h \in X_{(r-1)s}(\la v \ra^{k})$.
Then, we have:
$$
\|S_{\BB}(t) h_0\|_{X_{rs}(\langle v \rangle^{k'})} \lesssim {e^{at} \over { 1 \wedge t^{1/2+s}}} \|h_0\|_{X_{(r-1)s}(\la v \ra^{k})}, \quad \forall \, t \ge 0.
$$
\end{lem}
%{\Green Pour \'ecrire l'in\'egalit\'e pr\'ec\'edente, on a besoin que $S_\BB(t)$ soit bien d\'efini sur $X_{rs}(\langle v \rangle^{k'})$ et donc dans
%$H^{E(rs)+1}_x \HH^{E(rs)+1}_v(\langle v \rangle^{k'})$ d'o\`u la condition $k' >\gamma/2+3+2(E(rs)+2)s$ dans l'\'enonc\'e.}

\begin{proof}
\noindent {\it Step 1.} In the first step, we focus on the short time regularization properties of~$S_{\BB}(t)$: We are going to prove that
$$
\|S_{\BB}(t) h_0\|_{X_{rs}(\langle v \rangle^{k'})} \lesssim {1 \over { t^{1/2+s}}} \|h_0\|_{X_{(r-1)s}(\la v \ra^{k})}, \quad \forall \, t \in(0,1].
$$
This estimate yields the conclusion of the lemma for short times $t \in (0,1]$. Recalling the decomposition~\eqref{decomptilde}, we have from Proposition~\ref{resultHs} that for $q \ge 0$,
\beqn \label{eq:lambdatilde1}
\|S_{\widetilde{\Lambda}_1}(t) h_0\|_{H^{rs}_{x,v}(\la v \ra^q)} \lesssim {1 \over t^{1/2+s}} \|h_0\|_{H^{(r-1)s}_{x,v}(\la v \ra^{q})}, \quad \forall \, t \in (0,1]
\eeqn
and for any $\varsigma \in \R^+$
\beqn \label{eq:lambdatilde2}
\|\widetilde{\Lambda}_2 h\|_{H^\varsigma_{x,v}(\la v \ra^{q'})} \lesssim \|h\|_{H^\varsigma_{x,v} (\langle v \rangle^{q})}, \quad q>q'+\gamma+5/2.
\eeqn
We now show how to propagate the regularization properties of $S_{\widetilde{\Lambda}_1}(t)$ to $S_\BB(t)$, using the Duhamel formula. We write:
$$
\BB= \widetilde{\Lambda}_1+(\widetilde{\Lambda}_2-\AA)
$$
so that we have:
$$
S_\BB(t) = S_{\widetilde{\Lambda}_1}(t) + \left(S_{\widetilde{\Lambda}_1}* (\widetilde{\Lambda}_2-\AA) S_\BB\right)(t).
$$
For the first term, using~\eqref{eq:embedX} and~\eqref{eq:lambdatilde1}, we have:
\begin{align*}
&\|S_{\widetilde{\Lambda}_1}(t) h_0\|_{X_{rs}(\langle v \rangle^{k'})}
\lesssim \|S_{\widetilde{\Lambda}_1}(t) h_0\|_{H^{rs}_{x,v}(\langle v \rangle^{k'})}
\lesssim {1 \over t^{1/2+s}} \|h_0\|_{H^{(r-1)s}_{x,v}(\la v \ra^{k'})}\\
&\quad \qquad \lesssim  {1 \over t^{1/2+s}} \|h_0\|_{X_{(r-1)s}(\langle v \rangle^{k'+ 2 (\lfloor (r-1)s \rfloor+1)s})} \lesssim {1 \over t^{1/2+s}} \|h_0\|_{X_{(r-1)s}(\langle v \rangle^k)}.
\end{align*}
For the second one, we introduce $k''$ such that
$$
k\ge k''+2(\lfloor (r-1)s \rfloor+1)s >k'+\gamma+5/2+2(\lfloor (r-1)s \rfloor+1)s
$$
and we use~\eqref{eq:embedX},~\eqref{eq:lambdatilde1} and~\eqref{eq:lambdatilde2}:
\begin{align*}
&\left\|\left(S_{\widetilde{\Lambda}_1}* (\widetilde{\Lambda}_2-\AA) S_\BB\right)(t) h_0\right\|_{X_{rs}(\langle v \rangle^{k'})}
\lesssim \int_0^t \left\|S_{\widetilde{\Lambda}_1}(t-\tau)(\widetilde{\Lambda}_2-\AA) S_\BB(\tau) h_0\right\|_{H^{rs}_{x,v}(\langle v \rangle^{k'})} \, \d\tau \\
&\qquad \lesssim \int_0^t {1 \over (t-\tau)^{1/2+s}} \left\|(\widetilde{\Lambda}_2-\AA) S_\BB(\tau) h_0\right\|_{H^{(r-1)s}_{x,v}(\langle v \rangle^{k'})} \, \d\tau \\
&\qquad \lesssim  \int_0^t {1 \over (t-\tau)^{1/2+s}} \left\|S_\BB(\tau) h_0\right\|_{H^{(r-1)s}_{x,v}(\langle v \rangle^{k''})} \, \d\tau \lesssim \|h_0\|_{X_{(r-1)s}(\langle v \rangle^k)}.
\end{align*}
\noindent {\it Step 2.} In this step, we use Lemma~\ref{lem:BdissipHk} and interpolation combined with the previous estimates for short times to prove the final estimate which holds for all times. If $t \ge 1$, we have
\begin{align*}
&\|S_\BB(t){h_0}\|_{X_{rs}(\la v \ra^{k'})} \\
&\qquad = \|S_\BB(1) S_\BB(t-1){h_0}\|_{X_{rs}(\la v \ra^{k'})}
\lesssim \|S_\BB(t-1) {h_0}\|_{X_{(r-1)s}(\langle v \rangle^k)} \lesssim e^{at} \|{h_0}\|_{X_{(r-1)s}(\langle v \rangle^k)},
\end{align*}
which concludes the proof.
\end{proof}

To apply Theorem~2.13 from~\cite{GMM*}, we study the regularization properties of $(\AA\SS_{\BB})^{(*p)}$ for~$p \in \N$ in the following corollary. We recall that the ``large'' space $\EE$ is given by~\eqref{def:EE} and the associated ``small'' one by $E=H^{\max(1,n)}_{x,v}(\mu^{-1/2})$.

Let $a<-\lambda_0$ where $\lambda_0>0$ is the spectral gap of $\Lambda$ on~$E$ (see \eqref{eq:spectralgapE}).
%and $r_0 \in \N$ the smallest integer such that $\lfloor r_0s \rfloor = \max(1,n)$.
We then consider $\BB$ such that the conclusion of Lemma~\ref{lem:BdissipHk} is satisfied in $H^{\max(1,n)}_x\HH^{\max(1,n)}_v(m)$ (resp. $\HH^{\max(1,n)}_x\HH^{\max(1,n)}_v(m)$) if $\EE=H^{n}_x\HH^\ell_v(m)$ (resp. $\EE=\HH^{n}_x\HH^\ell_v(m)$). Let us mention that it in particular implies that the conclusion of Lemma~\ref{lem:BdissipHk} is also satisfied in $\EE$ and the one of Lemma~\ref{lem:BdissipL2} is also true in~$L^2_{x,v}(m)$.

\begin{cor} \label{cor:ASBreg}
There exists $p \in \N$ such that
$$
\|(\AA\SS_{\BB})^{(*p)}(t)h_0\|_E \lesssim e^{a t} \|h_0\|_\EE, \quad \forall \, t \ge 0.
$$
\end{cor}

\begin{proof}
Let us treat the case $\EE=L^2_{x,v}(m)$ and $E=H^1_{x,v}(\mu^{-1/2})$ which is indicative of all the difficulties since we need to regularize both in space and velocity variables. We consider~$r_0 \in \N^*$ the smallest positive integer such that $\lfloor r_0s \rfloor=1$. Using then the fact that~$\AA$ is a truncation operator, Lemma~\ref{lem:BdissipHk} and Lemma~\ref{lem:regSB}, we get that for any $1 \le r \le r_0$,
$$
\|(\AA S_\BB)(t)\|_{\BBB(X_{(r-1)s}(m) , X_{rs}(m))} \lesssim {e^{at} \over {t^{1/2+s}\wedge 1}}.
$$
To conclude, we use Lemmas~\ref{lem:BdissipL2},~\ref{lem:BdissipHk} combined with the last estimate. Indeed, all those results allow us to use the criterion given in~\cite[Lemma~2.17]{GMM*} and gives us the conclusion.
\end{proof}

}

%-----------------------%-----------------------%-----------------------%-----------------------%-----------------------%-----------------------%-----------------------%-----------------------%-----------------------%-----------------------%-----------------------%-----------------------%
\medskip

\subsection{End of the proof of Theorem~\ref{theo:extension}}
Thanks to the estimates proven in the previous subsections, we now turn to the proof of Theorem~\ref{theo:extension}.
Let $\EE$ be one of the admissible space~\eqref{def:EE} and $E=H^{\max(1,n)}_{x,v}(\mu^{-1/2})$ so that in all the cases, we have $E \subset \EE$ and we already have the decay of the semigroup~$\SS_\Lambda(t)$ in~$E$ from Theorem~\ref{theo:gapE}. We then apply Theorem~2.13 from~\cite{GMM*} whose assumptions are fulfilled thanks to Lemmas~\ref{lem:BdissipL2},~\ref{lem:BdissipHk} and Corollary~\ref{cor:ASBreg}. \qed}

\bigskip
%%%%%%%%%%%%%%%%%%%%%%%%%%%%%%%%%%%%%%%%%%%%%%%%%%%%%%%%%%%%%%%%%%%%%%%%%%%%%%%%%%%%%%%%%%%%%%%%%%%%%%%%%%%%%%%%%%%%%%%%
\section{Cauchy theory for the Boltzmann equation} \label{sec:nonlinear}
\setcounter{equation}{0}
\setcounter{theo}{0}
%%%%%%%%%%%%%%%%%%%%%%%%%%%%%%%%%%%%%%%%%%%%%%%%%%%%%%%%%%%%%%%%%%%%%%%%%%%%%%%%%%%%%%%%%%%%%%%%%%%%%%%%%%%%%%%%%%%%%%%%

This section is devoted to the proof of Theorem~\ref{main1}. %Our proof is based on the study of the linearized equation that we made in previous sections.
The idea is to prove that, using suitable norms, there exists a neighborhood of the equilibrium in which the linear part of the equation is dominant and thus dictates the dynamic.
Consequently, taking an initial datum close enough to the equilibrium, one can construct solutions to the equation and prove exponential stability.

%-----------------------%-----------------------%-----------------------%-----------------------%-----------------------%-----------------------%-----------------------%-----------------------%-----------------------%-----------------------%-----------------------%-----------------------%
\medskip
\subsection{Functional spaces}

In what follows, we use notations of Subsection~\ref{subsec:nonhom}.
More precisely, we define the spaces $\mathbf X$, $\mathbf Y$, $\mathbf Y^*$, ${\bf \bar Y}$ and $\mathbf Y'$ as in~\eqref{eq:XY} and~\eqref{eq:Y'} with a weight
{
$$
m(v) = \la v \ra^k, \quad {k>{21 \over 2}+ \gamma+22s}.
$$}
Similarly, for $i=0,\dots,3$, we define the spaces $\mathbf X_i$, $\mathbf Y_i$, $   {\bf \bar Y}_i$ and $\mathbf Y'_i$ as in~\eqref{eq:XY} and~\eqref{eq:Y'} associated to the weights $m_i(v) = \langle v \rangle^{k_i}$.
The exponents $k_0$ and $k_1$ satisfy the following conditions:
{$$
k_0:=k-2s \quad \text{and} \quad 8+14s< k_1 < k_0-\gamma - {5\over 2}-6s.
$$}
Concerning $k_2$ and $k_3$, we set:
{$$
k_2 := k_1-2s \quad \text{and} \quad  4-\gamma+{3 \over 2} + 6s < k_3<k_2-\gamma - {5 \over 2} -6s.
$$}

\begin{rem} \label{rem:weights}
Notice first that
$$
k>k_0>k_1>k_2>k_3.
$$
Let us also comment briefly the conditions imposed on the weights and explain the introduction of so many spaces.
\begin{itemize}[leftmargin=*]
\item First, in the proof of Proposition~\ref{prop:conv}, we need to be able to apply the result from Proposition~\ref{prop:dissipative} in $\mathbf X_1$, this explains the introduction of the spaces $\mathbf X_2$ and $\mathbf X_3$.
\item The last condition
$$
k_3>4-\gamma+{3 \over 2} + 6s
$$
comes from the fact that we want to apply Theorem~\ref{theo:extension} and Lemma~\ref{lem:nonlinearnonhom} in $X_3$.
\item In our argument explained in the two next subsections, there are two levels in which we have a loss of weight. The first one comes from the regularization estimate~\eqref{eq:regv} ($m_0$ to $m_1$ and $m_2$ to $m_3$), which explains the conditions: $k_1 < k_0-\gamma - 5/2-6s$ and~$k_3<k_2-\gamma - 5/2-6s$. The second one comes from the nonlinear estimates in Lemma~\ref{lem:nonlinearnonhom} ($m$ to $m_0$ and $m_1$ to $m_2$),  which explains the conditions: $k_0:=k-2s$ and~$k_2 := k_1-2s$ (a key element is that we have $\| f \|_{\bf \bar{Y}_0} \lesssim \| f \|_{\mathbf Y}$ and $\|f\|_{\bf \bar Y_2} \lesssim \|f\|_{\mathbf Y_1}$).
\item The two first conditions
$$
k_1>8+14s
\quad \text{and} \quad
k>\gamma+{21 \over 2}+22s
$$
are then naturally induced.
\end{itemize}
\end{rem}

%-----------------------%-----------------------%-----------------------%-----------------------%-----------------------%-----------------------%-----------------------%-----------------------%-----------------------%-----------------------%-----------------------%-----------------------%
\medskip
\subsection{Dissipative norm for the whole linearized operator}
Before going into the proof of an a priori estimate which is going to be the cornerstone of our Cauchy theory, we introduce a norm which is (better than) dissipative for the whole linearized operator~$\Lambda$.
\begin{prop}\label{prop:dissipative}
	%Consider some weight function {$m$ satisfying ???}, and let $X := \HH^3_x L^2_v(m)$ and $Y := \HH^3_x (H^s_{v}(\la v \ra^{\gamma/2}m))$.
	%Consider another weight function {$m_0$ satisfying ???} and denote $X_0 := \HH^3_x L^2_v( m_0)$.
	Define for any $\eta >0$ and {any $\lambda_1 < \lambda$ (where $\lambda >0$ is the optimal rate in Theorem~\ref{theo:extension})} the equivalent norm on $X$ for $\Pi_0 h =0$,
		\beqn\label{norm-diss}
			\Nt h \Nt_{\mathbf X}^2 := \eta \| h \|_{\mathbf X}^2 + \int_0^\infty \| \SS_{\Lambda}(\tau) e^{\lambda_1 \tau} h \|_{\mathbf X_1}^2 \, \d\tau.
		\eeqn
	Then there is $\eta>0$ small enough such that the solution $\SS_\Lambda(t) h$ to the linearized equation satisfies, for any $t \ge 0$ and some constant $K>0$,
		$$
		\frac12\frac{d}{dt} \Nt \SS_{\Lambda}(t) h_0 \Nt_{\mathbf X}^2
		\le - \lambda_1 \Nt \SS_{\Lambda}(t) h_0 \Nt_{\mathbf X}^2  - K \| \SS_{\Lambda}(t) h_0 \|_{\mathbf Y^*}^2, \quad \forall\, h_0 \in \mathbf X, \, \Pi_0 h_0=0.
		$$
\end{prop}

\begin{proof}
	First we remark that the norm $\Nt \cdot \Nt_{\HH^3_x L^2_v(m)}$ is equivalent to the norm $\| \cdot \|_{\HH^3_x L^2_v(m)}$ defined in~{\eqref{eq:HHnxL2v}}
	for any $\eta >0$ and {any $\lambda_1 < \lambda$}. Indeed, using Theorem~\ref{theo:extension}, we have
		$$
		\bal
			\eta \| h \|_{\HH^3_x L^2_v(m)}^2
			&\le \Nt h \Nt_{\HH^3_x L^2_v(m)}^2 = \eta \| h \|_{\HH^3_x L^2_v(m)}^2  + \int_0^\infty \| \SS_{\Lambda}(\tau) e^{\lambda_1 \tau} h \|_{\HH^3_x L^2_v(m_1)}^2 \, \d\tau \\
			&\le \eta \| h \|_{\HH^3_x L^2_v(m)}^2 + \int_0^\infty C^2 e^{- 2 (\lambda - \lambda_1) \tau}  \| h \|_{\HH^3_x L^2_v(m_1)}^2 \, \d\tau \le C \| h \|_{\HH^3_x L^2_v(m)}^2.
		\eal
		$$
	We now compute, denoting $h_t = \SS_\Lambda(t) h_0$,
		$$
		\bal
			\frac12\frac{d}{dt} \Nt h_t \Nt_{\HH^3_x L^2_v(m)}^2
			 = \eta \la \Lambda h_t \,  h_t \ra_{\HH^3_x L^2_v(m)}
			 + \frac12 \int_0^\infty \frac{\partial}{\partial t} \| \SS_\Lambda(\tau) e^{\lambda_1 { \tau}} h_{t} \|_{\HH^3_x L^2_v(m_1)}^2 \, \d\tau
			 =: I_1 + I_2.
		\eal
		$$
	For $I_1$ we write $\Lambda = \AA + \BB$.
	Using the fact that $\AA$ is a truncation operator, we first obtain that
		$$
		\la \AA h_t , h_t \ra_{\HH^3_x L^2_v(m)} \le C \| h_t \|^2_{_{\HH^3_x L^2_v(m_1)}}.
		$$
	Moreover, repeating the estimates for the hypodissipativity of $\BB$ in Lemmas~\ref{lem:BdissipL2} and \ref{lem:BdissipHk} we easily get that for some $K>0$,
		$$
		\bal
			\la \BB h_t , h_t \ra_{\HH^3_x L^2_v(m)}
			&\le - \lambda \| h_t \|_{\HH^3_x L^2_v(m)}^2 - K \| h_t \|_{\HH^{3,s,*}_{x,v} (m)}^2,
		\eal
		$$
	therefore it follows
		$$
		I_1 \le  - \lambda \eta \| h_t \|_{\HH^3_x L^2_v(m)}^2 - \eta K \| h_t \|_{\HH^{3,s,*}_{x,v}(m)}^2 + \eta C \| h_t \|_{\HH^3_x L^2_v(m_1)}^2 .
		$$
	The second term is computed exactly
		$$
		\bal
			I_2
			&= \frac12 \int_0^\infty \frac{\partial}{\partial t} \| \SS_{\Lambda}(\tau + t) e^{\lambda_1 \tau} h_0 \|_{\HH^3_x L^2_v(m_1)}^2 \, \d\tau  \\
			&= \frac12 \int_0^\infty \frac{\partial}{\partial \tau} \| \SS_{\Lambda}(\tau + t) e^{\lambda_1 \tau} h_0 \|_{\HH^3_x L^2_v(m_1)}^2 \, \d\tau
			- \lambda_1 \int_0^\infty \| \SS_\Lambda(\tau) e^{\lambda_1 \tau} h_{t} \|_{\HH^3_x L^2_v(m_1)}^2 \, \d\tau \\
			&= \frac12\left[  \| \SS_\Lambda(\tau) e^{\lambda_1 \tau} h_{t} \|_{\HH^3_x L^2_v(m_1)}^2  \right]_{\tau=0}^{\tau=+\infty}
			- \lambda_1 \int_0^\infty \| \SS_\Lambda(\tau+t) e^{\lambda_1 \tau} h_{t} \|_{\HH^3_x L^2_v(m_1)}^2 \, \d\tau \\
			&= -\frac12  \| h_{t} \|_{\HH^3_x L^2_v(m_1)}^2  - \lambda_1 \int_0^\infty \| \SS_\Lambda(\tau) e^{\lambda_1 \tau} h_{t} \|_{\HH^3_x L^2_v(m_1)}^2 \, \d\tau
		\eal
		$$
	where we have used the semigroup decay from Theorem~\ref{theo:extension}.

	Gathering previous estimates and using that {$\lambda \ge \lambda_1$}, we obtain
		$$
		\bal
			I_1 + I_2
			&\le -\lambda_1 \left\{ \eta \| h_t \|_{\HH^3_x L^2_v(m)}^2 + \int_0^\infty \| \SS_\Lambda(\tau) e^{\lambda_1 \tau} h_{t} \|_{\HH^3_x L^2_v(m_1)}^2 \, \d\tau  \right\} \\
			&\quad
 			- \eta K \| h_t \|_{\HH^{3,s,*}_{x,v}(m)}^2 + \eta C\| h_t \|_{\HH^3_x L^2_v(m_1)}^2
			- \frac12\| h_t \|_{\HH^3_x L^2_v(m_1)}^2.
		\eal
	$$
	We complete the proof choosing $\eta>0$ small enough.
\end{proof}

%-----------------------%-----------------------%-----------------------%-----------------------%-----------------------%-----------------------%-----------------------%-----------------------%-----------------------%-----------------------%-----------------------%-----------------------%
\medskip
\subsection{Regularization properties of $\SS_\Lambda$} \label{subsec:reg2}
In this subsection, we state a result on the regularization properties of $\SS_\Lambda$ which is a key point for having a priori estimates on the nonlinear problem in the next subsection.
\begin{lem} \label{lem:regSlambda}
We have the following estimate:
\beqn \label{eq:regv}
 \|\SS_{\Lambda}(t) h_0\|_{\mathbf X_1} \lesssim \frac{1}{t^{1/2}} \|h_0\|_{\mathbf Y'_0}, \quad \forall \, t \in (0,1].
\eeqn

% We also have for $k'>k+\gamma+5/2$:
%\beqn \label{eq:regxv}
%\|\SS_{\Lambda}(t)f_0\|_{H^{r+s,0}_{x,v}(\langle v \rangle^k)} \lesssim {1 \over t^{1/2+s}} \|f_0\|_{H^{r-s,0}_{x,v}(\langle v \rangle^{k'})}, \quad \forall \, t \in (0,1].
%\eeqn
\end{lem}

\begin{proof}
The result that we want to prove is a twisted version of Theorem~\ref{thm:LIBthm}, the only difference being in the weights. First, we notice that
$$
\|\SS_{\Lambda}(t) h_0\|_{\mathbf X_1} \lesssim \|\SS_{\Lambda}(t) h_0\|_{H^{3,0}_{x,v}(\langle v \rangle^{k_1})}.
$$
Theorem~\ref{thm:LIBthm} gives us that for $k'>k_1+\gamma+5/2$, we have:
$$
\|\SS_{\Lambda}(t) h_0\|_{H^{3,0}_{x,v}(\langle v \rangle^{k_1})} \lesssim {1 \over \sqrt{t}} \|h_0\|_{(H^{3,s}_{x,v}(\langle v \rangle^{k'}))'}, \quad \forall \, t \in (0,1].
$$
% where $(H^{3,s}_{x,v}(\langle v \rangle^{k'}))'$ is the dual space of $H^{3,s}_{x,v}(\langle v \rangle^{k'})$ with respect to $H^{3,0}_{x,v}(\langle v \rangle^{k'})$.
 It remains to show that if $k_0=k'+6s >k_1+\gamma+5/2+6s$, we have
 $$
 \|h_0\|_{(H^{3,s}_{x,v}(\langle v \rangle^{k'}))'} \lesssim \|h_0\|_{(\HH^{3,s}_{x,v}(\langle v \rangle^{k_0}))'}.
 $$
 Indeed,
\begin{align*}
 \|h_0\|_{(H^{3,s}_{x,v}(\langle v \rangle^{k'}))'} &= \sup_{\sum_{j=0}^3 \|\nabla^j_x( \varphi \langle v \rangle^{k'})\|_{H^{0,s}_{x,v}} \le 1}
 \sum_{j=0}^3 \big\langle \nabla^j_x h_0 \langle v \rangle^{k_0-2js}, \nabla^j_x \varphi \langle v \rangle^{2k'-(k_0-2js)} \big\rangle_{L^2_{x,v}} \\
 &= \sup_{\sum_{j=0}^3 \|\nabla^j_x(\psi \langle v \rangle^{2(k_0-2js)-k'})\|_{H^{0,s}_{x,v}} \le 1}
 \sum_{j=0}^3 \big\langle \nabla^j_x h_0\langle v \rangle^{k_0-2js}, \nabla^j_x \psi \langle v \rangle^{k_0-2js} \big\rangle_{L^2_{x,v}} \\
 &\le \sup_{\sum_{j=0}^3 \|\nabla^j_x(\psi \langle v \rangle^{k_0-2js})\|_{H^{0,s}_{x,v}} \le 1}
 \sum_{j=0}^3 \big\langle \nabla^j_x h_0  \langle v \rangle^{k_0-2js}, \nabla^j_x \psi \langle v \rangle^{k_0-2js} \big\rangle_{L^2_{x,v}} \\
 &= \sup_{\|\psi\|_{\HH^{3,s}_{x,v}(\langle v \rangle^{k_0})}\le 1} \langle h_0, \psi \rangle_{\HH^{3,0}_{x,v}(\langle v \rangle^{k_0})} \\
 &\le \|h_0\|_{(\HH^{3,s}_{x,v}(\langle v \rangle^{k_0}))'},
\end{align*}
where we used~\eqref{prop:embed} to obtain the third bound and this concludes the proof of~\eqref{eq:regv}.
%Concerning~\eqref{eq:regxv}, it is a more direct consequence of {Theorem~1.2 from~\cite{HTT2*}}. Indeed, from the latter, we have:
%$$
% \|\SS_{\Lambda}(t) h\|_{H^{r,0}_{x,v}(\langle v \rangle^k)} \lesssim \frac{1}{t^{1/2+s}} \|h\|_{(H^{r+s,0}_{x,v}(\langle v \rangle^{k'}))'}, \quad \forall \, t \in (0,1].
% $$
%Then,
%\begin{align*}
%\|h\|_{(H^{r+s,0}_{x,v}(\langle v \rangle^{k'}))'} &= \sup_{\|\varphi\|_{H^{r+s,0}_{x,v}(\langle v \rangle^{k'})}\le 1} \langle h, \varphi \rangle_{H^{r,0}_{x,v}(\langle v \rangle^{k'})} \\
%&=  \sup_{\|\varphi\|_{H^{r+s,0}_{x,v}(\langle v \rangle^{k'})}\le 1} \int_{\R^3 \times \Z^3} \widehat{h \langle v \rangle^{k'}} (\xi, \eta) \, \widehat{ \varphi \langle v \rangle^{k'}} ( \xi, \eta) \langle \xi \rangle^{r+s} \,\langle \xi \rangle^{r-s} d\xi \, \d\eta \\
%&\le \sup_{\|\varphi\|_{H^{r+s,0}_{x,v}(\langle v \rangle^{k'})}\le 1} \|\varphi\|_{H^{r+s,0}_{x,v}(\langle v \rangle^{k'})} \|h\|_{H^{r-s,0}_{x,v}(\langle v \rangle^{k'})} \\
%&\le  \|h\|_{H^{r-s,0}_{x,v}(\langle v \rangle^{k'})},
%\end{align*}
%which concludes the proof of~\eqref{eq:regxv}.
\end{proof}

%----------------------------------%----------------------------------%----------------------------------%----------------------------------%----------------------------------%----------------------------------%----------------------------------%----------------------------------%
\medskip
\subsection{Proof of Theorem \ref{main1}} \label{subsec:conclusion}
We consider the Cauchy problem for the perturbation $h$ defined through $h= f - \mu$. The equation satisfied by $h=h(t,x,v)$ is
	\beqn\label{eq:h}
		\left\{
			\bal
			&\partial_t h  = \Lambda h + Q(h,h) \\
			&h_{|t=0} = h_0= f_0 - \mu.
			\eal
		\right.
	\eeqn
From the conservation laws (see \eqref{eq:conserv}), for all $t >0$,  $\Pi_0 { h(t,\cdot,\cdot) }= 0$ since $\Pi_0 h_0 = 0$, more precisely $\int_{\T^3 \times \R^3} { h(t,x,v)} \,\dv \,\d x = \int_{\T^3 \times \R^3} v_j { h(t,x,v)} \,\dv \,\d x = \int_{\T^3 \times \R^3} |v|^2 { h(t,x,v)} \,\dv \,\d x=0$ for~$j = 1,2,3$.
Note that we also have $\Pi_0 Q(h_t,h_t) = 0$.

%----------------------------------%----------------------------------%----------------------------------%----------------------------------%----------------------------------%----------------------------------%----------------------------------%----------------------------------%
\smallskip
\subsubsection{A priori estimates}

\begin{prop} \label{prop:stab}
	Any solution {$ h=h(t,\cdot,\cdot)$} to \eqref{eq:h} satisfies, at least formally, the following differential inequality: For any {$\lambda_1 < \lambda$
	(where $\lambda >0$ is one rate given by Theorem~\ref{theo:extension})}, there holds
		$$
		\frac12\frac{d}{dt} \Nt h \Nt_{\mathbf X}^2 \le - \lambda_1 \Nt h \Nt_{\mathbf X}^2 -\big( K - C\Nt h \Nt_{\mathbf X} \big) \| h \|_{\mathbf Y^*}^2,
		$$
	for some constants $K,C>0$ and where we recall that the norm $\Nt \cdot \Nt$ is defined in Proposition~\ref{prop:dissipative}.
\end{prop}

\begin{proof}
	We compute the evolution of $\Nt h \Nt$ where ${ h}$ is solution of~\eqref{eq:h}:
		$$
		\bal
			\frac12\frac{d}{dt} \Nt h \Nt_{\mathbf X}^2
			&= \eta \la h , \Lambda h \ra_{\HH^3_x L^2_v(m)}  + \int_0^\infty \la \SS_\Lambda(\tau) e^{\lambda_1 \tau} h , \SS_\Lambda(\tau)e^{\lambda_1 \tau} \Lambda h \ra_{\HH^3_x L^2_v(m_1)} \, \d\tau  \\
			&\quad + \eta \la h , Q(h,h) \ra_{\HH^3_x L^2_v(m)} + \int_0^\infty \la \SS_\Lambda(\tau)e^{\lambda_1 \tau} h , \SS_\Lambda(\tau)e^{\lambda_1 \tau}  Q(h,h) \ra_{\HH^3_x L^2_v(m_1)} \, \d\tau \\
			&=: I_1+I_2 + I_3 + I_4.
		\eal
		$$
	For the linear part $I_1+I_2$, we already have from Proposition~\ref{prop:dissipative} that, for any $\lambda_1 < \lambda$,
		$$
		I_1+I_2 \le - \lambda_1 \Nt h \Nt_{\mathbf X}^2 - K \| h \|_{\mathbf Y^*}^2 .
		$$
	We now deal with the nonlinear part, using first Lemma~\ref{lem:nonlinearnonhom}:
		$$
		\bal
			I_3 &\lesssim \la Q(h,h),h \ra_{\mathbf X} \lesssim \| h \|_{\mathbf X} \, \| h \|_{\mathbf Y^*}^2  \lesssim \Nt h \Nt_{\mathbf X} \, \| h \|_{\mathbf Y^*}^2.
		\eal
		$$
	For the last term $I_4$, we use the fact that $\Pi_0 {h} = 0$ and $\Pi_0 Q({h},{h})=0$ for all $t\ge 0$, together with the estimate~\eqref{eq:regv} from Lemma~\ref{lem:regSlambda}.
	More precisely, if $\Pi_0 h = 0$, using Theorem~\ref{theo:extension} in $\mathbf X_1$, we have:
		$$
		\forall \, {\tau} \ge 0, \quad \|S_{\Lambda}({\tau}) h \|_{\mathbf X_1}\lesssim e^{-\lambda {\tau}}\| h \|_{\mathbf X_1}.
		$$
	Combined with the estimate~\eqref{eq:regv} from Lemma~\ref{lem:regSlambda}, we deduce that for $\Pi_0 h =0$,
		$$
		\forall \, {\tau} >0, \quad \|S_{\Lambda}({\tau})h\|_{\mathbf X_1} \lesssim \frac{e^{-\lambda {\tau}}}{1 \wedge \sqrt{{\tau}}}  \|h\|_{\mathbf Y'_0}.
		$$
	It implies
		$$
		\begin{aligned}
			&\int_0^\infty \la \SS_\Lambda (\tau) e^{\lambda_1 \tau} h , \SS_\Lambda (\tau) e^{\lambda_1 \tau} Q(h,h) \ra_{\mathbf X_1} \, \d\tau \\
			&\qquad \le \int_0^\infty \| \SS_\Lambda (\tau) e^{\lambda_1 \tau} h\|_{\mathbf X_1} \, \| \SS_\Lambda (\tau) e^{\lambda_1 \tau} Q(h,h) \|_{\mathbf X_1} \, \d\tau \\
			&\qquad \lesssim \| h \|_{\mathbf X_1} \, \| Q(h,h) \|_{\mathbf Y'_0} \int_0^\infty e^{-(\lambda - \lambda_1) \tau} \, \frac{e^{-(\lambda - \lambda_1) \tau}}{1 \wedge \sqrt{\tau}} \, \d\tau \\
			&\qquad \lesssim \| h \|_{\mathbf X_1} \, \| Q(h,h) \|_{\mathbf Y'_0}.
		\end{aligned}
		$$
	To conclude, we use Lemma~\ref{lem:nonlinearnonhom}:
		$$
		I_4 \lesssim \|h\|_{\mathbf X_1} \, \|h\|_{\mathbf X_0} \,  \|h\|_{{\bf \bar Y}_0} \lesssim  \Nt h \Nt_{\mathbf X} \, \| h \|_{\mathbf Y}^2 \lesssim \Nt h \Nt_{\mathbf X} \, \| h \|_{\mathbf Y^*}^2.
		$$
\end{proof}

We prove now an a priori estimate on the difference of two solutions to \eqref{eq:h}.
\begin{prop}\label{prop:conv}
	Consider two solutions $g$ and $h$ to \eqref{eq:h} associated to initial data $g_0$ and $h_0$, respectively. Then, at least formally, the difference $g-h$ satisfies the following differential inequality
		$$
		\bal
			\frac12\frac{d}{dt} \Nt g-h \Nt_{\mathbf X_1}^2
			&\le
			%- \lambda_2 \Nt f - g \Nt_{\mathbf X_1}^2
			- K \| g-h \|_{\mathbf Y_1^*}^2  + C \big(\|g\|_{\mathbf X_1} + \| h \|_{\mathbf X_1}) \, \|g-h \|_{\mathbf Y_1^*}^2 \\
			&\quad + C \big( \| h \|_{\mathbf Y_1} + \| g \|_{\mathbf Y} \big)\, \| g-h \|_{\mathbf X_1} \, \| g-h \|_{\mathbf Y_1},
		\eal
		$$
	for some constants $K,C>0$ and where $\Nt \cdot \Nt_{\mathbf X_1}$ is defined as $\Nt \cdot \Nt_{\mathbf X}$ in~\eqref{norm-diss}:
	$$
	\Nt h \Nt_{\mathbf X_1}^2 := \eta \| h \|_{\mathbf X_1}^2 + \int_0^\infty \| \SS_{\Lambda}(\tau) e^{\lambda_1 \tau} h \|_{\mathbf X_3}^2 \, \d\tau.
	$$
	
\end{prop}

\begin{proof}
	We write the equation safisfied by {$g_t-h_t$, denoting $g=g_t$ and $h=h_t$}:
		$$
		\left\{
			\bal
				&\partial_t (g-h) = \Lambda(g-h) + Q(h,g-h) + Q(g-h,g), \\
				&(g-h)_{|t=0} = g_0 - h_0.
			\eal
		\right.
		$$
	We compute
		$$
		\bal
			&\quad \frac12\frac{d}{dt} \Nt g - h \Nt_{\mathbf X_1}^2 \\
			&= \eta \la (g - h) , \Lambda (g - h) \ra_{\mathbf X_1} + \int_0^\infty \la \SS_\Lambda(\tau) e^{\lambda_1 \tau} (g - h) , \SS_\Lambda(\tau)e^{\lambda_1 \tau} \Lambda (g - h) \ra_{\mathbf X_3} \, \d\tau     \\
			&\quad + \eta \la (g - h) , Q(h,g - h) \ra_{\mathbf X_1} + \int_0^\infty \la \SS_\Lambda(\tau)e^{\lambda_1 \tau} (g - h) , \SS_\Lambda(\tau)e^{\lambda_1 \tau}  Q(h,g-h) \ra_{\mathbf X_3} \, \d\tau \\
			&\quad + \eta \la (g - h) , Q(g - h, g ) \ra_{\mathbf X_1} + \int_0^\infty \la \SS_\Lambda(\tau)e^{\lambda_1 \tau} (g - h) , \SS_\Lambda(\tau)e^{\lambda_1 \tau}  Q(g - h,g) \ra_{\mathbf X_3} \, \d\tau \\
			&=: T_1+T_2 + T_3 + T_4 + T_5 + T_6.
		\eal
		$$
	Since the proof follows closely the one of Proposition \ref{prop:stab}, we do not give too much details here (notice that the spaces indexed by $2$ are implicitly used in the following estimates as the spaces indexed by $0$ were used in Proposition~\ref{prop:stab}). We have:
	%for any $\lambda_2 < \lambda_1$,
		$$
		T_1+T_2 \le
		%- \lambda_2 \Nt f - g \Nt_{X_0}^2
		- K \| g-h \|_{\mathbf Y_1^*}^2,
		$$
	and also
		$$
		T_3+T_4 \lesssim \| h \|_{\mathbf X_1} \, \| g-h \|_{\mathbf Y_1^*}^2 +
		\| h \|_{\mathbf Y_1} \, \| g-h \|_{\mathbf X_1} \, \| g-h \|_{\mathbf Y_1}.
		$$
	Moreover, for the last part $T_5 + T_6$, using Lemma~\ref{lem:nonlinearnonhom}-$(i)$, we get
		\begin{align*}
		T_5 + T_6 &\lesssim \| g-h \|_{\mathbf X_1} \, \| g \|_{{\bf \bar Y}_1} \, \| g-h \|_{\mathbf Y_1} + \| g \|_{\mathbf X_1} \|g-h\|_{\mathbf Y_1}^2 \\
		&\lesssim \| g-h \|_{\mathbf X_1} \, \| g \|_{\mathbf Y} \, \| g-h \|_{\mathbf Y_1}+ \| g \|_{\mathbf X_1} \|g-h\|_{\mathbf Y_1}^2,
		\end{align*}
	which completes the proof.
\end{proof}

%----------------------------------%----------------------------------%----------------------------------%----------------------------------%----------------------------------%----------------------------------%----------------------------------%----------------------------------%
\smallskip
\subsubsection{End of the proof}
The end of the proof of Theorem~\ref{main1} is classical and we do not enter into details here. It follows a standard argument by introducing an iterative scheme whose convergence and stability is shown thanks to Propositions~\ref{prop:stab} and~\ref{prop:conv}. The framework being exactly the same, we refer to Subsections~3.4.2. and~3.4.3 from~\cite{CTW-ARMA} in which a more precise proof is given.

\appendix
\bigskip
%%%%%%%%%%%%%%%%%%%%%%%%%%%%%%%%%%%%%%%%%%%%%%%%%%%%%%%%%%%%%%%%%%%%%%%%%%%%%%%%%%%%%%%%%%%%%%%%%%%%%%%%%%%%%%%%%%%%%%%%
\section{Proof of Lemma~\ref{lem:nonlinearnonhom}} \label{app:nonlinear}
\setcounter{equation}{0}
\setcounter{theo}{0}
%%%%%%%%%%%%%%%%%%%%%%%%%%%%%%%%%%%%%%%%%%%%%%%%%%%%%%%%%%%%%%%%%%%%%%%%%%%%%%%%%%%%%%%%%%%%%%%%%%%%%%%%%%%%%%%%%%%%%%%%
In this proof, we use Lemma~\ref{lem:nonlinearhom}-$(i)$ and $(ii)$ together with the following inequalities when integrating in $x \in \T^3$,
		\beqn \label{eq:Sob_x}
			\| u \|_{L^\infty (\T^3_x)} \lesssim \| u \|_{H^2 (\T^3_x)},
			\quad \| u \|_{L^6 (\T^3_x)} \lesssim \| u \|_{H^1 (\T^3_x)},
			\quad \| u \|_{L^3 (\T^3_x)} \lesssim \|  u \|_{H^1 (\T^3_x)}.
		\eeqn
\noindent {\it Proof of (i).} We write
		$$
		\la  Q(f,g),h \ra_{\HH^3_x L^2_v(m)} = \la  Q(f,g),h \ra_{L^2_{x,v}(m)}
		+ \sum_{1 \le |\beta| \le 3} \la \partial^\beta_x Q(f,g), \partial^\beta_x h \ra_{L^2_{x,v}(m \la v \ra^{-2|\beta|s})},
		$$
	and
		$$
		\partial^\beta_x Q(f,g) = \sum_{\beta_1 + \beta_2 = \beta} C_{\beta_1,\beta_2} \, Q ( \partial^{\beta_1}_x f , \partial^{\beta_2}_x g).
		$$In the following steps we will always consider $\ell \in (\gamma+1+3/2,k-6s]$ which is possible since $k>\gamma/2+3+8s$, $\gamma\le 1$ and $s \ge 0$.

	\noindent {\it Step 1.}  Using Lemma~\ref{lem:nonlinearhom}-(i) applied with $\varsigma_1=\varsigma_2=s$, $N_1=\gamma/2+2s$, $N_2=\gamma/2$ and~\eqref{eq:Sob_x} we have
		$$
		\bal
 			&\quad \la  Q(f,g),h \ra_{L^2_{x,v}(m)} \\
			&\lesssim \int_{\T^3} \bigg(\|f\|_{L^2_v(\la v \ra^\ell)} \, \|g\|_{H^s_v(\la v \ra^{\gamma/2+2s}m)} \, \|h\|_{H^s_v(\la v \ra^{\gamma/2}m)} \\
			&\quad + \|f\|_{L^2_v (\la v \ra^{\gamma/2}m)} \, \|g\|_{L^2_v(\la v \ra^\ell)} \, \|h\|_{L^2_v (\la v \ra^{\gamma/2}m)}\bigg) \\
			&\lesssim \|f\|_{H^2_xL^2_v(\la v \ra^\ell)} \, \|g\|_{L^2_xH^s_v(\la v \ra^{\gamma/2+2s}m)} \, \|h\|_{L^2_xH^s_v(\la v \ra^{\gamma/2}m)} \\
			&\quad +  \|f\|_{L^2_{x,v} (\la v \ra^{\gamma/2}m)}\, \|g\|_{H^2_xL^2_v(\la v \ra^\ell)} \,  \|h\|_{L^2_{x,v} (\la v \ra^{\gamma/2}m)} \\
			&\lesssim \|f\|_{\mathbf X} \, \|g\|_{{\bf \bar Y}} \, \|h\|_{\mathbf Y} + \|f\|_{\mathbf Y} \,  \|g\|_{\mathbf X} \, \|h\|_{\mathbf Y}.
		\eal
		$$
	\noindent {\it Step 2.} {\it Case $|\beta|=1$.} Arguing as in the previous step,
		$$
		\bal
 			&\quad \la  Q(f,\partial^\beta_xg),\partial^\beta_x h \ra_{L^2_{x,v}(\la v \ra^{-2s} m)} \\
			&\lesssim \int_{\T^3} \bigg(\|f\|_{L^2_v(\la v \ra^\ell)} \, \|\nabla_xg\|_{H^s_v(\la v \ra^{\gamma/2}m)} \, \|\nabla_xh\|_{H^s_v(\la v \ra^{\gamma/2-2s}m)} \\
			&\qquad + \|f\|_{L^2_v (\la v \ra^{\gamma/2-2s}m)} \, \|\nabla_xg\|_{L^2_v(\la v \ra^\ell)} \, \|\nabla_xh\|_{L^2_v (\la v \ra^{\gamma/2-2s}m)}\bigg) \\
			&\lesssim \|f\|_{H^2_xL^2_v(\la v \ra^\ell)} \,\|\nabla_xg\|_{L^2_xH^s_v(\la v \ra^{\gamma/2}m)} \, \|\nabla_xh\|_{L^2_xH^s_v(\la v \ra^{\gamma/2-2s}m)} \\
			&\qquad +  \|f\|_{L^2_{x,v} (\la v \ra^{\gamma/2-2s}m)}\, \|\nabla_xg\|_{H^2_xL^2_v(\la v \ra^\ell)} \,  \|\nabla_x h\|_{L^2_{x,v} (\la v \ra^{\gamma/2-2s}m)} \\
			&\lesssim \|f\|_{\mathbf X} \, \|g\|_{{\bf \bar Y}} \, \|h\|_{\mathbf Y} + \|f\|_{\mathbf Y} \,  \|g\|_{\mathbf X} \, \|h\|_{\mathbf Y}.
		\eal
		$$
	Moreover,
		$$
		\bal
			&\quad \la Q(\partial^\beta_x f, g) , \partial^\beta_x h \ra_{L^2_{x,v}(\la v \ra^{-2s} m)} \\
			&\lesssim \int_{\T^3} \bigg(\|\nabla_x f\|_{L^2_v(\la v \ra^\ell)} \, \|g\|_{H^s_v(\la v \ra^{\gamma/2} m)} \, \|\nabla_x h\|_{H^s_v(\la v \ra^{\gamma/2-2s} m)} \\
			&\quad + \|\nabla_x f\|_{L^2_v (\la v \ra^{\gamma/2-2s}m)} \, \|g\|_{L^2_v(\la v \ra^\ell)}  \,  \|\nabla_x h\|_{L^2_v (\la v \ra^{\gamma/2-2s}m)} \bigg) \\
			&\lesssim \|\nabla_x f\|_{H^2_xL^2_v(\la v \ra^\ell)}\, \|g\|_{L^2_xH^s_v(\la v \ra^{\gamma/2} m)} \, \|\nabla_x h\|_{L^2_xH^s_v(\la v \ra^{\gamma/2-2s} m)} \\
			&\quad +  \|\nabla_x f\|_{L^2_{x,v} (\la v \ra^{\gamma/2-2s}m)} \,\|g\|_{H^2_xL^2_v(\la v \ra^\ell)} \,   \|\nabla_x h\|_{L^2_{x,v} (\la v \ra^{\gamma/2-2s}m)} \\
			&\lesssim \|f\|_{\mathbf X} \, \|g\|_{{\bf \bar Y}} \, \|h\|_{\mathbf Y} + \|f\|_{\mathbf Y} \,  \|g\|_{\mathbf X} \, \|h\|_{\mathbf Y}.
		\eal
		$$
	\noindent {\it Step 3.} {\it Case $|\beta|=2$.} When $\beta_2=\beta$, we have
		$$
		\bal
 			&\quad \la  Q(f,\partial^\beta_xg),\partial^\beta_x h \ra_{L^2_{x,v}(\la v \ra^{-4s} m)} \\
			&\lesssim \int_{\T^3} \bigg(\|f\|_{L^2_v(\la v \ra^\ell)} \, \|\nabla^2_xg\|_{H^s_v(\la v \ra^{\gamma/2-2s}m)} \,
			\|\nabla^2_xh\|_{H^s_v(\la v \ra^{\gamma/2-4s}m)}\\
			&\quad + \|f\|_{L^2_v (\la v \ra^{\gamma/2-4s}m)} \, \|\nabla_x^2g\|_{L^2_v(\la v \ra^\ell)} \, \|\nabla^2_xh\|_{L^2_v (\la v \ra^{\gamma/2-4s}m)}\bigg) \\
			&\lesssim \|f\|_{H^2_xL^2_v(\la v \ra^\ell)} \, \|\nabla^2_xg\|_{L^2_{x}H^s_{v}(\la v \ra^{\gamma/2-2s}m)} \, \|\nabla^2_xh\|_{L^2_x H^s_v (\la v \ra^{\gamma/2-4s}m)}\\
			&\quad +   \|f\|_{H^2_x L^2_v (\la v \ra^{\gamma/2-4s}m)}\, \|\nabla^2_xg\|_{L^2_{x,v}(\la v \ra^\ell)} \,  \|\nabla^2_x h\|_{L^2_{x,v} (\la v \ra^{\gamma/2-4s}m)} \\
			&\lesssim \|f\|_{\mathbf X} \, \|g\|_{{\bf \bar Y}} \, \|h\|_{\mathbf Y} + \|f\|_{\mathbf Y} \,  \|g\|_{\mathbf X} \, \|h\|_{\mathbf Y}.
		\eal
		$$
	When $\beta_1=\beta$, we have
		$$
		\bal
			&\quad \la Q(\partial^\beta_x f, g) , \partial^\beta_x h \ra_{L^2_{x,v}(\la v \ra^{-4s} m)} \\
			&\lesssim \int_{\T^3} \bigg(\|\nabla^2_x f\|_{L^2_v(\la v \ra^\ell)} \, \|g\|_{H^s_v(\la v \ra^{\gamma/2-2s} m)} \, \|\nabla^2_x h\|_{H^s_v(\la v \ra^{\gamma/2-4s} m)} \\
			&\quad + \|\nabla^2_x f\|_{L^2_v (\la v \ra^{\gamma/2-4s}m)} \, \|g\|_{L^2_v(\la v \ra^\ell)}  \,  \|\nabla^2_x h\|_{L^2_v (\la v \ra^{\gamma/2-4s}m)} \bigg) \\
			&\lesssim \|\nabla^2_x f\|_{L^2_{x,v}(\la v \ra^\ell)}\, \|g\|_{H^{2,s}_{x,v}(\la v \ra^{\gamma/2-2s} m)} \, \|\nabla^2_x h\|_{L^2_xH^s_v(\la v \ra^{\gamma/2-4s} m)} \\
			&\quad +  \|\nabla^2_x f\|_{L^2_{x,v} (\la v \ra^{\gamma/2-4s}m)} \,\|g\|_{H^2_xL^2_v(\la v \ra^\ell)} \,   \|\nabla^2_x g\|_{L^2_{x,v} (\la v \ra^{\gamma/2-4s}m)} \\
			&\lesssim \|f\|_{\mathbf X} \, \|g\|_{{\bf \bar Y}} \, \|h\|_{\mathbf Y} + \|f\|_{\mathbf Y} \,  \|g\|_{\mathbf X} \, \|h\|_{\mathbf Y}.
		\eal
		$$
	Finally, when $|\beta_1|=|\beta_2|=1$, we obtain
		$$
		\bal
			&\quad \la Q( \partial^{\beta_1}_x f, \partial^{\beta_2}_x g ), \partial^{\beta}_x h \ra_{L^2_x L^2_v(\la v \ra^{-4s}m)} \\
			&\lesssim \int_{\T^3}\bigg(\|\nabla_x f\|_{L^2_v(\la v \ra^\ell)} \, \|\nabla_x g\|_{H^s_v(\la v \ra^{\gamma/2-2s} m)} \, \|\nabla^2_x h\|_{H^s_v(\la v \ra^{\gamma/2-4s} m)} \\
			&\quad + \|\nabla_x f\|_{L^2_v (\la v \ra^{\gamma/2-4s}m)} \, \|\nabla_xg\|_{L^2_v(\la v \ra^\ell)}  \,  \|\nabla^2_x h\|_{L^2_v (\la v \ra^{\gamma/2-4s}m)} \bigg) \\
			&\lesssim  \|\nabla_x f\|_{H^2_xL^2_v(\la v \ra^\ell)}\, \|\nabla_xg\|_{L^2_xH^s_v(\la v \ra^{\gamma/2-2s} m)} \, \|\nabla^2_x h\|_{L^2_xH^s_v(\la v \ra^{\gamma/2-4s} m)} \\
			&\quad +\|\nabla_x f\|_{L^2_{x,v} (\la v \ra^{\gamma/2-4s}m)} \,  \|\nabla_xg\|_{H^2_xL^2_v(\la v \ra^\ell)} \,   \|\nabla^2_x h\|_{L^2_{x,v} (\la v \ra^{\gamma/2-4s}m)} \\
			&\lesssim \|f\|_{\mathbf X} \, \|g\|_{{\bf \bar Y}} \, \|h\|_{\mathbf Y} + \|f\|_{\mathbf Y} \,  \|g\|_{\mathbf X} \, \|h\|_{\mathbf Y}.
		\eal
		$$
		
	\noindent {\it Step 4. Case $|\beta|=3$.} When $\beta_2=\beta$ we obtain
		$$
		\bal
 			&\quad \la  Q(f,\partial^\beta_xg),\partial^\beta_x h \ra_{L^2_{x,v}(\la v \ra^{-6s} m)} \\
			&\lesssim \int_{\T^3} \bigg(\|f\|_{L^2_v(\la v \ra^\ell)} \,
			\|\nabla^3_x g\|_{H^s_v(\la v \ra^{\gamma/2-4s} m)} \, \|\nabla^3_x h\|_{H^s_v(\la v \ra^{\gamma/2-6s} m)} \\
			&\quad + \|f\|_{L^2_v (\la v \ra^{\gamma/2-6s}m)} \, \|\nabla_x^3g\|_{L^2_v(\la v \ra^\ell)} \, \|\nabla^3_xh\|_{L^2_v (\la v \ra^{\gamma/2-6s}m)}\bigg) \\
			&\lesssim \|f\|_{H^2_xL^2_v(\la v \ra^\ell)} \, \|\nabla^3_x g\|_{L^2_xH^s_v(\la v \ra^{\gamma/2-4s} m)} \, \|\nabla^3_x h\|_{L^2_xH^s_v(\la v \ra^{\gamma/2-6s} m)} \\
			&\quad +    \|f\|_{H^2_xL^2_{v} (\la v \ra^{\gamma/2-6s}m)}\, \|\nabla^3_xg\|_{L^2_{x,v}(\la v \ra^\ell)} \,\|\nabla^3_x h\|_{L^2_{x,v} (\la v \ra^{\gamma/2-6s}m)} \\
			&\lesssim \|f\|_{\mathbf X} \, \|g\|_{{\bf \bar Y}} \, \|h\|_{\mathbf Y} + \|f\|_{\mathbf Y} \,  \|g\|_{\mathbf X} \, \|h\|_{\mathbf Y}.
		\eal
		$$
	If $|\beta_1|=1$ and $|\beta_2|=2$ then
		$$
		\bal
			&\quad \la  Q( \partial^{\beta_1}_x f, \partial^{\beta_2}_x g ), \partial^{\beta}_x h \ra_{L^2_x L^2_v( \la v \ra^{-6s}m)} \\
			&\lesssim \int_{\T^3}\bigg(\|\nabla_x f\|_{L^2_v(\la v \ra^\ell)} \, \|\nabla^2_x g\|_{H^s_v(\la v \ra^{\gamma/2-4s} m)} \, \|\nabla^3_x h\|_{H^s_v(\la v \ra^{\gamma/2-6s} m)} \\
			&\quad + \|\nabla_x f\|_{L^2_v (\la v \ra^{\gamma/2-6s}m)} \, \|\nabla^2_xg\|_{L^2_v(\la v \ra^\ell)}  \,  \|\nabla^3_x h\|_{L^2_v (\la v \ra^{\gamma/2-6s}m)} \bigg) \\
			&\lesssim \|\nabla_x f\|_{H^2_xL^2_v(\la v \ra^\ell)}\, \|\nabla^2_xg\|_{L^2_xH^s_v(\la v \ra^{\gamma/2-4s} m)} \, \|\nabla^3_x h\|_{L^2_xH^s_v(\la v \ra^{\gamma/2-6s} m)} \\
			&\quad +  \|\nabla_x f\|_{H^2_xL^2_{v} (\la v \ra^{\gamma/2-6s}m)} \, \|\nabla^2_xg\|_{L^2_{x,v}(\la v \ra^\ell)} \,   \|\nabla^3_x h\|_{L^2_{x,v} (\la v \ra^{\gamma/2-6s}m)}\\
			&\lesssim \|f\|_{\mathbf X} \, \|g\|_{{\bf \bar Y}} \, \|h\|_{\mathbf Y} + \|f\|_{\mathbf Y} \,  \|g\|_{\mathbf X} \, \|h\|_{\mathbf Y}.
		\eal
		$$
	When $|\beta_1|=2$ and $|\beta_2|=1$ then we get
		$$
		\bal
			&\quad \la  Q( \partial^{\beta_1}_x f, \partial^{\beta_2}_x g ), \partial^{\beta}_x h \ra_{L^2_x L^2_v(m \la v \ra^{-6s})} \\
			&\lesssim \int_{\T^3}\bigg(\|\nabla^2_x f\|_{L^2_v(\la v \ra^\ell)} \, \|\nabla_x g\|_{H^s_v(\la v \ra^{\gamma/2-4s} m)} \, \|\nabla^3_x h\|_{H^s_v(\la v \ra^{\gamma/2-6s} m)} \\
			&\quad + \|\nabla^2_x f\|_{L^2_v (\la v \ra^{\gamma/2-6s}m)} \, \|\nabla_xg\|_{L^2_v(\la v \ra^\ell)}  \,  \|\nabla^3_x h\|_{L^2_v (\la v \ra^{\gamma/2-6s}m)} \bigg) \\
			&\lesssim \|\nabla^2_x f\|_{L^2_{x,v}(\la v \ra^\ell)}\, \|\nabla_xg\|_{H^{2,s}_{x,v}(\la v \ra^{\gamma/2-4s} m)} \, \|\nabla^3_x h\|_{L^2_xH^s_v(\la v \ra^{\gamma/2-6s} m)} \\
			&\quad + \|\nabla^2_x f\|_{L^2_{x,v} (\la v \ra^{\gamma/2-6s}m)} \,  \|\nabla_xg\|_{H^2_xL^2_v(\la v \ra^\ell)} \,  \|\nabla^3_x h\|_{L^2_{x,v} (\la v \ra^{\gamma/2-6s}m)} \\
			&\lesssim \|f\|_{\mathbf X} \, \|g\|_{{\bf \bar Y}} \, \|h\|_{\mathbf Y} + \|f\|_{\mathbf Y} \,  \|g\|_{\mathbf X} \, \|h\|_{\mathbf Y}.
		\eal
		$$
	Finally, when $\beta_1=\beta$, it follows
		$$
		\bal
			&\quad \la  Q( \partial^{\beta}_x f, g ), \partial^{\beta}_x h \ra_{L^2_x L^2_v(m \la v \ra^{-6s})} \\
			&\lesssim \int_{\T^3} \bigg(\|\nabla^3_x f\|_{L^2_v(\la v \ra^\ell)} \, \|g\|_{H^s_v(\la v \ra^{\gamma/2-4s} m)} \, \|\nabla^3_x h\|_{H^s_v(\la v \ra^{\gamma/2-6s} m)} \\
			&\quad + \|\nabla^3_x f\|_{L^2_v (\la v \ra^{\gamma/2-6s}m)} \, \|g\|_{L^2_v(\la v \ra^\ell)}  \,  \|\nabla^3_x h\|_{L^2_v (\la v \ra^{\gamma/2-6s}m)} \bigg) \\
			&\lesssim \|\nabla^3_x f\|_{L^2_{x,v}(\la v \ra^\ell)}\, \|g\|_{H^{2,s}_{x,v}(\la v \ra^{\gamma/2-4s} m)} \, \|\nabla^3_x h\|_{L^2_xH^s_v(\la v \ra^{\gamma/2-6s} m)} \\
			&\quad +  \|\nabla^3_x f\|_{L^2_{x,v} (\la v \ra^{\gamma/2-6s}m)} \,  \|g\|_{H^2_xL^2_v(\la v \ra^\ell)} \, \|\nabla^3_x h\|_{L^2_{x,v} (\la v \ra^{\gamma/2-6s}m)} \\
			&\lesssim \|f\|_{\mathbf X} \, \|g\|_{{\bf \bar Y}} \, \|h\|_{\mathbf Y} + \|f\|_{\mathbf Y} \,  \|g\|_{\mathbf X} \, \|h\|_{\mathbf Y}.
		\eal
		$$
\smallskip

{ \noindent {\it Proof of (ii).}
	As in the proof of~(i), we write
		$$
		\la  Q(f,g),g \ra_{\HH^3_x L^2_v(m)} = \la  Q(f,g),g \ra_{L^2_{x,v}(m)}
		+ \sum_{1 \le |\beta| \le 3} \la \partial^\beta_x Q(f,g), \partial^\beta_x g \ra_{L^2_{x,v}(m \la v \ra^{-2|\beta|s})},
		$$
	and
		$$
		\partial^\beta_x Q(f,g) = \sum_{\beta_1 + \beta_2 = \beta} C_{\beta_1,\beta_2} \, Q ( \partial^{\beta_1}_x f , \partial^{\beta_2}_x g).
		$$
%	We use results coming from Lemma~\ref{lem:nonlinearhom} together with the following inequalities, that we shall use in the sequel when integrating in $x \in \T^3$,
%		\beqn \label{eq:Sob_x}
%			\| u \|_{L^\infty (\T^3_x)} \lesssim \| u \|_{H^2 (\T^3_x)},
%			\quad \| u \|_{L^6 (\T^3_x)} \lesssim \| u \|_{H^1 (\T^3_x)},
%			\quad \| u \|_{L^3 (\T^3_x)} \lesssim \|  u \|_{H^1 (\T^3_x)}^{1/2} \, \| u \|_{L^2 (\T^3_x)}^{1/2}.
%		\eeqn
In the following steps, we will always consider $\ell\in (4-\gamma+3/2,k-6s]$. Notice that since $\gamma \le 1$ and $s \le 1/2$, the condition $k>4-\gamma+3/2+6s$ implies $k>\gamma/2+3+8s$ so that  we can apply results from Lemma~\ref{lem:nonlinearhom}.

	\noindent {\it Step 1.}  Using Lemma~\ref{lem:nonlinearhom}-(ii) and~\eqref{eq:Sob_x}, we have
		$$
		\bal
 			&\quad \la  Q(f,g),g \ra_{L^2_{x,v}(m)} \\
			&\lesssim \int_{\T^3} \left(\|f\|_{L^2_v(\la v \ra^\ell)} \, \|g\|^2_{H^{s,*}_v(m)} + \|f\|_{L^2_v (\la v \ra^{\gamma/2}m)} \, \|g\|_{L^2_v(\la v \ra^\ell)} \, \|g\|_{L^2_v (\la v \ra^{\gamma/2}m)}\right) \\
			&\lesssim \|f\|_{H^2_xL^2_v(\la v \ra^\ell)} \, \|g\|^2_{L^2_{x}H^{s,*}_v(m)} +  \|f\|_{L^2_{x,v} (\la v \ra^{\gamma/2}m)}\, \|g\|_{H^2_xL^2_v(\la v \ra^\ell)} \,  \|g\|_{L^2_{x,v} (\la v \ra^{\gamma/2}m)} \\
			&\lesssim \|f\|_{\mathbf X} \, \|g\|^2_{\mathbf Y^*} + \|f\|_{\mathbf Y} \,  \|g\|_{\mathbf X} \, \|g\|_{\mathbf Y}.
		\eal
		$$
	\noindent {\it Step 2.} {\it Case $|\beta|=1$.} Arguing as in the previous step,
		$$
		\bal
 			&\quad \la  Q(f,\partial^\beta_xg),\partial^\beta_x g \ra_{L^2_{x,v}(\la v \ra^{-2s} m)} \lesssim \int_{\T^3} \bigg(\|f\|_{L^2_v(\la v \ra^\ell)} \, \|\nabla_xg\|^2_{H^{s,*}_v(\la v \ra^{-2s} m)} \\
			&\qquad + \|f\|_{L^2_v (\la v \ra^{\gamma/2-2s}m)} \, \|\nabla_xg\|_{L^2_v(\la v \ra^\ell)} \, \|\nabla_xg\|_{L^2_v (\la v \ra^{\gamma/2-2s}m)}\bigg) \\
			&\lesssim \|f\|_{H^2_xL^2_v(\la v \ra^\ell)} \, \|\nabla_xg\|^2_{L^2_xH^{s,*}_v(\la v \ra^{-2s} m)} \\
			&\qquad +\|f\|_{L^2_{x,v} (\la v \ra^{\gamma/2-2s}m)}\,   \|\nabla_xg\|_{H^2_xL^2_v(\la v \ra^\ell)} \,  \|\nabla_x g\|_{L^2_{x,v} (\la v \ra^{\gamma/2-2s}m)} \\
			&\lesssim \|f\|_{\mathbf X} \, \|g\|^2_{\mathbf Y^*} + \|f\|_{\mathbf Y} \,  \|g\|_{\mathbf X} \, \|g\|_{\mathbf Y}.
		\eal
		$$
	Moreover, we also have using Lemma~\ref{lem:nonlinearhom}-(i),
		$$
		\bal
			&\quad \la Q(\partial^\beta_x f, g) , \partial^\beta_x g \ra_{L^2_{x,v}(\la v \ra^{-2s} m)} \\
			&\lesssim \int_{\T^3} \bigg(\|\nabla_x f\|_{L^2_v(\la v \ra^\ell)} \, \|g\|_{H^s_v(\la v \ra^{\gamma/2} m)} \, \|\nabla_x g\|_{H^s_v(\la v \ra^{\gamma/2-2s} m)} \\
			&\quad + \|\nabla_x f\|_{L^2_v (\la v \ra^{\gamma/2-2s}m)} \, \|g\|_{L^2_v(\la v \ra^\ell)}  \,  \|\nabla_x g\|_{L^2_v (\la v \ra^{\gamma/2-2s}m)} \bigg) \\
			&\lesssim \|\nabla_x f\|_{H^2_xL^2_v(\la v \ra^\ell)}\, \|g\|_{L^2_xH^s_v(\la v \ra^{\gamma/2} m)} \, \|\nabla_x g\|_{L^2_xH^s_v(\la v \ra^{\gamma/2-2s} m)} \\
			&\quad + \|\nabla_x f\|_{L^2_{x,v} (\la v \ra^{\gamma/2-2s}m)} \, \|g\|_{H^2_xL^2_v(\la v \ra^\ell)} \,   \|\nabla_x g\|_{L^2_{x,v} (\la v \ra^{\gamma/2-2s}m)} \\
			&\lesssim \|f\|_{\mathbf X} \, \|g\|^2_{\mathbf Y} + \|f\|_{\mathbf Y} \,  \|g\|_{\mathbf X} \, \|g\|_{\mathbf Y}.
		\eal
		$$
	\noindent {\it Step 3.} {\it Case $|\beta|=2$.} When $\beta_2=\beta$, we have
		$$
		\bal
 			&\quad \la  Q(f,\partial^\beta_xg),\partial^\beta_x g \ra_{L^2_{x,v}(\la v \ra^{-4s} m)} \lesssim \int_{\T^3} \bigg(\|f\|_{L^2_v(\la v \ra^\ell)} \, \|\nabla^2_xg\|^2_{H^{s,*}_v(\la v \ra^{-4s} m)} \\
			&\qquad + \|f\|_{L^2_v (\la v \ra^{\gamma/2-4s}m)} \, \|\nabla_x^2g\|_{L^2_v(\la v \ra^\ell)} \, \|\nabla^2_xg\|_{L^2_v (\la v \ra^{\gamma/2-4s}m)}\bigg) \\
			&\lesssim \|f\|_{H^2_xL^2_v(\la v \ra^\ell)} \, \|\nabla^2_xg\|^2_{L^2_xH^{s,*}_v(\la v \ra^{-4s} m)} \\
			&\qquad +    \|f\|_{H^2_{x}L^2_v (\la v \ra^{\gamma/2-4s}m)}\,\|\nabla^2_xg\|_{L^2_{x,v}(\la v \ra^\ell)} \, \|\nabla^2_x g\|_{L^2_{x,v} (\la v \ra^{\gamma/2-4s}m)} \\
			&\lesssim \|f\|_{\mathbf X} \, \|g\|^2_{\mathbf Y^*} +  \|f\|_{\mathbf Y} \,  \|g\|_{\mathbf X} \, \|g\|_{\mathbf Y}.
		\eal
		$$
	When $\beta_1=\beta$, we have
		$$
		\bal
			&\quad \la Q(\partial^\beta_x f, g) , \partial^\beta_x g \ra_{L^2_{x,v}(\la v \ra^{-4s} m)} \\
			&\lesssim \int_{\T^3} \bigg(\|\nabla^2_x f\|_{L^2_v(\la v \ra^\ell)} \, \|g\|_{H^s_v(\la v \ra^{\gamma/2-2s} m)} \, \|\nabla^2_x g\|_{H^s_v(\la v \ra^{\gamma/2-4s} m)} \\
			&\quad + \|\nabla^2_x f\|_{L^2_v (\la v \ra^{\gamma/2-4s}m)} \, \|g\|_{L^2_v(\la v \ra^\ell)}  \,  \|\nabla^2_x g\|_{L^2_v (\la v \ra^{\gamma/2-4s}m)} \bigg) \\
			&\lesssim \|\nabla^2_x f\|_{L^6_xL^2_v(\la v \ra^\ell)}\, \|g\|_{L^3_x H^s_v(\la v \ra^{\gamma/2-2s} m)} \, \|\nabla^2_x g\|_{L^2_x H^s_v(\la v \ra^{\gamma/2-4s} m)} \\
			&\quad +   \|\nabla^2_x f\|_{L^2_{x,v} (\la v \ra^{\gamma/2-4s}m)} \,\|g\|_{H^2_xL^2_v(\la v \ra^\ell)} \,  \|\nabla^2_x g\|_{L^2_{x,v} (\la v \ra^{\gamma/2-4s}m)} \\
			&\lesssim \|\nabla^2_x f\|_{H^1_xL^2_v(\la v \ra^\ell)}\, \|g\|_{H^{1,s}_{x,v}(\la v \ra^{\gamma/2-2s} m)}
			\, \|\nabla^2_x g\|_{L^2_xH^s_v(\la v \ra^{\gamma/2-4s} m)} \\
			&\quad + \|\nabla^2_x f\|_{L^2_{x,v} (\la v \ra^{\gamma/2-4s}m)} \,  \|g\|_{H^2_xL^2_v(\la v \ra^\ell)} \,  \|\nabla^2_x g\|_{L^2_{x,v} (\la v \ra^{\gamma/2-4s}m)} \\
			&\lesssim \|f\|_{\mathbf X} \, \|g\|^2_{\mathbf Y} + \|f\|_{\mathbf Y} \,  \|g\|_{\mathbf X} \, \|g\|_{\mathbf Y}.
		\eal
		$$
	Finally, when $|\beta_1|=|\beta_2|=1$, we obtain
		$$
		\bal
			&\quad \la Q( \partial^{\beta_1}_x f, \partial^{\beta_2}_x g ), \partial^{\beta}_x g \ra_{L^2_x L^2_v(\la v \ra^{-4s}m)} \\
			&\lesssim \int_{\T^3}\bigg(\|\nabla_x f\|_{L^2_v(\la v \ra^\ell)} \, \|\nabla_x g\|_{H^s_v(\la v \ra^{\gamma/2-2s} m)} \, \|\nabla^2_x g\|_{H^s_v(\la v \ra^{\gamma/2-4s} m)} \\
			&\quad + \|\nabla_x f\|_{L^2_v (\la v \ra^{\gamma/2-4s}m)} \, \|\nabla_xg\|_{L^2_v(\la v \ra^\ell)}  \,  \|\nabla^2_x g\|_{L^2_v (\la v \ra^{\gamma/2-4s}m)} \bigg) \\
			&\lesssim \|\nabla_x f\|_{H^2_xL^2_v(\la v \ra^\ell)}\, \|\nabla_xg\|_{L^2_x H^s_v(\la v \ra^{\gamma/2-2s} m)} \, \|\nabla^2_x g\|_{L^2_x H^s_v(\la v \ra^{\gamma/2-4s} m)} \\
			&\quad + \|\nabla_x f\|_{L^2_{x,v} (\la v \ra^{\gamma/2-4s}m)} \,   \|\nabla_xg\|_{H^2_xL^2_v(\la v \ra^\ell)} \, \|\nabla^2_x g\|_{L^2_{x,v} (\la v \ra^{\gamma/2-4s}m)} \\
			&\lesssim \|f\|_{\mathbf X} \, \|g\|^2_{\mathbf Y} + \|f\|_{\mathbf Y} \,  \|g\|_{\mathbf X} \, \|g\|_{\mathbf Y}.
		\eal
		$$
		
	\noindent {\it Step 4. Case $|\beta|=3$.} When $\beta_2=\beta$ we obtain
		$$
		\bal
 			&\quad \la  Q(f,\partial^\beta_xg),\partial^\beta_x g \ra_{L^2_{x,v}(\la v \ra^{-6s} m)} \lesssim \int_{\T^3} \bigg(\|f\|_{L^2_v(\la v \ra^\ell)} \, \|\nabla^3_xg\|^2_{H^{s,*}_v(\la v \ra^{-6s}m)} \\
			&\qquad + \|f\|_{L^2_v (\la v \ra^{\gamma/2-6s}m)} \, \|\nabla_x^3g\|_{L^2_v(\la v \ra^\ell)} \, \|\nabla^3_xg\|_{L^2_v (\la v \ra^{\gamma/2-6s}m)}\bigg) \\
			&\lesssim \|f\|_{H^2_xL^2_v(\la v \ra^\ell)} \, \|\nabla^3_xg\|^2_{L^2_xH^{s,*}_v(\la v \ra^{-6s}m)} \\
			&\qquad +  \|f\|_{H^2_{x} L^2_{v} (\la v \ra^{\gamma/2-6s}m)}\, \|\nabla^3_xg\|_{L^2_{x,v}(\la v \ra^\ell)} \,  \|\nabla^3_x g\|_{L^2_{x,v} (\la v \ra^{\gamma/2-6s}m)} \\
			&\lesssim \|f\|_{\mathbf X} \, \|g\|^2_{\mathbf Y^*} + \|f\|_{\mathbf Y} \,  \|g\|_{\mathbf X} \, \|g\|_{\mathbf Y}.
		\eal
		$$
	If $|\beta_1|=1$ and $|\beta_2|=2$ then
		$$
		\bal
			&\quad \la  Q( \partial^{\beta_1}_x f, \partial^{\beta_2}_x g ), \partial^{\beta}_x g \ra_{L^2_x L^2_v(m \la v \ra^{-6s})} \\
			&\lesssim \int_{\T^3}\bigg(\|\nabla_x f\|_{L^2_v(\la v \ra^\ell)} \, \|\nabla^2_x g\|_{H^s_v(\la v \ra^{\gamma/2-4s} m)} \, \|\nabla^3_x g\|_{H^s_v(\la v \ra^{\gamma/2-6s} m)} \\
			&\quad + \|\nabla_x f\|_{L^2_v (\la v \ra^{\gamma/2-6s}m)} \, \|\nabla^2_xg\|_{L^2_v(\la v \ra^\ell)}  \,  \|\nabla^3_x g\|_{L^2_v (\la v \ra^{\gamma/2-6s}m)} \bigg) \\
			&\lesssim \|\nabla_x f\|_{H^2_xL^2_v(\la v \ra^\ell)}\, \|\nabla^2_xg\|_{L^2_xH^s_v(\la v \ra^{\gamma/2-4s} m)} \, \|\nabla^3_x g\|_{L^2_xH^s_v(\la v \ra^{\gamma/2-6s} m)} \\
			&\quad +  \|\nabla_x f\|_{H^2_xL^2_{v} (\la v \ra^{\gamma/2-6s}m)} \,  \|\nabla^2_xg\|_{L^2_{x,v}(\la v \ra^\ell)} \, \|\nabla^3_x g\|_{L^2_{x,v} (\la v \ra^{\gamma/2-6s}m)}\\
%			&\lesssim \|\nabla_x f\|_{H^2_xL^2_v(\la v \ra^\ell)}\, \|\nabla^2_xg\|_{L^2_xH^s_v(\la v \ra^{\gamma/2-4s} m)} \, \|\nabla^3_x g\|_{L^2_xH^s_v(\la v \ra^{\gamma/2-6s} m)} \\
%			&\quad +  \|\nabla_x f\|_{H^2_xL^2_{v} (\la v \ra^{\gamma/2-6s}m)}  \,
% \|\nabla^2_xg\|_{L^2_{x,v}(\la v \ra^\ell)} \,			\|\nabla^3_x g\|_{L^2_{x,v} (\la v \ra^{\gamma/2-6s}m)}\\
			&\lesssim \|f\|_{\mathbf X} \, \|g\|^2_{\mathbf Y} + \|f\|_{\mathbf Y} \,  \|g\|_{\mathbf X} \, \|g\|_{\mathbf Y}.
		\eal
		$$
	When $|\beta_1|=2$ and $|\beta_2|=1$, we get
		$$
		\bal
			&\quad \la  Q( \partial^{\beta_1}_x f, \partial^{\beta_2}_x g ), \partial^{\beta}_x g \ra_{L^2_x L^2_v(m \la v \ra^{-6s})} \\
			&\lesssim \int_{\T^3}\bigg(\|\nabla^2_x f\|_{L^2_v(\la v \ra^\ell)} \, \|\nabla_x g\|_{H^s_v(\la v \ra^{\gamma/2-4s} m)} \, \|\nabla^3_x g\|_{H^s_v(\la v \ra^{\gamma/2-6s} m)} \\
			&\quad + \|\nabla^2_x f\|_{L^2_v (\la v \ra^{\gamma/2-6s}m)} \, \|\nabla_xg\|_{L^2_v(\la v \ra^\ell)}  \,  \|\nabla^3_x g\|_{L^2_v (\la v \ra^{\gamma/2-6s}m)} \bigg) \\
			&\lesssim \|\nabla^2_x f\|_{L^6_xL^2_v(\la v \ra^\ell)}\, \|\nabla_xg\|_{L^3_xH^s_v(\la v \ra^{\gamma/2-4s} m)} \, \|\nabla^3_x g\|_{L^2_xH^s_v(\la v \ra^{\gamma/2-6s} m)} \\
			&\quad + \|\nabla^2_x f\|_{L^2_{x,v} (\la v \ra^{\gamma/2-6s}m)} \,  \|\nabla_xg\|_{H^2_xL^2_v(\la v \ra^\ell)} \,   \|\nabla^3_x g\|_{L^2_{x,v} (\la v \ra^{\gamma/2-6s}m)} \\
			&\lesssim \|\nabla^2_x f\|_{H^1_xL^2_v(\la v \ra^\ell)}\, \|\nabla_xg\|_{H^{1,s}_{x,v}(\la v \ra^{\gamma/2-4s} m)}  \,
			 \|\nabla^3_x g\|_{L^2_xH^s_v(\la v \ra^{\gamma/2-6s} m)} \\
			&\quad + \|\nabla^2_x f\|_{L^2_{x,v} (\la v \ra^{\gamma/2-6s}m)} \, \|\nabla_xg\|_{H^2_xL^2_v(\la v \ra^\ell)} \,   \|\nabla^3_x g\|_{L^2_{x,v} (\la v \ra^{\gamma/2-6s}m)} \\
			&\lesssim \|f\|_{\mathbf X} \, \|g\|^2_{\mathbf Y} + \|f\|_{\mathbf Y} \,  \|g\|_{\mathbf X} \, \|g\|_{\mathbf Y}.
		\eal
		$$
	Finally, when $\beta_1=\beta$, it follows
		$$
		\bal
			&\quad \la  Q( \partial^{\beta}_x f, g ), \partial^{\beta}_x g \ra_{L^2_x L^2_v(m \la v \ra^{-6s})} \\
			&\lesssim \int_{\T^3} \bigg(\|\nabla^3_x f\|_{L^2_v(\la v \ra^\ell)} \, \|g\|_{H^s_v(\la v \ra^{\gamma/2-4s} m)} \, \|\nabla^3_x g\|_{H^s_v(\la v \ra^{\gamma/2-6s} m)} \\
			&\quad + \|\nabla^3_x f\|_{L^2_v (\la v \ra^{\gamma/2-6s}m)} \, \|g\|_{L^2_v(\la v \ra^\ell)}  \,  \|\nabla^2_x g\|_{L^2_v (\la v \ra^{\gamma/2-6s}m)} \bigg) \\
			&\lesssim \|\nabla^3_x f\|_{L^2_{x,v}(\la v \ra^\ell)}\, \|g\|_{H^{2,s}_{x,v}(\la v \ra^{\gamma/2-4s} m))} \, \|\nabla^3_x g\|_{L^2_xH^s_v(\la v \ra^{\gamma/2-6s} m)} \\
			&\quad +  \|\nabla^3_x f\|_{L^2_{x,v} (\la v \ra^{\gamma/2-6s}m)} \, \|g\|_{H^2_xL^2_v(\la v \ra^\ell)} \,  \|\nabla^3_x g\|_{L^2_{x,v} (\la v \ra^{\gamma/2-6s}m)} \\
			&\lesssim \|f\|_{\mathbf X} \, \|g\|^2_{\mathbf Y} + \|f\|_{\mathbf Y} \,  \|g\|_{\mathbf X} \, \|g\|_{\mathbf Y}.
		\eal
		$$
	We conclude noticing that $\|g\|^2_{\mathbf Y} \lesssim \|g\|^2_{\mathbf Y^*}$ from Lemma~\ref{lem:anis}.

\noindent {\it Proof of (iii).} The result is immediate from {\it(ii)} and the fact that $\|f\|^2_{\mathbf Y} \lesssim \|f\|^2_{\mathbf Y^*}$. }

\bigskip
%%%%%%%%%%%%%%%%%%%%%%%%%%%%%%%%%%%%%%%%%%%%%%%%%%%%%%%%%%%%%%%%%%%%%%%%%%%%%%%%%%%%%%
\section{{Cancellation lemma} and Carleman representation} \label{app:canc+Carl}
\setcounter{equation}{0}
\setcounter{theo}{0}
%%%%%%%%%%%%%%%%%%%%%%%%%%%%%%%%%%%%%%%%%%%%%%%%%%%%%%%%%%%%%%%%%%%%%%%%%%%%%%%%%%%%%%

{ We  state here two classical tools in the analysis of Boltzmann operator, the cancellation lemma and the Carleman representation. The cancellation lemma comes from~\cite[Lemma~1]{ADVW-ARMA}, we here state it for the kernel $B(v-v_*, \sigma) = b(\cos \theta) \, |v-v_*|^\gamma$ but it can be generalized to other kernels very easily (for example, we us it with $B_\delta(v-v_*, \sigma) = b_\delta(\cos \theta) \, |v-v_*|^\gamma$ in Subsections~\ref{subsec:hom} and~\ref{subsec:dissip} of with $\widetilde B_1 (v-v_*,\sigma) = \chi(|v'-v|) \, b(\cos \theta) \, |v-v_*|^\gamma$ in Subsection~\ref{subsec:estimlin}).
\begin{lem}[Cancellation lemma]
Let $f$ be a measurable function defined on $\R^3$. For almost every $v \in \R^3$, we have:
$$
\int_{\R^3 \times \Sp^2}  B(v-v_*,\sigma) (f'_*-f_*) \, \dv_* \, \d\sigma = (f *S)(v)
$$
where
$$
S(z) := 2 \pi \, \int_0^{\pi/2} \sin\theta \, b(\cos \theta) \left( \frac{|z|^\gamma}{\cos^{\gamma+3} (\theta/2)} - |z|^\gamma \right) \, \d\theta .
$$
\end{lem}}

For the Carleman representation, we refer to~\cite{AHL*} for more details on the version that we state here.
\begin{lem}[Carleman representation] \label{lem:Carleman}
	Let $F$ be a measurable function defined on~$(\R^3)^4$. For any vector $\vartheta \in \R^3$, we denote by $E_{0,\vartheta}$ the (hyper)vector plane orthogonal to $\vartheta$. Then, when all sides are well defined, we have the following equality :
		\begin{align*}
		&\int_{\R^3 \times \S^2} b(\cos \theta) |v-v_*|^\gamma F(v,v_*,v',v'_*) \, \dv_* \, \d\sigma \\
		&\quad = \int_{\R^3_\vartheta} \d \vartheta \int_{E_{0,\vartheta}} \d\alpha \, \tilde{b}(\alpha,\vartheta)\,\mathds{1}_{|\alpha| \ge |\vartheta|} \frac{|\alpha+\vartheta|^{\gamma+1+2s} }{|\vartheta|^{3+2s}} \,F(v,v+\alpha-\vartheta,v-\vartheta,v+\alpha)
		\end{align*}
	where $\tilde{b}(\alpha,\vartheta)$ is bounded from above and below by positive constants and $\tilde{b}(\alpha,\vartheta)=\tilde{b}(\pm\alpha,\pm \vartheta)$.
\end{lem}

\bigskip
%%%%%%%%%%%%%%%%%%%%%%%%%%%%%%%%%%%%%%%%%%%%%%%%%%%%%%%%%%%%%%%%%%%%%%%%%%%%%%%%%%%%%%
\section{Pseudodifferential tools} \label{app:pseudo}
\setcounter{equation}{0}
\setcounter{theo}{0}
%%%%%%%%%%%%%%%%%%%%%%%%%%%%%%%%%%%%%%%%%%%%%%%%%%%%%%%%%%%%%%%%%%%%%%%%%%%%%%%%%%%%%%
\subsection{Pseudodifferential calculus} \label{pseud}

%{La notation $\sigma$ n'est peut-\^etre pas tr\`es opportune dans ce qui suit ?}

We first recall the definitions of the quantizations we shall use in the following.
Let us consider a temperate symbol~${\sigma} \in \SS$, we define its standard quantization~${\sigma}(v,D_v)$ for ~$f\in L^2(\R^d)$ by
$$
{\sigma}(v,D_v) f (v) := \frac{1}{(2\pi)^d} \int e^{iv \cdot\eta} {\sigma}(v,\eta) \hat{f}(\eta) \, \d \eta.
$$
The Weyl quantization is defined by
$$
{\sigma}^w f (v) := \frac{1}{(2\pi)^d} \iint e^{i(v-w) \cdot\eta} {\sigma}\left(\frac{v+w}{2},\eta\right) f(w) \, \d \eta \, \d w.
$$
%{\Rd Est-ce que ce ne sont pas des $+$ plutt que des $-$ dans les deux formules pr\'ec\'edentes habituellement ?}
We recall that  for two symbols ${\sigma}$ and $\tau$ we have
\begin{equation} \label{defsharp}
{{\sigma}^w \tau^w = ({\sigma} \sharp \tau)^w, \quad {\sigma} \sharp \tau = {\sigma}\tau + \int_0^1 (\D_\eta {\sigma} \sharp_\theta \D_v \tau - \D_v {\sigma} \sharp_\theta \D_\eta \tau ) \, \d \theta}
\end{equation}
where for $V=(v,\eta)$ we have $\sharp = \sharp_1$ and for $\theta \in (0,1]$,
$$
{\sigma} \sharp_\theta \tau (V):= \frac{1}{2i} \iint e^{-2i [V-V_1,V-V_2]/\theta} {\sigma}(V_1) \tau(V_2)\, \d V_1 \, \d V_2 /(\pi \theta)^d
$$
with $[V_1,V_2] = v_2\cdot\eta_1 - v_1\cdot\eta_2$ the canonical symplectic form on $\R^{2d}$.
We shall also use the Wick quantization, which has very nice properties concerning positivity of operators (see~\cite{Lerner-CME,LernerBook1,LernerBook2} for more details on the subject). For this, we first introduce the Gaussian in  phase variables
\begin{equation} \label{wickN}
N(v,\eta) := (2\pi)^{-d} \exp^{-(|v|^2+|\eta|^2)/2}.
\end{equation}
The Wick quantization is then defined by
\begin{equation} \label{wickN2}
{\sigma}^\wick f(v) := ({\sigma} \star N )^w f(v),
\end{equation}
where $\star$ denotes the usual convolution in $(v,\eta)$ variables.
Recall that one of the main property of Wick quantization is its positivity:
\begin{equation} \label{wick2}
\forall \, (v,\eta) \in \R^6, \; {\sigma}(v,\eta)\geq 0 \Rightarrow {\sigma}^{\rm Wick}\geq 0,
\end{equation}
and that the following relation holds (see e.g. \cite[Proposition 3.4]{Lerner-CME}):
\begin{equation}\label{wick}
[ g^{\rm Wick},iv\cdot \xi]=\{  g, v\cdot \xi\}^{\rm Wick}.
\end{equation}
The previous definitions extend to symbols in $\SS'$ by duality.

%----------------------------------%----------------------------------%----------------------------------%----------------------------------%----------------------------------%----------------------------------%----------------------------------%----------------------------------%
\medskip
\subsection{The weak semiclassical class $S_K(g)$}

 Let $\Gamma := |\dv|^2 + |\d \eta|^2$ be the flat metric on~$\R^6_{v,\eta}$. The first point is to verify that the introduced symbols and weights are indeed   in a suitable symbolic calculus with large parameter $K$ uniformly in the parameter $\xi$.
For this, we first recall  that a weight $1 \leq g$ is said to be temperate with respect to $\Gamma$ if there exist $N \geq 1$ and $C_{N}$ such that  for all $(v, \eta)$, $(v',\eta') \in \R^6$
$$
g(v',\eta') \leq C_N \, g(v,\eta) ( 1 + |v'-v| + |\eta'-\eta|)^N
$$
%{\color{red} Il y avait un carr\'e en trop ? Et dans la suite admissible = temperate ?}
We now introduce adapted classes of symbols.

\begin{defin}
Let  $g$ be a temperate weight.  We denote by $S(g)$ the symbol class of all smooth functions ${\sigma}(v,\eta)$ (possibly depending on parameters $K$ and $\xi$) such that
$$
\abs{\D_v^\alpha \D_\eta^\beta {\sigma}(v,\eta) } \leq C_{\alpha,\beta} g(v,\eta)
$$
where for any multiindex $\alpha$ and $\beta$, $C_{\alpha,\beta}$ is uniform in $K$ and $\xi$. We denote also $S_K(g)$ the symbol class of all smooth functions ${\sigma}(v,\eta)$ (possibly depending on $K$ and $\xi$ again) such that
$$
\abs{{\sigma}} \leq C_{0,0} g \quad \text{and} \quad \forall \, |\beta| \geq 1, \quad \abs{\D_v^\alpha \D_\eta^\beta {\sigma} } \leq
 C_{\alpha,\beta} K^{-1/2} g
$$
uniformly in $K$ and $\xi$. Note that $S_K(g) \subset S(g)$ and that these definitions are with respect to the flat metric.

Eventually, we shall say that a symbol ${\sigma}$ is elliptic positive in $S(g)$ or $S_K(g)$ if in addition~${\sigma} \geq 1$ and if there exists a constant $C$ uniform in parameters such that we have~$C^{-1} g \leq {\sigma} \leq C g$.
\end{defin}
Before focusing on the class $S_K(g)$, we first recall one of the main results concerning the class without parameter (and without weight) $S(1)$:

\begin{lem}[Calderon Vaillancourt Theorem]
Let ${\sigma} \in S(1)$, then ${\sigma}^w$ is a bounded operator with norm depending only on a finite number of semi-norms of ${\sigma}$ in $S(1)$.
\end{lem}
The classes $S_K$ and $S$ have standard internal properties:

\begin{lem} \label{stab} For $K$ sufficiently large, we have the following:
\begin{enumerate}[leftmargin=1.1cm]
\item[{\it (i)}] Let $g$ be a temperate weight and consider ${\sigma}$ an elliptic positive symbol in $S_K(g)$ then for all $\nu \in \R$, ${\sigma}^\nu \in S_K(g^\nu)$;
\item[{\it (ii)}] Let $g$, $h$ be temperate weights and consider ${\sigma}$ in $S_K(g)$, $\tau$ in $S_K(h)$, then ${\sigma}\tau$ is in~$ S_K(g h)$.
\end{enumerate}
\end{lem}

\begin{proof} For point \it a)\rm, just notice that if ${\sigma}$ is an elliptic positive symbol in $S_K(g)$, then ${\sigma} \simeq g$ so that ${\sigma}^\nu \simeq g^\nu$. We also have directly for $\beta$ a multiindex of length 1
$$
\abs{ \D_\eta^\beta {\sigma}^\nu } = |\nu| {\sigma}^{\nu-1} \abs{\D_\eta^\beta {\sigma}} \leq C g^{\nu-1} K^{-1/2} g = C K^{-1/2} g^\nu
$$
using ${\sigma} \simeq g$. Estimates on  higher order derivatives are straightforward.

For point \it b)\rm, the computation is also straightforward using the Leibniz rule.
\end{proof}

Now we can quantize the previously introduced symbols. The main \it semiclassical \rm idea behind the introduction of the class $S_K$ for $K$ large is that invertibility and powers of operators associated to symbols are direct consequences of similar properties of symbols, essentially independently of the quantization.

We first check that the class $S_K$ is essentially stable by change of quantization.

\begin{lem} \label{allsk}
Let $g$ be a temperate weight and consider $\tilde{{\sigma}}$  a positive elliptic symbol in~$S_K(g)$. We denote ${{\sigma}}$ the Weyl symbol of the operator  $\tilde{{\sigma}}(v,D_v)$ so that
${\sigma}^w = \tilde{{\sigma}}(v,D_v)$  and recall that the Weyl symbol of ${\sigma}^\wick$ is ${\sigma} \star N$.  Then
%\begin{equation} \label{equivquant}
%{\sigma} \equiv \tilde{{\sigma}} \equiv {\sigma}\star N \qquad modulo \ \ K^{-1/2} S(g).
%\end{equation}
 ${\sigma}$ and  $ {\sigma}\star N$ are both in $S_K(g)$. If in addition $\tilde{{\sigma}}$ is elliptic positive, then $\Re {\sigma}$ and $\Re {\sigma}\star N$ are elliptic positive.
\end{lem}

\begin{proof}
We first prove the result for ${\sigma}$ supposing that $\tilde{{\sigma}}$ is elliptic positive. From for e.g.~\cite{LernerBook2} and an adaptation of Lemma 4.4 in~\cite{AHL*}, we know that
\begin{equation} \label{diff}
{\sigma} - \tilde{{\sigma}} \in K^{-1/2} S(g).
\end{equation}
Since $K^{-1/2} S(g) \subset S_K(g)$, this gives that ${{\sigma}} \in S_K(g)$.
If in addition $\tilde{{\sigma}}$ is elliptic positive, then let us prove that   $\Re {\sigma}$ also is.  There exist constants $C$, $C'$ uniform in $K$ large such that
$$
C^{-1} g - C'K^{-1/2} g \leq \Re {{\sigma}} \leq  C g + C'K^{-1/2} g
$$
if $C^{-1} g \leq {\sigma}  \leq  C g$. Taking $K$ sufficiently large then gives the result.

We now deal with ${\sigma}\star N$, supposing that ${\sigma}$  is in $S_K(g)$. For $V=(v,\eta)$ we have
$$
{\sigma}\star N (V)
 = \iint {\sigma}(V-W) N(W) \d W
 $$
 and using the temperance property of $g$, we get uniformly in all other possible parameters (including $K$)
 $$
 \abs{{\sigma}\star N(V)} \leq \iint C g(V) (1+|W|)^N N(W) \, \d W \leq C' g(V).
 $$
 %{\Rd il me semble qu'on obtient la deuxi\`eme in\'egalit\'e directement ? Pourquoi a-t-on besoin de l'in\'egalit\'e interm\'ediaire ?}
 For the derivatives, we get similarly for multiindex $\alpha$ and $\beta$ with $|\beta| \geq 1$
 \begin{equation}
 \begin{split}
 \abs{\D^\alpha_v \D^\beta_\eta {\sigma}\star N (V)}
 &  \leq  \iint \abs{\D^\alpha_v \D^\beta_\eta {\sigma}(V-W)} N(W) \, \d W\\
 & \leq C K^{-1/2} \iint g(V-W) N(W) \, \d W \\
 & \leq C' K^{-1/2} \iint g(V) (1+|W|)^N N(W) \, \d W \\
 & \leq C'' K^{-1/2} g(V).
 \end{split}
 \end{equation}
Suppose now that in addition
 $\tilde{{\sigma}}$  is elliptic positive, then $\Re {\sigma}$ is {elliptic positive} and $C^{-1} g(V) \leq \Re {\sigma}(V) \leq C g(V)$ for a constant $C>0$. Since $\Re {\sigma}\star N$ is positive, this implies with the temperance of $g$ that
\begin{multline}
c' g(V) { \le} \iint C^{-1} C_N^{-1} g(V) (1+|W|)^{-N} N(W) \d W \\ \leq \Re {\sigma}\star N(V)  \leq \iint C C_N g(V) (1+|W|)^N N(W) \d W = C' g(V)
\end{multline}
for some positive constants $c'$ and $C'$, so that $\Re {\sigma}\star N$ is indeed elliptic positive.
 \end{proof}

\begin{rem} \label{allskrem}
Note that using exactly the same argument as in the proof before, we also get that if $\tau$ is a given  elliptic positive symbol in $S_K(g)$, with $g$ a temperate weight,  then~$\tau \star N$  is also an elliptic positive symbol in $S_K(g)$.
\end{rem}

\bigskip
The next technical lemma is also proven in~\cite{AHL*}:

\begin{lem}[Lemma 4.2 in \cite{AHL*}] \label{AHLinverse} Let $g$ be a temperate weight and ${\sigma} \in S_K(g)$. Then for~$K$ sufficiently large (depending on a finite number of semi-norms of ${\sigma}$), the operator~${\sigma}^w$ is invertible and there exists $H_L$ and $H_R$ bounded invertible operators that are close to identity as well as their inverse such that
$$
({\sigma}^w)^{-1} = H_L ({\sigma}^{-1})^w = ({\sigma}^{-1})^w H_R.
$$
The norms of operators $H_L$ and
$H_R$ and their inverse can be bounded uniformly in parameters (including $K$).
\end{lem}
Note that by ``close to  identity uniformly in parameters'', we mean that
$$
\norm{H_L f} \simeq \norm{H_R f} \simeq \norm{f}.
$$
with constants uniform in parameters (including $K$ sufficiently large).

\begin{proof} The proof follows exactly the lines of the one given in \cite[Lemma 4.2. i)]{AHL*}.
\end{proof}

We now give the main Proposition that will be used in the proof of the technical Lemmas in Subsection \ref{subsec:technical}.

\begin{prop} \label{mainpseudo} Let $g$ be a temperate  weight and consider ${\sigma}$ an elliptic positive symbol in $S_K(g)$. Then for $K$ sufficiently large, we have the following
\begin{equation} \label{maingamma}
 \norm{({\sigma}^w)^{1/2} f} \simeq  \norm{({\sigma}^{1/2})^w f} \qquad \textrm{ and } \qquad
  \norm{({\sigma}^w)^{-1} f} \simeq  \norm{({\sigma}^{-1})^w f}.
\end{equation}
In addition, suppose that $\tau$ is another elliptic positive symbol in $S_K(g)$ then
\begin{equation} \label{mainequiv}
\norm{{\sigma}^w f} \simeq \norm{\tau^w f}.
\end{equation}
In particular, we have
\begin{equation} \label{mainnorm}
\norm{{\sigma}^w f}^2 %\simeq \norm{{\sigma}(v,D_v) f}^2
\simeq \norm{{\sigma}^\wick f}^2 \simeq \sep{ ({\sigma}^2)^\wick f,f}
\end{equation}
and
\begin{equation} \label{mainscalar}
\sep{{\sigma}^w f,f} %\simeq \Re \sep{{\sigma}(v,D_v) f,f}
\simeq \sep{{\sigma}^\wick f,f}
\end{equation}
uniformly in parameters (in particular $K$).
%In addition,
%for all $\gamma \in \R$, we have
%\begin{equation} \label{maingamma}
% \norm{({\sigma}^\wick)^\gamma f} \simeq  \norm{({\sigma}^\gamma)^\wick f}.
%\end{equation}
%In particular if ${\sigma}$  a elliptic positive symbol in $S_K(1)$ (uniformly in $K$ again) then
%\begin{equation} \label{mainidentity}
%\norm{{\sigma}^\wick f}^2 \simeq \norm{f}^2
%\end{equation}
\end{prop}

\begin{proof}
We first prove \eqref{maingamma}. For the second almost equality, we just have to notice that from Lemma \ref{AHLinverse}, we have
$$
 \norm{({\sigma}^w)^{-1} f} =  \norm{H_L({\sigma}^{-1})^w f} \simeq \norm{({\sigma}^{-1})^w f}
 $$
 since $H_L$ is close to identity (uniformly in parameters). For the first part of~\eqref{maingamma}, we write that
\begin{equation} \label{eqnormprod}
\begin{split}
\norm{{\sigma}^w f}^2 & =  ( ({\sigma} \sharp {\sigma})^w f,f) =
 ( ({\sigma}^2)^w f,f) + ( r^w f,f)
\end{split}
\end{equation}
where $r = {\sigma}\sharp {\sigma} - {\sigma}^2 \in K^{-1/2} S(g^2)$ by standard symbolic calculus. More precisely, we can write from \eqref{defsharp}
$$
r = \int_0^1 (\D_v {\sigma} \sharp_\theta \D_\eta{\sigma}  -\D_\eta{\sigma} \sharp_\theta \D_v {\sigma}) \, \d\theta
$$
and using that $\D_v {\sigma} \in S(g)$ and $\D_\eta {\sigma} \in K^{-1/2} S(g)$ gives the result by stability of the flat symbol class $S(g)$. We therefore get that
\begin{equation*}
\begin{split}
 |( r^w f,f)| & = \big| \big(  ({\sigma}^w)^{-1}r^w ({\sigma}^w)^{-1} {\sigma}^w f, {\sigma}^w f ) \big| \\
 & = \big| \big( H_L ({\sigma}^{-1})^w r^w ({\sigma}^{-1})^w H_R {\sigma}^w f, {\sigma}^w f ) \big|.
 \end{split}
 \end{equation*}
Now ${\sigma}^{-1} \sharp r \sharp {\sigma}^{-1} \in K^{-1/2} S(1)$ since ${\sigma}^{-1} \in S(g)$, so that $({\sigma}^{-1})^w r^w ({\sigma}^{-1})^w$ is a bounded operator with norm controlled by a constant times $K^{-1/2}$. Since $H_L$ and $H_R$ are bounded operators independently of $K$,  there exists a constant such that
 $$
 |( r^w f,f)| \leq C K^{-1/2}\norm{{\sigma}^w f}^2.
 $$
This estimate and \eqref{eqnormprod}, gives that for $K$ sufficiently large,
 \begin{equation} \label{eqnormprodbis}
\begin{split}
\frac{1}{2} \norm{{\sigma}^w f}^2 \leq
 ( ({\sigma}^2)^w f,f) \leq 2\norm{{\sigma}^w f}^2.
\end{split}
\end{equation}
Taking ${\sigma}^{1/2} \in S_K(g^{1/2})$ (by Lemma \ref{stab})  instead of ${\sigma}$, we obtain
 \begin{equation*}
\begin{split}
 \norm{({\sigma}^{1/2})^w f}^2 \simeq
 ( {\sigma}^w f,f) & = \norm{({\sigma}^w)^{1/2}}^2
\end{split}
\end{equation*}
and the proof of \eqref{maingamma} is complete.

Concerning \eqref{mainequiv}, we just have to prove one inequality since the result is symmetric in $\tau$ and ${\sigma}$. For $K$ sufficiently large, we have
$$
\norm{\tau^w f} = \norm{\tau^w ({\sigma}^w)^{-1} {\sigma}^w f} = \norm{\tau^w ({\sigma}^{-1})^w H_R {\sigma}^w f} = \norm{( \tau \sharp ({\sigma}^{-1}))^w H_R {\sigma}^w f} \leq C \norm{ {\sigma}^w f}
$$
 since $\tau \sharp ({\sigma}^{-1}) \in S(1)$, so that $( \tau \sharp ({\sigma}^{-1}))^w$ is bounded (with bound independent of $K$). By symmetry, this proves \eqref{mainequiv}.

We then prove \eqref{mainnorm}. We first recall that ${\sigma}^\wick = ({\sigma}\star N)^w$ and that ${\sigma}\star N$ is elliptic positive in $S_K(g)$ by Lemma \ref{allsk}.
 From \eqref{mainequiv}, this directly yields
 $$
 \norm{ {\sigma}^w f} \simeq \norm{ ({\sigma} \star N)^w f} = \norm{ {\sigma}^\wick f}.
 $$
 By direct computation $({\sigma}^2\star N)^{1/2}$ is also in $S_K(g)$ by
 point b) of Lemma \ref{stab} with  $\nu=2$ and $\nu=1/2$, respectively, and Lemma \ref{allsk}. Using again
 \eqref{mainequiv} and  \eqref{maingamma}, yields
 that
 $$
  \norm{ {\sigma}^w f} \simeq  \norm{ (({\sigma}^2\star N)^{1/2})^w f} \simeq \norm{ (({\sigma}^2\star N)^w)^{1/2} f} = (({\sigma}^2\star N)^w f,f) = (({\sigma}^2)^\wick f,f).
  $$

The proof of the last point \eqref{mainscalar} follows exactly the same lines and we skip it.
\end{proof}

\bigskip
\bibliographystyle{acm}
\bibliography{biblioBoltzmann}

\end{document}